\documentclass[a4paper,12pt]{amsart}

\usepackage{bbm}
\usepackage{units}
\usepackage{mathtools}
\usepackage{cancel}
\usepackage{amsmath}
\usepackage{amsfonts}
\usepackage{latexsym}
\usepackage{amssymb}
\usepackage{mathrsfs}
\usepackage{amscd}
\usepackage{hyperref}
\usepackage{soul}
\usepackage{psfrag}
\usepackage{graphicx}
\usepackage{ulem}
\usepackage{fullpage}
\usepackage[all]{xy}
\usepackage{rotating}
\usepackage{url}
\usepackage{color}
\usepackage{enumerate}
\usepackage{cleveref}
\usepackage{dsfont}
\usepackage{tikz}
\usetikzlibrary{arrows,backgrounds,arrows.meta,decorations.markings}
\usepackage{tikz-cd}

\numberwithin{equation}{section}
\numberwithin{figure}{section}

\theoremstyle{plain}
\newtheorem{thm}{Theorem}[section]

\theoremstyle{plain}
\newtheorem{lem}[thm]{Lemma}

\theoremstyle{remark}
\newtheorem{rem}[thm]{Remark}

\theoremstyle{plain}
\newtheorem{cor}[thm]{Corollary}

\theoremstyle{definition}
\newtheorem{defn}[thm]{Definition}

\theoremstyle{definition}

\theoremstyle{definition}

\theoremstyle{plain}
\newtheorem{prop}[thm]{Proposition}

\theoremstyle{plain}
\newtheorem{fact}[thm]{Fact}

\theoremstyle{definition}

\theoremstyle{plain}

\newcommand{\comments}[1]{}
\newcommand{\ra}{\rightarrow}
\newcommand{\rab}{\rangle}
\newcommand{\lra}{\longrightarrow}

\newcommand{\lab}{\langle}

\newcommand{\mcal}{\mathcal}

\newcommand{\N}{\mathbb N}

\newcommand{\C}{\mathbb{C}}

\newcommand{\mscr}{\mathscr}
\newcommand{\vlon}{\varepsilon}

\newcommand{\btimes}{\boxtimes}

\title{Q-system completeness of unitary connections}
\author{Mainak Ghosh}
\newcommand{\Contact}{{
		\bigskip
		\footnotesize
		
		\medskip

		Mainak Ghosh, \textsc{Stat-Math Unit, Indian Statistical Institute}\par\nopagebreak
		\textit{E-mail address}: \texttt{main\_ghosh@rediffmail.com}
		
}}
\begin{document}
\maketitle
	\global\long\def\vlon{\varepsilon}
	\global\long\def\bt{\bowtie}
	\global\long\def\ul#1{\underline{#1}}
	\global\long\def\ol#1{\overline{#1}}
	\global\long\def\norm#1{\left\|{#1}\right\|}
	\global\long\def\os#1#2{\overset{#1}{#2}}
	\global\long\def\us#1#2{\underset{#1}{#2}}
	\global\long\def\ous#1#2#3{\overset{#1}{\underset{#3}{#2}}}
	\global\long\def\t#1{\text{#1}}
	\global\long\def\lrsuf#1#2#3{\vphantom{#2}_{#1}^{\vphantom{#3}}#2^{#3}}
	\global\long\def\tr{\triangleright}
	\global\long\def\tl{\triangleleft}
	\global\long\def\cc90#1{\begin{sideways}#1\end{sideways}}
	\global\long\def\turnne#1{\begin{turn}{45}{#1}\end{turn}}
	\global\long\def\turnnw#1{\begin{turn}{135}{#1}\end{turn}}
	\global\long\def\turnse#1{\begin{turn}{-45}{#1}\end{turn}}
	\global\long\def\turnsw#1{\begin{turn}{-135}{#1}\end{turn}}
	\global\long\def\fusion#1#2#3{#1 \os{\textstyle{#2}}{\otimes} #3}
	
	\global\long\def\abs#1{\left|{#1}\right	|}
	\global\long\def\red#1{\textcolor{red}{#1}}

\begin{abstract}
	A Q-system is a unitary version of a separable Frobenius algebra object in a C*-tensor category. In a recent joint work with P. Das, S. Ghosh and C. Jones, the author has categorified Bratteli diagrams and unitary connections by building a $2$-category $\textbf{UC}$. We prove that every $Q$-system in $\textbf{UC}$ splits.  
\end{abstract}

\section{Introduction}
V. Jones' groundbreaking results on index for subfactors \cite{J83} has led to remarkable progress in the development of the theory of subfactors. The standard invariant of a finite index subfactor of a $II_1$ factor was first defined as a $\lambda$-lattice \cite{P95}. In \cite{M03}, a $Q$-system which is a unitary version of a Frobenius algebra object in a C*-tensor category or C*-2-category, is exhibited as an alternative axiomatization of the standard invariant of a finite index subfactor (\cite{O88}, \cite{P95}, \cite{J99}). This further fostered classification of small index subfactors (\cite{JMS}, \cite{AMP}). Q-systems were first introduced in \cite{Lon} to characterize canonical endomorphism associated to a finite index subfactor of an infinite factor. 

Given any rigid, semisimple, C*-tensor category $\mcal C$ with simple tensor unit $\mathbbm{1}$, an indecomposable Q-system $Q \in \mcal C \ \left(\t{End}_{Q-Q} \left( Q \right) \simeq \C \right)$ and a fully-faithful unitary tensor functor $H : \mcal C \to \t{Bim}(N)$ for some $II_1$ factor $N$, we can apply realization procedure (\cite{JP19}, \cite{JP20}) to construct a $II_1$ factor $M$ containing $N$ as a generalized crossed product $N \rtimes_{H} Q$. Also, every irreducible finite index extension of $N$ is of this form.

In the context of C*-2-categories, a Q-system is a 1-cell $_bQ_b \in \mcal C_1(b,b)$ along with two 2-cells $m : Q \boxtimes Q \to Q$ (multiplication) and $i : 1_b \to Q$ (unit), which are graphically denoted by the following:
\[m = \raisebox{-6mm}{
\begin{tikzpicture}
\draw[red,in=90,out=90,looseness=2] (-0.5,0.5) to (-1.5,0.5);
\node at (-1,1.1) {${\color{red}\bullet}$};
\draw[red] (-1,1.1) to (-1,1.6);
\node[left,scale=0.7] at (-1,1.4) {$Q$};
\node[left,scale=0.7] at (-1.6,0.5) {$Q$};
\node[right,scale=0.7] at (-.5,.5) {$Q$};
\end{tikzpicture}} \ \ \ \ \ \ i = \raisebox{-6mm}{
\begin{tikzpicture}
\draw [red] (-0.8,-.6) to (-.8,.6);
\node at (-.8,-.6) {${\color{red}\bullet}$};
\node[left,scale=0.7] at (-.8,0) {$Q$};
\end{tikzpicture}} \ \ \ \ \ \ m^* = \raisebox{-6mm}{
\begin{tikzpicture}
\draw[red,in=-90,out=-90,looseness=2] (-0.5,0.5) to (-1.5,0.5);
\node at (-1,-.1) {${\color{red}\bullet}$};
\draw[red] (-1,-.1) to (-1,-.6);
\node[left,scale=0.7] at (-1,-.4) {$Q$};
\node[left,scale=0.7] at (-1.6,0.5) {$Q$};
\node[right,scale=0.7] at (-.5,.5) {$Q$};
\end{tikzpicture}} \ \ \ \ \ \ i^* = \raisebox{-6mm}{
\begin{tikzpicture}
\draw [red] (-0.8,-.6) to (-.8,.6);
\node at (-.8,.6) {${\color{red}\bullet}$};
\node[left,scale=0.7] at (-.8,0) {$Q$};
\end{tikzpicture}} \]
These $2$-cells satisfy the following:
\[\raisebox{-6mm}{
	\begin{tikzpicture}
	\draw[red,in=90,out=90,looseness=2] (0,0) to (1,0);
	\draw[red,in=90,out=90,looseness=2] (0.5,.6) to (-.5,.6);
	\draw[red] (-.5,.6) to (-.5,0);
	\node at (.5,.6) {$\red{\bullet}$};
	\node at (0,1.2) {$\red{\bullet}$};
	\draw[red] (0,1.2) to (0,1.6);
	\end{tikzpicture}}
= \raisebox{-6mm}{\begin{tikzpicture}
	\draw[red,in=90,out=90,looseness=2] (0,0) to (1,0);
	\draw[red,in=90,out=90,looseness=2] (.5,.6) to (1.5,.6);
	\node at (.5,.6) {$\red{\bullet}$};
	\draw[red] (1.5,.6) to (1.5,0);
	\node at (1,1.2) {$\red{\bullet}$};
	\draw[red] (1,1.2) to (1,1.6);
	\end{tikzpicture}} \ \t{(associativity)} \ \ \ \ \ \ \raisebox{-4mm}{
	\begin{tikzpicture}
	\draw[red,in=90,out=90,looseness=2] (0,0) to (1,0);
	\node at (.5,.6) {$\red{\bullet}$};
	\node at (0,0) {$\red{\bullet}$};
	\draw[red] (.5,.6) to (.5,1.2);
	\end{tikzpicture}}=
\raisebox{-4mm}{
\begin{tikzpicture}
\draw[red,in=90,out=90,looseness=2] (0,0) to (1,0);
\node at (.5,.6) {$\red{\bullet}$};
\node at (1,0) {$\red{\bullet}$};
\draw[red] (.5,.6) to (.5,1.2);
\end{tikzpicture}}
=
\raisebox{-2mm}{
\begin{tikzpicture}
\draw[red] (0,0) to (0,1.2);
\end{tikzpicture}} \ \t{(unitality)} \]

\[\raisebox{-8mm}{
	\begin{tikzpicture}
	\draw[red,in=90,out=90,looseness=2] (0,0) to (1,0);
	\draw[red,in=-90,out=-90,looseness=2] (1,0) to (2,0);
	\node at (.5,.6) {$\red{\bullet}$};
	\node at (1.5,-.6) {$\red{\bullet}$};
	\draw[red] (.5,.6) to (.5,1.2);
	\draw[red] (1.5,-.6) to (1.5,-1.2);
	\draw[red] (0,0) to (0,-.6);
	\draw[red] (2,0) to (2,.6);
	\end{tikzpicture}} =
\raisebox{-6mm}{
	\begin{tikzpicture}
	\draw[red,in=90,out=90,looseness=2] (0,0) to (1,0);
	\node at (.5,.6) {$\red{\bullet}$};
	\draw[red] (.5,.6) to (.5,1.2);
	\draw[red,in=-90,out=-90,looseness=2] (0,1.8) to (1,1.8);
	\node at (.5,1.2) {$\red{\bullet}$};
	\end{tikzpicture}} =
\raisebox{-8mm}{
	\begin{tikzpicture}
	\draw[red,in=-90,out=-90,looseness=2] (0,0) to (1,0);
	\draw[red,in=90,out=90,looseness=2] (1,0) to (2,0);
	\node at (.5,-.6) {$\red{\bullet}$};
	\node at (1.5,.6) {$\red{\bullet}$};
	\draw[red] (.5,-.6) to (.5,-1.2);
	\draw[red] (1.5,.6) to (1.5,1.2);
	\draw[red] (0,0) to (0,.6);
	\draw[red] (2,0) to (2,-.6);
	\end{tikzpicture}} \ \t{(Frobenius condition)} \ \ \ \ \ \ \raisebox{-6mm}{
	\begin{tikzpicture}
	\draw[red,in=90,out=90,looseness=2] (0,0) to (1,0);
	\draw[red,in=-90,out=-90,looseness=2] (0,0) to (1,0);
	\node at (.5,-.6) {$\red{\bullet}$};
	\node at (.5,.6) {$\red{\bullet}$};
	\draw[red] (.5,-.6) to (.5,-1.2);
	\draw[red] (.5,.6) to (.5,1.2);
	\end{tikzpicture}} = 
\raisebox{-6mm}{
\begin{tikzpicture}
\draw[red] (0,0) to (0,2.4);
\end{tikzpicture}} \ \t{(Separability)} \]

\vspace*{4mm}

Recently \cite{CPJP} introduced the notion of \textit{Q-system completion} for C*/W*-2-categories which is another version of a higher idempotent completion for C*/W*-2-categories in comparison with 2-categories of separable monads \cite{DR18} and condensation monads in \cite{GJF19}. Given a C*/W*-2-category $\mcal C$, it's \textit{Q-system completion} is the 2-category $\textbf{QSys}(\mcal C)$ of Q-systems, bimodules and intertwiners in $\mcal C$ . There is a canonical inclusion *-2-functor $\iota_{\mcal C} : \mcal C \hookrightarrow \textbf{QSys}(\mcal C)$ which is always an equivalence on all hom categories. $\mcal C$ is said to be \textit{Q-system complete} if $\iota_{\mcal C}$ is a *-equivalence of *-2-categories.
We study \textit{Q-system completeness} in the context of pre-C*-2-categories. We call a pre-C*-2-category $\mcal C$ to be \textit{Q-system complete} if every $Q$-system in $\mcal C$ `splits'.

In our recent joint paper \cite{DGGJ}, we gave a higher categorical interpretation of Bratteli diagrams and unitary connections in terms of a larger W*-2-category $\textbf{UC}^\t{tr}$. The $0$-cells of $\textbf{UC}^\t{tr}$ are Bratteli diagrams with tracial weighting data. These generalize the Bratteli diagrams appearing from taking the tower of relative commutants of a finite-index subfactor. $1$-cells of $\textbf{UC}^\t{tr}$ are \textit{unitary connections} between Bratteli diagrams which are compatible with the tracial data. Finally the $2$-cells are defined as certain fixed points of a ucp (unital completely positive) map. To define $\textbf{UC}^\t{tr}$, we had to first consider a purely algebraic category $\textbf{UC}$. The $0$-cells of $\textbf{UC}$ are Bratteli diagrams (without the tracial data). $1$-cells of $\textbf{UC}$ are unitary connections and $2$-cells are natural intertwiners between connections which we call flat sequences. $\textbf{UC}$ has a close resemblance to the $ 2 $-category studied in \cite{CPJ} in the context of fusion category actions on AF-C*-algebras, with minor differences at the level of $0$-cells and $2$-cells only.

We investigate Q-system completeness of $\textbf{UC}$. The following is the main theorem of the paper.
\begin{thm}
	$\normalfont \textbf{UC}$ is $Q$-system complete.
\end{thm}

Given a $Q$-system in $\textbf{UC}$, to exhibit it's `splitting' one needs to construct a suitable $0$-cell and a suitable dualizable $1$-cell from the initial $0$-cell to the newly constructed one which enables the splitting. The idea to construct our suitable $0$-cell in $\textbf{UC}$ comes from \cite{CPJP} and we use subfactor theoretic ideas (\cite{B97}, \cite{EvKaw}, \cite{P89},\cite{P94}) to build our appropriate $1$-cell in $\textbf{UC}$.

There at least two natural questions appearing from our investigations. Bi-faithfulness of functors (that is, both the functor and its adjoint are faithful) plays a major role in achieving our results. So the first question is, if we drop the bi-faithfulness condition of $0$-cells and $1$-cells in $\textbf{UC}$ (see \Cref{UCdefn}), then will the modifed 2-category be still Q-system complete. Second, is $\textbf{UC}^\t{tr}$ Q-system complete ? We will try to answer these questions in our future work.

The outline of the paper is as follows. In Section 2, we will quickly go through some basic results and definitions and set up some pictorial notations. In Section 3, we explore $Q$-systems in $\textbf{UC}$ and prove some results that will be useful to construct our appropriate $0$-cell in $\textbf{UC}$. In Section 4, we build the $0$-cell and the dualizable $1$-cell and then proceed to prove our main theorem. 

\subsection*{Acknowledgements}
The author would like to thank Shamindra Kumar Ghosh, Corey Jones and David Penneys for several fruitful discussions.

\section{Preliminaries}

In this section we will furnish the necessary background on \textit{Q-system completion} and the 2-category of \textit{Unitary connections} \textbf{UC}.

\subsection{Notations related to 2-categories}\label{graphcalc}
We refer the reader to \cite{JY21} for basics of 2-categories.\\
Suppose $\mcal C$ is a 2-category and  $a,b \in \mcal C_0$ be two $0$-cells. A $1$-cell from $a  \xrightarrow{X} b$ is denoted by $_bX_a$. Pictorially, a $1$-cell will be denoted by a red strand and a $2$-cell will be denoted by a box with strings with passing through it. Suppose we have two $1$-cells $X, Y \in \mcal C_1(a,b)$ and $f \in \mcal C_2(X,Y)$ be a $2$-cell. Then we will denote $f$ as 
\raisebox{-11mm}{
\begin{tikzpicture}
\draw[red] (0,.8) to (0,-.8);
\node[draw,thick,rounded corners,fill=white, minimum width=20] at (0,0) {$f$};
\node[right] at (0,-.7) {$X$};
\node[right] at (0,.7) {$Y$};
\end{tikzpicture}}
We write tensor product $\boxtimes$ of $1$-cells from right to left $ _cY \us{b}\boxtimes X_a$. 
The notion of C*-2-categories is believed to first appear in \cite{LR}. For basics of C*/W*-2-categories we refer the reader to (\cite{CPJP}, \cite{GLR85}).

\subsection{Q-system completion}
\begin{defn}
	A pre-C*-2-category is a $2$-category such that the hom-1-categories satisfies all the conditions of a C*-category except that the $2$-cell spaces need not be complete with repsect to the given norm.
\end{defn}

Let $\mcal C$ be a pre-C*-2-category.
\begin{defn} \label{Qsysdefn}
	A Q-system in $\mcal C$ is a $1$-cell $_bQ_b \in \mcal C_1(b,b)$ along with multiplication map $m \in \mcal C_2(Q \boxtimes_b Q, Q)$ and unit map $i \in \mcal C_2(1_b,Q)$, as mentioned in section 1, satisfying the following properties:
	\end{defn}
	\begin{itemize}
		\item[(Q1)]
		 \raisebox{-6mm}
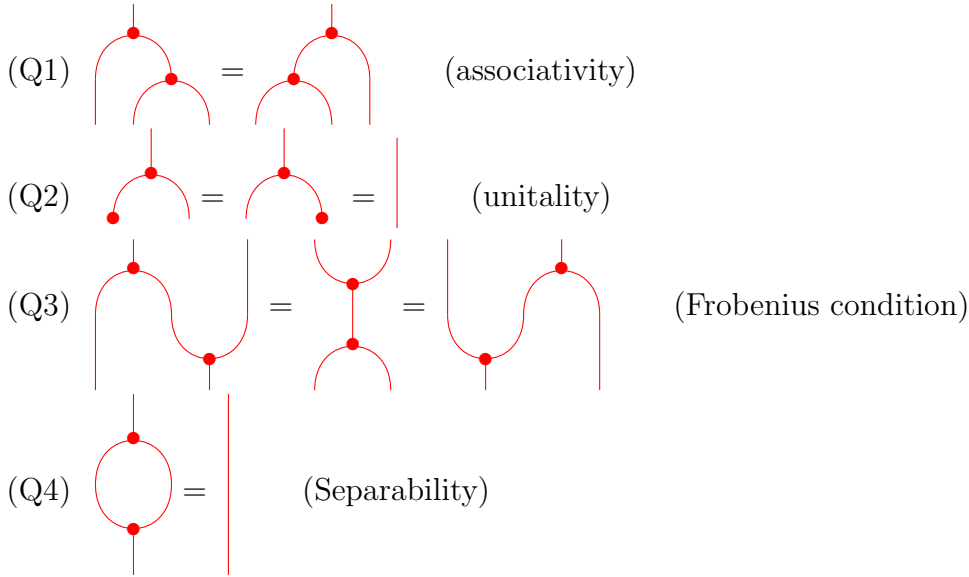
{
			\begin{tikzpicture}
			\draw[red,in=90,out=90,looseness=2] (0,0) to (1,0);
			\draw[red,in=90,out=90,looseness=2] (0.5,.6) to (-.5,.6);
			\draw[red] (-.5,.6) to (-.5,0);
			\node at (.5,.6) {$\red{\bullet}$};
			\node at (0,1.2) {$\red{\bullet}$};
			\draw[red] (0,1.2) to (0,1.6);
			\end{tikzpicture}}
		= \raisebox{-6mm}{\begin{tikzpicture}
			\draw[red,in=90,out=90,looseness=2] (0,0) to (1,0);
			\draw[red,in=90,out=90,looseness=2] (.5,.6) to (1.5,.6);
			\node at (.5,.6) {$\red{\bullet}$};
			\draw[red] (1.5,.6) to (1.5,0);
			\node at (1,1.2) {$\red{\bullet}$};
			\draw[red] (1,1.2) to (1,1.6);
			\end{tikzpicture}} \ \ \ \ \ \ \t{(associativity)} 
		
		\item[(Q2)]
		\raisebox{-4mm}{
		\begin{tikzpicture}
		\draw[red,in=90,out=90,looseness=2] (0,0) to (1,0);
		\node at (.5,.6) {$\red{\bullet}$};
		\node at (0,0) {$\red{\bullet}$};
		\draw[red] (.5,.6) to (.5,1.2);
		\end{tikzpicture}} = 
	\raisebox{-4mm}{
	\begin{tikzpicture}
	\draw[red,in=90,out=90,looseness=2] (0,0) to (1,0);
	\node at (.5,.6) {$\red{\bullet}$};
	\node at (1,0) {$\red{\bullet}$};
	\draw[red] (.5,.6) to (.5,1.2);
	\end{tikzpicture}}
  = 
\raisebox{-3mm}{
\begin{tikzpicture}
\draw[red] (0,0) to (0,1.2);
\end{tikzpicture}} \ \ \ \ \ \ \t{(unitality)} 

\item[(Q3)]
\raisebox{-10mm}{
\begin{tikzpicture}
\draw[red,in=90,out=90,looseness=2] (0,0) to (1,0);
\draw[red,in=-90,out=-90,looseness=2] (1,0) to (2,0);
\node at (.5,.6) {$\red{\bullet}$};
\node at (1.5,-.6) {$\red{\bullet}$};
\draw[red] (.5,.6) to (.5,1);
\draw[red] (1.5,-.6) to (1.5,-1);
\draw[red] (0,0) to (0,-1);
\draw[red] (2,0) to (2,1);
\end{tikzpicture}} \ =
\raisebox{-10mm}{
\begin{tikzpicture}
\draw[red,in=90,out=90,looseness=2] (0,0) to (1,0);
\node at (.5,.6) {$\red{\bullet}$};
\draw[red] (.5,.6) to (.5,1.4);
\draw[red,in=-90,out=-90,looseness=2] (0,2) to (1,2);
\node at (.5,1.4) {$\red{\bullet}$};
\end{tikzpicture}} =
\raisebox{-10mm}{
\begin{tikzpicture}
\draw[red,in=-90,out=-90,looseness=2] (0,0) to (1,0);
\draw[red,in=90,out=90,looseness=2] (1,0) to (2,0);
\node at (.5,-.6) {$\red{\bullet}$};
\node at (1.5,.6) {$\red{\bullet}$};
\draw[red] (.5,-.6) to (.5,-1);
\draw[red] (1.5,.6) to (1.5,1);
\draw[red] (0,0) to (0,1);
\draw[red] (2,0) to (2,-1);
\end{tikzpicture}} \ \ \ \ \ \ \t{(Frobenius condition)} 

\item[(Q4)]
 \raisebox{-10mm}{
\begin{tikzpicture}
\draw[red,in=90,out=90,looseness=2] (0,0) to (1,0);
\draw[red,in=-90,out=-90,looseness=2] (0,0) to (1,0);
\node at (.5,-.6) {$\red{\bullet}$};
\node at (.5,.6) {$\red{\bullet}$};
\draw[red] (.5,-.6) to (.5,-1.2);
\draw[red] (.5,.6) to (.5,1.2);
\end{tikzpicture}} = 
\raisebox{-10mm}{
\begin{tikzpicture}
\draw[red] (0,0) to (0,2.4);
\end{tikzpicture}} \ \ \ \ \ \ \t{(Separability)}

\end{itemize}

\begin{defn}\cite{CPJP}
	Given a Q-system $\left(Q,m,i\right)$, we define
	\[ d_Q \coloneqq 
	\raisebox{-5mm}{
	\begin{tikzpicture}
	\draw[red] (0,0) to (0,.8);
	\node at (0,0) {$\red{\bullet}$};
	\node at (0,.8) {$\red{\bullet}$};
	\end{tikzpicture}} \in \t{End}_{\mcal C}\left(1_b \right)^+ \].
\begin{itemize}
	\item If $d_Q$ is invertible, we call $Q$ \textit{non-degenerate} or an \textit{extension} of $1_b$. 
	\item If $d_Q$ is an idempotent, we call $Q$ a \textit{summand} of $1_b$.
\end{itemize}
 \end{defn}
We recall some facts about $Q$-systems in C*-tensor categories already mentioned in (\cite{CPJP}, \cite{Z07}).
\begin{fact}
	\begin{itemize}
		\item[(F1)]\label{F1} Q is a self-dual 1-cell with $ev_Q \coloneqq 
		\raisebox{-4mm}{
		\begin{tikzpicture}[scale=.8]
		\draw[red,in=-90,out=-90,looseness=2] (-0.5,0.5) to (-1.5,0.5);
		\node at (-1,-.1) {${\color{red}\bullet}$};
		\draw[red] (-1,-.1) to (-1,-.6);
		\node[left,scale=0.7] at (-1,-.4) {$Q$};
		\node[left,scale=0.7] at (-1.6,0.4) {$Q$};
		\node[right,scale=0.7] at (-.5,.4) {$Q$};
		\node at (-1,-.6) {${\color{red}\bullet}$};
		\end{tikzpicture}}
		$ and $coev_Q \coloneqq \raisebox{-4mm}{
			\begin{tikzpicture}[scale=.8]
			\draw[red,in=90,out=90,looseness=2] (-0.5,0.5) to (-1.5,0.5);
			\node at (-1,1.1) {${\color{red}\bullet}$};
			\draw[red] (-1,1.1) to (-1,1.6);
			\node[left,scale=0.7] at (-1,1.4) {$Q$};
			\node[left,scale=0.7] at (-1.6,0.6) {$Q$};
			\node[right,scale=0.7] at (-.5,.6) {$Q$};
			\node at (-1,1.6) {${\color{red}\bullet}$};
			\end{tikzpicture}}$.
		\item[(F2)] Using (F1) and [\cite{Z07} lemma 1.16] we have the following inequalitites:
		\[ 
		\raisebox{-8mm}{
		\begin{tikzpicture}[scale=.8]
		\draw[red,in=-90,out=-90,looseness=2] (-0.5,0.5) to (-1.5,0.5);
		\node at (-1,-.1) {${\color{red}\bullet}$};
		\draw[red] (-1,-.1) to (-1,-.6);
		\node[left,scale=0.7] at (-1,-.4) {$Q$};
		\node[left,scale=0.7] at (-1.6,0.4) {$Q$};
		\node[right,scale=0.7] at (-.5,.4) {$Q$};
		\node at (-1,-.6) {${\color{red}\bullet}$};
		\draw[red,in=90,out=90,looseness=2] (-0.5,-1.9) to (-1.5,-1.9);
		\node at (-1,-1.3) {${\color{red}\bullet}$};
		\draw[red] (-1,-1.3) to (-1,-.9);
		\node[left,scale=0.7] at (-1,-1.1) {$Q$};
		\node[left,scale=0.7] at (-1.5,-1.8) {$Q$};
		\node[right,scale=0.7] at (-.5,-1.8) {$Q$};
		\node at (-1,-.9) {${\color{red}\bullet}$};
		\end{tikzpicture}} \ \leq \  
		\raisebox{-8mm}{
		\begin{tikzpicture}
		\draw[red] (.2,.4) to (.2,-.4);
		\draw[red] (.4,1) to (.4,-1);
		\draw[red] (.9,1) to (.9,-1);
		\node at (.2,.4) {${\color{red}\bullet}$};
		\node at (.2,-.4) {${\color{red}\bullet}$};
		\end{tikzpicture}} \ \leq \ 
	\raisebox{-8mm}{
	\begin{tikzpicture}
	\draw[red] (.4,1) to (.4,-1);
	\draw[red] (.9,1) to (.9,-1);
	\node at (-.1,-.1) {$\norm{d_Q}$};
	\end{tikzpicture}} \]

\item[(F3)] By [\cite{Z07} corollary 1.19] either $d_Q$ is invertible, or zero is an isolated point in Spec$\left(d_Q \right)$.
Define , $f : \t{Spec}(d_Q) \to \C$ by   
\[f(x)= \begin{cases}
0 & x=0 \\
x^{-1} & x \neq 0
\end{cases}\]
By abuse of notation, set $d_Q^{-1} \coloneqq f\left(d_Q\right)$. By continuous functional calculus, set $s_Q \coloneqq d_Q d_Q^{-1}$. Then we have the following :
\begin{itemize}
	\item[(a)] \[ \begin{tikzpicture}
	\draw[red] (0,0) to (0,1);
	\draw[red] (.2,.2) to (.2,.8);
	\node[draw,thick,rounded corners,scale=.8] at (.8,.5) {$d_Q^{-1}$};
	\node at (.2,.8) {$\red{\bullet}$};
	\node at (.2,.2) {$\red{\bullet}$};
	\end{tikzpicture} \  \
	\begin{tikzpicture}
	\draw[red] (0,0) to (0,1);
	\node[draw,thick,rounded corners,scale=.8] at (.4,.5) {$s_Q$};
	\node at (.9,.5) {$=$};
	\node at (-.3,.5) {$=$};
	\draw[red] (1.2,0) to (1.2,1);
	\end{tikzpicture} \]
\item[(b)] \[ \begin{tikzpicture}
\draw[red] (0,0) to (0,1);
\draw[red] (.2,.2) to (.2,.8);
\node at (.2,.8) {$\red{\bullet}$};
\node at (.2,.2) {$\red{\bullet}$};
\node at (.6,.5) {$\leq$};
\draw[red] (1.7,0) to (1.7,1);
\node at (1.2,.5) {$\norm{d_Q}$};
\end{tikzpicture} \]	
	 
\end{itemize}

	\end{itemize}
\end{fact}

\begin{defn}
	Suppose $\mcal C$ is a pre-C*-2-category and $_bX_a \in \mcal C_1(a,b)$. A \textit{unitarily separable left dual} for $_b X_a$ is a dual $\left( _a \ol X_b, ev_X, coev_X \right)$ such that $ev_X \circ ev_X^* = \t{id}_{1_a}$ [cf. Example 3.9 \cite{CPJP}].
\end{defn}
Given a unitarily separable left dual for $_bX_a \in \mcal C_1(a,b)$, $_b X \us{a} \boxtimes \ol X_b \in \mcal C_1(b,b)$ is a Q-system with multiplication map $m \coloneqq \t{id}_X \boxtimes ev_X \boxtimes \t{id}_{\ol X}$ and unit map $i \coloneqq coev_X$.

Given a Q-system $Q \in \mcal C_1(b,b)$, if it is of the above form then we say that the Q-system $Q$ `splits'.

\begin{defn}\label{QsyspreC*def}
	A pre-C*-2-category $\mcal C$ is said to be \textit{Q-system complete} if every Q-system in $\mcal C$ `splits', that is, given a  Q-system $Q \in \mcal C_1(b,b)$ there is an object $c \in \mcal C_0$ and a dualizable 1-cell $X \in \mcal C_1(c,b)$ which admits a unitary separable dual $\left(\ol X, ev_X, coev_X \right)$ such that $\left(Q,m,i\right)$ is isomorphic to $_b X \us{c} \boxtimes \ol X_b$ as Q-systems.
	
\end{defn}

\begin{rem}
	In \cite{CPJP}, Q-system completion has been treated in the context of C*/W*-2-categories. It has been proved that \Cref{QsyspreC*def} is equivalent to their definition of \textit{Q-system completeness} (see Thm 3.36 \cite{CPJP}) of C*/W*-2-categories. 
\end{rem}

\subsection{Unitary connections}

\textbf{Pictorial notations.} We will apply the graphical calculus as mentioned in \Cref{graphcalc} to the 2-category of \textit{Categories}. 
 \begin{itemize}\label{Graphcalc}
 	\item[(i)] Let $\mcal C$ be a category and let $ f \in \mcal C(C , D)$. It will be denoted by
 \raisebox{-8mm}{
 	\begin{tikzpicture}
 	\node[draw,thick,rounded corners, fill=white] at (0,0) {$f$};
 	\begin{scope}[on background layer]
 	\draw[dashed,thick] (0,-.9) to (0,.9);
 	\end{scope}	
 	\node[right] at (-0.1,0.6) {$D$};
 	\node[right] at (-0.1,-0.6) {$C$};
 	\end{tikzpicture}
 } and composition of two morphisms will be represented by two vertically stacked labelled boxes. 
 
 \item[(ii)] Let $\mcal C$ and $\mcal D$ be two categories and $F,G: \mcal C \ra \mcal D$ be two functors.
 Then a natural transformation $\eta:F \ra G$ will be denoted by 
 \raisebox{-8mm}{
 	\begin{tikzpicture}
 	\node[draw,thick,rounded corners,fill = white] at (0,0) {$\eta$};
 	\begin{scope}[on background layer]
 	\draw (0,-.8) to (0,.8);
 	\end{scope}
 	\node[right] at (-0.1,0.6) {$G$};
 	\node[right] at (-0.1,-0.6) {$F$};
 	\end{tikzpicture}
 }.
 For an object $x$ in $\mcal C$, the morphism $\eta_{x} :Fx \ra Gx$ will be denoted by
 \raisebox{-8mm}{
 	\begin{tikzpicture}
 	\node[draw,thick,rounded corners,fill=white] at (0,0) {$\eta_{x}$};
 	\begin{scope}[on background layer]
 	\draw (-.2,-.8) to (-.2,.8);
 	\draw[thick,dashed] (.2,-.8) to (.2,.8);
 	\end{scope}
 	\node[left] at (-0.07,0.6) {$G$};
 	\node[left] at (-0.07,-0.6) {$F$};
 	\node[right] at (0.07,0.6) {$x$};
 	\node[right] at (0.07,-0.6) {$x$};
 	\node[left=] at (1,0) {$ = $};
 	\node[draw,thick,rounded corners,fill=white] at (1.3,0) {$\eta$};
 	\begin{scope}[on background layer]
 	\draw (1.3,-.8) to (1.3,.8);
 	\end{scope}
 	\node[right] at (.8,0.6) {$G$};
 	\node[right] at (.8,-0.6) {$F$};
 	\draw[thick,dashed] (1.7,-0.8) to (1.7,0.8);
 	\node[right] at (1.6,0) {$x$};
 	\end{tikzpicture}
 }.
 
 \item[(iii)] For a $ * $-linear functor $ F: \mcal C \ra \mcal D$ between two semisimple C*-categories categories, we will denote a solution to conjugate equation by
 \[
 \rho = \raisebox{-2.5mm}{
 	\begin{tikzpicture}
 	\draw (-.6,.3) to[out=-90,in=180] (0,-.3) to[out=0,in=-90] (.6,.3);
 	\node[right] at (0.4,0.1) {$F'$};
 	\node[left] at (-0.5,0.1) {$F$};
 	\end{tikzpicture}
 }\!\!\! : \t{id}_{\mcal D} \lra FF' \ \ \t{ and } \ \ 
 \rho' = \raisebox{-2.5mm}{
 	\begin{tikzpicture}
 	\draw (-.6,.3) to[out=-90,in=180] (0,-.3) to[out=0,in=-90] (.6,.3);
 	\node[right] at (0.4,0.1) {$F$};
 	\node[left] at (-0.5,0.1) {$F'$};
 	\end{tikzpicture}
 }\!\!\! : \t{id}_{\mcal C} \lra F 'F 
 \]
 \[
 \rho^* = \raisebox{-2.5mm}{
 	\begin{tikzpicture}
 	\draw (-.6,-.3) to[out=90,in=180] (0,.3) to[out=0,in=90] (.6,-.3);
 	\node[right] at (0.5,-0.1) {$F'$};
 	\node[left] at (-0.45,-0.1) {$F$};
 	\end{tikzpicture}
 }\!\!\! : FF' \lra \t{id}_{\mcal D} \ \ \t{ and } \ \ 
 \left[\rho' \right]^* = \raisebox{-2.5mm}{
 	\begin{tikzpicture}
 	\draw (-.6,-.3) to[out=90,in=180] (0,.3) to[out=0,in=90] (.6,-.3);
 	\node[right] at (0.5,-0.1) {$F$};
 	\node[left] at (-0.4,-0.1) {$F'$};
 	\end{tikzpicture}
 }\!\!\! : F 'F \lra  \t{id}_{\mcal C}
 \]
 where $ F' : \mcal D \ra \mcal C $ is an adjoint functor of $ F $.
\end{itemize}

\vspace*{4mm}

 We will extend the above dictionary (between things appearing in the category world and pictures) in an obvious way.
 For instance, composition of morphisms and natural transformations will be pictorially represented by stacking the boxes vertically whereas tensor product (resp., composition) of objects (resp., functors) by parallel vertical strings.
 For simplicity, sometimes we will not label all of the strings (with any object or functor) emanating from a box (labelled with a morphism or a natural transformation) when it can be read off from the context.
 To distinguish between a functor arising in $0$-cell and a functor arising in $1$-cell, we will denote the former by a black strand and the latter by a red strand unless otherwise mentioned. 
  
  \vspace*{2mm}
 
Let us recall the definition of the pre-C*-2-category of \textit{unitary connections} \textbf{UC} described in \cite{DGGJ} 
\begin{defn}\label{UCdefn}
	The 2-category \textbf{UC} consists of the following :
	\begin{itemize}
		\item[(1)] 0-cells are $ * $-linear, bi-faithful functors $ \Gamma_k : \mcal M_{k-1} \ra \mcal M_k $ (where $ \mcal M_k $ is a finite, semisimple, C*-category whose isomorphism classes of the simple objects are indexed by the vertex set $ V_{\mcal M_k} $).
		We will denote such a $ 0 $-cell by $ \left\{ \mcal M_{k-1}  \os {\displaystyle \Gamma_k} \lra \mcal M_k \right\}_{k\geq 1}$ or sometimes simply $ \Gamma_\bullet $.
		
		\item[(2)] A $ 1 $-cell from the $ 0 $-cell $ \left\{\mcal M_{k-1}  \os{\displaystyle \Gamma_k}\lra \mcal M_k \right\}_{k\geq 1}$ to the $ 0 $-cell $ \left\{\mcal N_{k-1} \os{\displaystyle \Delta_k}\lra \mcal N_k \right\}_{k\geq 1}$ consists of a sequence of $ * $-linear bi-faithful functors $ \left\{\Lambda_k : \mcal M_k \ra \mcal N_k\right\}_{k\geq 0} $ and natural unitaries $ W_k: \Delta_k \Lambda_{k-1} \ra \Lambda_k \Gamma_k $ for $ k\geq 1 $. Such a $ 1 $-cell will be denoted by $ \left(\Lambda_\bullet, W_\bullet \right)$ or simply by $ \Lambda_\bullet $, and $ W_\bullet $ will be referred as a \textit{unitary connection associated to $ \Lambda_\bullet $}.
		Denote the set of $ 1 $-cells from $ \Gamma_\bullet$ to $\Delta_\bullet $ by $ \textbf{UC}_1 \left(\Gamma_\bullet , \Delta_\bullet\right) $.
		
		Pictorially, the natural unitary $ W_k $ appearing in the $ 1 $-cell will be represented by
		\raisebox{-5mm}{
			\begin{tikzpicture}
			\draw[white,line width=1mm,out=-90,in=90] (-.25,.5) to (.25,-.5);
			\draw[out=-90,in=90,red] (-.25,.5) to (.25,-.5);
			\begin{scope}[on background layer]
			\draw[out=90,in=-90] (-.25,-.5) to (.25,.5);
			\end{scope}
			\node[right] at (0.15,-0.4) {$\Lambda_{k-1}$};
			\node[right] at (0.15,0.3) {$\Gamma_k$};
			\node[left] at (-0.15,-0.4) {$\Delta_k$};
			\node[left] at (-0.15,0.3) {$\Lambda_k$};
			\end{tikzpicture}
		}
		and $ W^*_k $ by
		\raisebox{-5mm}{
			\begin{tikzpicture}
			\draw[white,line width=1mm,out=-90,in=90] (.25,.5) to (-.25,-.5);
			\draw[out=-90,in=90,red] (.25,.5) to (-.25,-.5);
			\begin{scope}[on background layer]
			\draw[out=90,in=-90] (.25,-.5) to (-.25,.5);
			\end{scope}
			\node[right] at (0.15,-0.4) {$\Gamma_k$};
			\node[right] at (0.15,0.3) {$\Lambda_{k-1}$};
			\node[left] at (-0.15,-0.4) {$\Lambda_k$};
			\node[left] at (-0.15,0.3) {$\Delta_k$};
			\end{tikzpicture}
		}
		\vspace*{2mm}
		
		\item[(3)] Let $ \Lambda_\bullet  , \Omega_\bullet \in \textbf{UC}_1  \left(\Gamma_\bullet , \Delta_\bullet\right)$. For describing 2-cells we need the following definition:
		 \begin{defn}\label{XreldefpreC*}
			A pair $ (\eta , \kappa) \in \t{NT}(\Lambda_k , \Omega_k) \times \t{NT}(\Lambda_{k+1} , \Omega_{k+1}) $ is said to satisfy \textit{exchange relation} if the condition
			\raisebox{-1cm}{
				\begin{tikzpicture}
				\draw[in=-90,out=90] (0,0) to (0,1.2);
				\draw[in=-90,out=90] (0,1.2) to (0.6,2);
				\draw[in=-90,out=90] (3.4,1) to (3.4,2);
				\draw[in=-90,out=90] (2.6,0) to (3.4,1);
				\draw[red,in=-90,out=90] (.6,0) to (0.6,1.2);
				\draw[white,line width=1mm,in=-90,out=90] (0.6,1.2) to (0,2);
				\draw[red,in=-90,out=90] (0.6,1.2) to (0,2);
				\draw[red,in=-90,out=90] (2.6,1) to (2.6,2);
				\draw[white,line width=1mm,in=-90,out=90] (3.4,0) to (2.6,1);
				\draw[red,in=-90,out=90] (3.4,0) to (2.6,1);
				\node[draw,thick,rounded corners, fill=white] at (0.6,0.8) {$\eta$};
				\node[draw,thick,rounded corners, fill=white] at (2.6,1.4) {$\kappa$};
				\node[right] at (.5,0.2) {$ \Lambda_k $};
				\node[right] at (.5,1.35) {$\Omega_k $};
				\node[left] at (0.2,1.8) {$\Omega_{k+1} $};
				\node[right] at (3.3,0.2) {$ \Lambda_k $};
				\node[left] at (2.8,.75) {$\Lambda_{k+1} $};
				\node[left] at (2.75,1.85) {$\Omega_{k+1} $};
				\node at (1.5,1) {$ = $};
				\end{tikzpicture}
			}
			holds.
		\end{defn}
		\begin{rem}\label{xrelunique}
			The exchange relation pair is unique separately in each variable, that is, if $ (\eta, \kappa_1) $ and $ (\eta , \kappa_2) $ (resp., $ (\eta_1, \kappa) $ and $ (\eta_2 , \kappa) $) both satisfy exchange relation, then $ \kappa_1 = \kappa_2 $ (resp., $ \eta_1 = \eta_2 $); this is because the connections are unitary and the functors $ \Gamma_k $ and $ \Delta_k $ are bi-faithful.
		\end{rem}
		Let $ \text{Ex}(\Lambda_{\bullet}, \Omega_{\bullet})$ denote the space of sequences  $\{\eta^{(k)} \in \t{NT} \left(\Lambda_k , \Omega_k\right)\}_{k\geq 0}$ such that there exists an $N$ such that $(\eta_{k},\eta_{k+1})$ satifies the exchange relation for all $k\ge N$.
		Consider the subspace
		$$
		\text{Ex}_{0}(\Lambda_{\bullet}, \Omega_{\bullet}):=\left\{\{\eta_{k}\}_{k\geq 0}  \in \text{Ex}(\Lambda_{\bullet}, \Omega_{\bullet}) :\ \eta_{k}=0 \ \text{ for all }\ k\ge N \  \t{  for some }  \ N \in \N  \right\}
		$$
			
			We define the space of $ 2 $-cells $$\textbf{UC}_2 \left(\Lambda_\bullet , \Omega_\bullet \right):=  \frac {\text{Ex}(\Lambda_{\bullet}, \Omega_{\bullet})}  {\text{Ex}_{0}(\Lambda_{\bullet}, \Omega_{\bullet})} $$
			
		\item[(4)] 
		For $\Omega_\bullet\in \textbf{UC}_1 \left(\Delta_\bullet ,\Sigma_\bullet\right) $ and $\Lambda_\bullet  \in  \textbf{UC}_1 \left(\Gamma_\bullet ,\Delta_\bullet \right) $, define
		\begin{equation}\label{tensor1cell}
		\Omega_\bullet \boxtimes \Lambda_\bullet \coloneqq \left( \left\{\Omega_k \, \Lambda_k\right\}_{k\geq 0} \, , \,  \left\{\raisebox{-1.35cm}{
			\begin{tikzpicture}
			\draw[in=-90,out=90] (0,0) to (1,2);
			\draw[white,line width=1mm,in=-90,out=90] (0.5,0) to (0,2);
			\draw[red,in=-90,out=90] (0.5,0) to (0,2);
			\draw[white,line width=1mm,in=-90,out=90] (1,0) to (0.5,2);
			\draw[red,in=-90,out=90] (1,0) to (0.5,2);
			\node[left] at (0.1,0.1) {$ \Sigma_k $};
			\node[left] at (1.2,-0.1) {$\Omega_{k-1} $};
			\node[right] at (.9,0.1) {$ \Lambda_{k-1} $};
			\node[right] at (.9,1.8) {$ \Gamma_k$};
			\node[right] at (.2,2.15) {$\Lambda_k$};
			\node[left] at (.1,1.8) {$ \Omega_k$};
			\end{tikzpicture}
		}\right\}_{k\geq 1} \right) \ .
		\end{equation}
		\end{itemize}

		For notational convenience, instead of denoting a $ 2 $-cell by an equivalence class of sequences, we simply use a sequence in the class and truncate upto a level after which the exchange relation holds for every consecutive pair, namely, $ \left\{ \eta^{(k)}\right\}_{k \geq N} \in \textbf{UC}_2 \left(\Lambda_\bullet , \Omega_\bullet \right) $ where $(\eta^{(k)},\eta^{(k+1)})$ satifies the exchange relation for all $k\ge N$.
		
		\begin{rem}\label{UCisomorphism}
			From the definition of $\textbf{UC}_2 \left( \Lambda_\bullet, \Omega_\bullet \right)$, we observe that two $2$-cells $ \left\{ \eta^{(k)}\right\}_{k \geq N}$, $\left\{ \tau^{(k)}\right\}_{k \geq L}  \in \textbf{UC}_2 \left(\Lambda_\bullet , \Omega_\bullet \right) $ are equal if and only if $\eta^{(k)} = \tau^{(k)}$ eventually. So, two $1$-cells $\Lambda_\bullet$ and $\Omega_\bullet$ are isomorphic in $\textbf{UC}$ if there is a sequence of natural transformations $U_k : \Lambda_k \to \Omega_k $ which satisfies exchange relation from some level $l$ and which implements isomorphism between $\Lambda_k$ and $\Omega_k$ eventually.
		\end{rem}
\end{defn}	
For horizontal and vertical composition of 2-cells we refer the reader to \cite{DGGJ}.

\vspace*{4mm}

Given a $ 0 $-cell $\Gamma_\bullet \in \textbf{UC}_0$, we fix an object $ m_0 \coloneqq \us{v \in V_0} \bigoplus v \in \t{ob} (\mcal M_0)$.
Consider the sequence of finite dimensional C*-algebras $ \left\{ A_k \coloneqq \t{End} (\Gamma_k \cdots \Gamma_1 m_0) \right\}_{k\geq 0}$ (assuming $ A_0 = \t{End}(m_0) $) along with the unital $ * $-algebra inclusions given by 
\begin{equation}\label{Akinclusion}
 A_{k-1} \ni \alpha \hookrightarrow \Gamma_k \, \alpha \in A_{k} \ .
 \end{equation}
Indeed, the Bratteli diagram of $ A_{k-1} $ inside $ A_{k} $ is given by the graph $ \Gamma_k $.
To the $ 0 $-cell $ \Gamma_\bullet $, we associate the $ * $-algebra $ A_\infty \coloneqq \us{k\geq 0}{\cup} A_k $ \vspace*{2mm}

To each $ 1 $-cell $ (\Lambda_\bullet , W_\bullet) \in \textbf{UC}_1 \left(\Gamma_\bullet, \Delta_\bullet \right)$, we will associate an $ A_\infty $-$ B_\infty $ right correspondence where $ n_0 $ and $ B_k $'s are related to $ \left\{\mcal N_{k-1} \os{\displaystyle \Delta_k}\lra \mcal N_k \right\}_{k\geq 1}$ exactly the way $ m_0 $ and $ A_k $'s are related to $ \left\{ \mcal M_{k-1} \os{\displaystyle \Gamma_k}\lra \mcal M_k \right\}_{k\geq 1}$ respectively.
For $ k\geq 0 $, set $ H_k \coloneqq \mcal N_k \left(\Delta_k \cdots \Delta_1 n_0 , \Lambda_k \Gamma_k \cdots \Gamma_1 m_0\right) $.
We have an obvious $ A_k $-$ B_k $-bimodule structure on $ H_k $ in the following way:

\begin{equation}\label{Hkbimod}
A_k \times H_k \times B_k \ni (\alpha , \xi , \beta) \longmapsto \Lambda_k (\alpha) \circ \xi \circ \beta \ \in H_k \ .
\end{equation}
Again, there is a $ B_k $-valued inner product on $ H_k $ given by

\begin{equation}\label{Hkinp}
H_k \times H_k \ni (\xi,\zeta ) \os{\displaystyle \left\langle \cdot , \cdot \right \rangle_{B_k}} \longmapsto \left\langle \xi , \zeta  \right \rangle_{B_k} \coloneqq \zeta^* \circ \xi \ \in B_k \ .
\end{equation}

Next, observe that $ H_k $ sits inside $ H_{k+1} $ via the map
\begin{equation}\label{Hkinclusion}
H_k \ni \xi   \xmapsto {I_{k+1}} \left[\left(W_{k+1}\right)_{\Gamma_k \cdots \Gamma_1 m_0}\right] \circ \left[\Delta_{k+1} \xi \right] \ =
\raisebox{-.9cm}{
	\begin{tikzpicture}
	\draw[thick,dashed] (.5,-.7) to (.5,1);
	\draw[in=-90,out=90] (-.9,-.7) to (-.9,.3) to (-.6,1);
	\draw[in=-90,out=90] (-.3,-.7)to (-.3,1);
	\draw[in=-90,out=90] (.3,-.7)to (.3,1);
	\draw[white,line width=1mm,out=-90,in=90] (-.6,0.3) to (-.9,1);
	\draw[red,in=-90,out=90] (-.6,0.3) to (-.9,1);
	\node[draw,thick,rounded corners, fill=white,minimum width=45] at (0,0) {$\xi$};
	\node[right] at (-.4,.6) {$ \cdots$};
	\node[right] at (-.4,-.6) {$ \cdots$};
	\node[left] at (-.7,.8) {$ \Lambda_{k+1}$};
	\node[left] at (-.8,-.3) {$ \Delta_{k+1}$};
	\node[right] at (.4,-.6) {$ n_0$};
	\node[right] at (.4,.6) {$ m_0$};
	\end{tikzpicture}
} \in H_{k+1} \ .
\end{equation}
\begin{lem}(\cite{DGGJ}) \label{inclusionactioncompatible}
	The inclusions $ H_k \hookrightarrow H_{k+1} $, $ A_k \hookrightarrow A_{k+1} $, $ B_k \hookrightarrow B_{k+1} $ and the corresponding actions are compatible in the obvious sense.
\end{lem}

Set $ H_\infty \coloneqq \us {k \geq 0} \cup H_k $ which clearly becomes an $ A_\infty $-$ B_\infty $ right correspondence.
To the $ 1 $-cell $ (\Lambda_\bullet , W_\bullet)$ we associate the $ A_\infty $-$ B_\infty $ right correspondence $H_\infty$.

We also have a Pimsner-Popa (PP) basis of the right-$ B_\infty $-module $ {H_\infty }$ with respect to the $ B_\infty $-valued inner product.
\begin{lem}(\cite{DGGJ}) \label{PPbasispreC*}
	There exists a subset $ \mscr S $ of $ H_0 $ such that $ \displaystyle \sum_{\sigma \in \mscr S} \sigma \circ  \sigma^* \  = \ 1_{\Lambda_0 m_0}$; moreover, any such $ \mscr S $ is a PP-basis for the right $ B_\infty $-module $ H_\infty $.
\end{lem}

\begin{thm}(\cite{DGGJ}) \label{flatthm}
	Starting from a $ 2 $-cell $ \left\{\eta^{(k)} \in \normalfont \t{NT}\left(\Lambda_k , \Omega_k \right)\right\}_{k\geq K} $, we have an intertwiner $ \Phi_{\eta^{(k)}_{\Gamma_k \cdots \Gamma_1 m_0}} \in \vphantom{\mcal L}_{A_\infty} \mcal L_{B_\infty} (H_\infty, G_\infty) $ which is independent of $ k \geq K $.
	
	Conversely, for every $ T \in \vphantom{\mcal L}_{A_\infty} \mcal L_{B_\infty} (H_\infty, G_\infty) $ ($ = $ the space of $ A_\infty $-$ B_\infty $-linear adjointable operator) and for all $ k \geq K_T\coloneqq \min \left\{l : T H_0 \subset G_l \right\} $, there exists unique $ \eta^{(k)} \in \normalfont \t{NT} (\Lambda_k , \Omega_k) $ such that $ T  =  \Phi_{\eta^{(k)}_{\Gamma_k \cdots \Gamma_1 m_0}} $.
	Further, $ \left(\eta^{(k)} , \eta^{(k+1)}\right) $ satisfies exchange relation for all $ k \geq K_T $.
\end{thm}

Clearly \textbf{UC} becomes a pre-C*-2-category.

\begin{rem}
	We will denote the object $m_0$ by dashed lines and any other object by dotted lines in (ii) of the pictorial notations mentioned at the beginning of \Cref{Graphcalc}.
\end{rem}

\section{Q-system in \textbf{UC}}
In this section, given a Q-system in \textbf{UC} for a $0$-cell, we explore certain structural properties of the associated bimodules that will further enable us to construct new $0$-cells and a new dualizable $1$-cell in the next section, that will implement Q-system completion of $\textbf{UC}$.  

\vspace*{2mm}

Let $(\Gamma_\bullet,\mcal M_\bullet)$ be a $0$-cell in \textbf{UC}.
and $(Q_\bullet,W_\bullet^Q,m_\bullet,i_\bullet)$ be a Q-system in $\textbf{UC}_1 ((\Gamma_\bullet,\mcal M_\bullet),(\Gamma_\bullet,\mcal M_\bullet))$. Graphicaly,
each natural transformation $m_k, i_k$ and $W_{k+1}^Q$ will be represented by the following respective diagrams:

\[
m_k \coloneqq
\raisebox{-6mm}{
	\begin{tikzpicture}
	\draw[red,in=90,out=90,looseness=2] (-0.5,0.5) to (-1.2,0.5);
	\draw[red] (-1.2,0.5) to (-1.2,-.2); 
	\draw[red] (-.5,0.5) to (-.5,-.2);
	\node at (-0.85,0.9) {$\red{\bullet}$};
	\draw[red] (-.85,0.9) to (-.85,1.6);
	\node[left,scale=0.7] at (-.85,1.4) {$Q_k$};
	\node[left,scale=0.7] at (-1.2,0) {$Q_k$};
	\node[right,scale=0.7] at (-.5,0) {$Q_k$};
	\end{tikzpicture}
}\!\!\! \ \ , \ \ \ \
i_k \coloneqq \ \
\raisebox{-6mm}{
	\begin{tikzpicture}
	\draw [red] (-0.8,-.6) to (-.8,1);
	\node at (-.8,-.6) {$\red{\bullet}$};
	\node[left,scale=0.7] at (-.8,0) {$Q_k$};
	\end{tikzpicture}
}\ \ \ \ \t{and} \ \ \ \
W_{k+1}^Q \coloneqq
\raisebox{-8mm}{
\begin{tikzpicture}
\draw[in=-90,out=90] (0,0) to (.5,2);
\draw[white,line width=1mm,in=-90,out=90] (.7,0) to (-.1,2); 
\draw[red,in=-90,out=90] (.7,0) to (-.1,2);
\node[scale=0.7] at (-.4,0) {$\Gamma_{k+1}$}; 
\node[scale=0.7] at (1,0) {$Q_k$};
\node[scale=0.7] at (-.4,1.7) {$Q_{k+1}$}; 
\node[scale=0.7] at (.8,1.5) {$\Gamma_{k+1}$};
\end{tikzpicture}}
\ \ \forall k \geq 0 \ \
\].

Pictorially, exchange relation of $m_k$'s and $i_k$'s with respect to $W_\bullet$ will be denoted as follows:
\[
\raisebox{-6mm}{
	\begin{tikzpicture}[rotate=40,transform shape]
	\draw (-2,1.3) to (.5,1.3);
	\draw[white,line width=1mm] (-.85,1) to (-.85,1.6);
	\draw[red,in=90,out=90,looseness=2] (-0.5,0.5) to (-1.2,0.5);
	\draw[red] (-1.2,0.5) to (-1.2,-.2); 
	\draw[red] (-.5,0.5) to (-.5,-.2);
	\node at (-0.85,0.9) {$\red{\bullet}$};
	\draw[red] (-.85,0.9) to (-.85,1.6);
	\node[left,scale=0.7] at (-.85,1.1) {$Q_k$};
	\node[left,scale=0.7] at (-.85,1.6) {$Q_{k+1}$};
	\node[left,scale=0.7] at (-1.2,0) {$Q_k$};
	\node[right,scale=0.7] at (-.5,0) {$Q_k$};
	\node[left,scale=0.7] at (-1.4,1) {$\Gamma_{k+1}$};
	\node[left,scale=0.7] at (.5,1) {$\Gamma_{k+1}$};
	\end{tikzpicture}
}
=
\raisebox{-4mm}{
	\begin{tikzpicture}[rotate=40,transform shape]
	\draw (-2,.2) to (.5,.2);
	\draw[white,line width=1mm] (-1.2,.5) to (-1.2,-.2);
	\draw[white,line width=1mm] (-.5,.5) to (-.5,-.2);
	\draw[red,in=90,out=90,looseness=2] (-0.5,0.5) to (-1.2,0.5);
	\draw[red] (-1.2,0.5) to (-1.2,-.2); 
	\draw[red] (-.5,0.5) to (-.5,-.2);
	\node at (-0.85,0.9) {$\red{\bullet}$};
	\draw[red] (-.85,0.9) to (-.85,1.6);
	\node[left,scale=0.7] at (-.85,1.1) {$Q_{k+1}$};
	\node[left,scale=0.7] at (-1.2,0) {$Q_k$};
	\node[right,scale=0.7] at (-.5,0) {$Q_k$};
	\node[left,scale=0.7] at (-1.4,.4) {$\Gamma_{k+1}$};
	\node[left,scale=0.7] at (.5,.4) {$\Gamma_{k+1}$};
	\end{tikzpicture}} \ \
\t{and} \ \
\raisebox{-6mm}{
	\begin{tikzpicture}[rotate=40,transform shape]
	\draw (-1.7,.5) to (.1,.5);
	\draw[white,line width=1mm] (-0.8,-.6) to (-.8,1);
	\draw [red] (-0.8,-.6) to (-.8,1);
	\node at (-.8,-.6) {$\red{\bullet}$};
	\node[left,scale=0.7] at (-.8,0) {$Q_k$};
	\node[left,scale=0.7] at (-.8,.7) {$Q_{k+1}$};
	\node[left,scale=0.7] at (-1.5,.4) {$\Gamma_{k+1}$};
	\node[left,scale=0.7] at (.3,.3) {$\Gamma_{k+1}$};
	\end{tikzpicture}}
=
\raisebox{-6mm}{
	\begin{tikzpicture}
	\draw [red] (-0.8,-.6) to (-.8,1);
	\node at (-.8,-.6) {$\red{\bullet}$};
	\node[left,scale=0.7] at (-.8,0) {$Q_{k+1}$};
	\draw (-.1,-.6) to (-.1,1);
	\node[right,scale=0.7] at (-.1,0) {$\Gamma_{k+1}$};
	\end{tikzpicture}}\]
eventually for all $ k $.

\begin{rem}\label{lremark}
	From \Cref{UCisomorphism}, we observe that for our $Q$-system $\left( Q_\bullet, m_\bullet, i_\bullet\right)$ in $\textbf{UC}$ the natural transformations $m_k$ and $i_k$ satisfy (Q1)-(Q4) as in \Cref{Qsysdefn} eventually for all $k \ $. For the rest of the paper we fix a natural number $l$ such that $m_k$ and $i_k$ satisfy (Q1)-(Q4) and the exchange relations for $k \geq l \ $. 
\end{rem}

Consider the filtration of finite dimensional C*-algebras $\left\{A_k \coloneqq \t{End}(\Gamma_k \cdots \Gamma_1 m_0)\right\}_{k\geq 1} $ associated to the $ 0 $-cell $ \Gamma_\bullet $ where $m_0$ is direct sum of a maximal set of mutually non-isomorphic simple objects in $\mcal M_0$.
Let $\left\{H_k \coloneqq \mcal M_k(\Gamma_k \cdots \Gamma_1 m_0,Q_k \Gamma_k \cdots \Gamma_1 m_0 )\right\}_{k \geq 1}$ be the right correspondence associated to $ Q_\bullet $.
By construction (\Cref{Hkbimod} and \Cref{Hkinp}) , $H_k$ is a right $A_k$-$A_k$ correspondence.
Thus, one may view $ H_k $ as a $ 1 $-cell in the $ 2 $-category \textbf{C*Alg} of right correspondence bimodules over pairs of C*-algebras.
We will further establish that each $H_k$ is a Q-system in $\textbf{C*Alg}(A_k , A_k )$ for $k \geq l \ $.
In order to do this, we will use the following identification.
\begin{rem}\label{connesfus}
Let $ \left\{Y_k \coloneqq \mcal M_k(\Gamma_k \cdots \Gamma_1 m_0,Q_k^{2} \Gamma_k \cdots \Gamma_1 m_0 ) \right\}_{k\geq 1} $ denote the right correspondence associated to the $ 1 $-cell $ Q_\bullet \btimes Q_\bullet $ in \textbf{UC}.
The proof of [Propn. 3.12 of \cite{DGGJ}] tells us that the map $\xi \boxtimes \eta \longmapsto $
\raisebox{-18mm}{
	\begin{tikzpicture}
		\draw[dashed] (0.5,0.4) -- (0.5,1.2);
		\draw (0.4,0.4) to (0.4,1.2);
		\draw (-0.25,0.4) to (-0.25,1.2);
		\node[left] at (0.5,0.6) {$\cdots$};
		\draw (0.4,-0.15) to (0.4,-1.2);
		\draw (-0.25,-0.15) to (-0.25,-1.2);
		\draw[dashed] (0.5,-0.2) -- (0.5,-1.2);
		\node[left] at (0.5,-0.6) {$\cdots$};
		\node[draw,thick,rounded corners,minimum width=56] at (-0.3,-1.45) {$\eta$};
		\draw (0.4,-1.7) to (0.4,-2.2);
		\draw (-0.25,-1.7) to (-0.25,-2.2);
		\draw[dashed] (0.5,-1.7) -- (0.5,-2.2);
		\node[left] at (0.5,-2) {$\cdots$};
		\draw[red] (-0.5,0.4) to (-0.5,1) ;
		\draw[red] (-1.2,1) to (-1.2,-1.21); 
		\node[left] at (-1.1,0) {$Q_k$};
		\node [right,scale=0.5] at (-.25,1) {$\Gamma_k$};
		\node [right,scale=0.5] at (0.08,1.2) {$\Gamma_1$};
		\node [right,scale=0.7] at (0.5,1) {$m_0$};
		\node[draw,thick,rounded corners,fill=white,minimum width=40] at (0,0.14) {$\xi$}; 
\end{tikzpicture}}
is an isomorphism between $H_k \us {A_k} \boxtimes H_k$ and $ Y_k $ as right $A_k$-$A_k$ correspondence.
\end{rem}
Via the above identification, the multiplication $ 2 $-cell $ m_\bullet $ and the unit $ 2 $-cell $ i_\bullet $ in \textbf{UC} corresponds to the maps $\widetilde m_k : H_k \us{A_k}\boxtimes H_k \to H_k$ and $\widetilde i_k : A_k \to H_k $ respectively at the level of bimodules; more explicitly 
\[
\widetilde m_k(\xi \btimes \eta)\coloneqq
\raisebox{-18mm}{
             \begin{tikzpicture}
             \draw[dashed] (0.5,0.4) -- (0.5,1.4);
             \draw (0.4,0.4) to (0.4,1.4);
             \draw (-0.25,0.4) to (-0.25,1.4);
             \node[left] at (0.5,0.6) {$\cdots$};
             \draw (0.4,-0.15) to (0.4,-1.2);
             \draw (-0.25,-0.15) to (-0.25,-1.2);
             \draw[dashed] (0.5,-0.2) -- (0.5,-1.2);
             \node[left] at (0.5,-0.6) {$\cdots$};
             \node[draw,thick,rounded corners,minimum width=56] at (-0.3,-1.45) {$\eta$};
             \draw (0.4,-1.7) to (0.4,-2.2);
             \draw (-0.25,-1.7) to (-0.25,-2.2);
             \draw[dashed] (0.5,-1.7) -- (0.5,-2.2);
             \node[left] at (0.5,-2) {$\cdots$};
             \draw[red] (-0.5,0.4) to (-0.5,0.5) ;
             \draw[red,in=90,out=90,looseness=2] (-0.5,0.5) to (-1.2,0.5);
             \draw[red] (-1.2,0.5) to (-1.2,-1.21); 
             \node at (-0.85,0.9) {$\red{\bullet}$};
             \draw[red] (-.85,.9) to (-.85,1.4);
             \node[left] at (-1.1,0) {$Q_k$};
             \node [right,scale=0.8] at (-.8,1.2) {$\Gamma_k$};
             \node [right,scale=0.8] at (-.1,1.2) {$\Gamma_1$};
             \node at (0.85,1) {$m_0$};
             \node[draw,thick,rounded corners,fill=white,minimum width=40] at (0,0.14) {$\xi$}; 
             \end{tikzpicture}}\!\!\! \ \              
             \t{and} \ \ \ \widetilde i_k(x) \coloneqq
              			\raisebox{-10mm}{
						\begin{tikzpicture}
						\node[draw,thick,rounded corners,minimum width=30] at (0,0) {$x$};
						\draw (0.3,0.21) to (.3,1);
						\draw (-0.3,0.21) to (-.3,1);
						\draw (0.3,-0.21) to (.3,-1);
						\draw (-0.3,-.21) to (-.3,-1);
						\draw[dashed] (0.4,0.27) to (.4,1);
						\draw[dashed] (0.4,-0.27) to (.4,-1);
						\node at (0.05,0.4) {$\cdots$};
						\node at (0.05,-0.5) {$\cdots$};
						\node[left,scale=0.8] at (-.2,0.8) {$\Gamma_k$};
						\node[left,scale=0.8] at (0.4,0.8) {$\Gamma_1$};
						\node[right] at (0.3,.6) {$m_0$};
						\draw [red] (-0.8,-1) to (-.8,1);
						\node at (-.8,-.95) {$\red{\bullet}$};
						\node[left,scale=0.7] at (-.8,0) {$Q_k$};
						\end{tikzpicture}} \ .
\]
\begin{prop}
For each $k \geq l$, $\widetilde m_k$ and $\widetilde i_k$ are adjointable maps and hence $ 2 $-cells in {\normalfont\textbf{C*Alg}}.
Moreover, $(H_k,\widetilde m_k,\widetilde i_k)$ becomes a Q-system in $\normalfont\textbf{C*Alg}(A_k , A_k)$ for each $k \geq l \ $.
\end{prop} 
\begin{proof}
Using the identification in \Cref{connesfus}, the adjoint of $\widetilde m_k$ is given by 
\raisebox{-12mm}{
\begin{tikzpicture}
\draw (0.6,0.28) to (.6,1);
\draw (-.1,0.28) to (-.1,1);
\draw (-.1,-0.28) to (-.1,-1);
\draw (.6,-.28) to (.6,-1);
\draw[dashed] (0.7,0.28) to (.7,1);
\draw[dashed] (0.7,-0.28) to (.7,-1);
\node at (0.3,0.5) {$\cdots$};
\node at (0.3,-0.5) {$\cdots$};
\node[right,scale=0.6] at (-.1,0.8) {$\Gamma_k$};
\node[left,scale=0.6] at (0.7,0.8) {$\Gamma_1$};
\node[right] at (0.6,.6) {$m_0$};
\draw[red,in=-90,out=-90,looseness=2] (-.8,1) to (-.25,1);
\draw[red] (-.5,.28) to (-.5,.6);
\node at (-.5,.65) {$\red{\bullet}$};
\node[left,scale=0.6] at (-.8,.7) {$Q_k$};
\node[draw,thick,rounded corners,fill=white,minimum width=48] at (0,0) {$\xi$};
\end{tikzpicture}}
and that of $\widetilde i_k$ is given by 
\raisebox{-10mm}{
\begin{tikzpicture}
\draw (0.4,0.28) to (.4,1);
\draw (-0.4,0.28) to (-.4,1);
\draw (0.4,-0.28) to (.4,-1);
\draw (-0.4,-.28) to (-.4,-1);
\draw[dashed] (0.5,0.28) to (.5,1);
\draw[dashed] (0.5,-0.28) to (.5,-1);
\node at (0.05,0.5) {$\cdots$};
\node at (0.05,-0.5) {$\cdots$};
\node[right,scale=0.6] at (-.4,0.8) {$\Gamma_k$};
\node[left,scale=0.6] at (0.5,0.8) {$\Gamma_1$};
\node[right] at (0.4,.6) {$m_0$};
\draw[red] (-.56,.28) to (-.56,.9);
\node at (-.56,.9) {$\red{\bullet}$};
\node[left,scale=0.6] at (-.5,.6) {$Q_k$};
\node[draw,thick,rounded corners,fill=white,minimum width=40] at (0,0) {$\xi$};
\end{tikzpicture}}.
Now using $\widetilde m_k$ and $\widetilde i_k$ and their adjoints, and the properties (Q1-Q4) of $ m_\bullet $ and $ i_\bullet  $  mentioned at in preliminaries , associativity, unitality, frobenius property and separability of $\left(H_k, \widetilde m_k, \widetilde i_k\right)$ easily follows. 		
\end{proof}
We now explore certain structural properties of $H_k$.

We prove the following proposition using ideas from \cite{CPJP}.
\begin{prop}\label{HkCstar}
For each $ k \geq l $, the space $H_k$ is a unital {\normalfont C*}-algebra with multiplication, adjoint and unit given by
\[
\xi \cdot \eta \coloneqq \widetilde m_k \left(\xi \btimes \eta \right) \ \t{ , } \ 
	\left(\raisebox{-8mm}{	
		\begin{tikzpicture}
			\draw (0.4,0.28) to (.4,.8);
			\draw (-0.4,0.28) to (-.4,.8);
			\draw (0.4,-0.28) to (.4,-.8);
			\draw (-0.4,-.28) to (-.4,-.8);
			\draw[dashed] (0.5,0.28) to (.5,.8);
			\draw[dashed] (0.5,-0.28) to (.5,-.8);
			\node at (0.05,0.47) {$\cdots$};
			\node at (0.05,-0.5) {$\cdots$};
			\node[right,scale=0.8] at (-1,-0.7) {$\Gamma_k$};
			\node[left,scale=0.8] at (0.5,-0.9) {$\Gamma_1$};
			\node[right] at (0.44,.5) {$m_0$};
			\draw[red] (-.56,.28) to (-.56,.8);
			\node[left,scale=0.8] at (-.5,.5) {$Q_k$};
			\node[draw,thick,rounded corners,fill=white,minimum width=40] at (0,0) {$\xi$};
	\end{tikzpicture}} 
	\right)^{\dagger} = 
	\raisebox{-10mm}{
		\begin{tikzpicture}
			\draw (0.4,0.28) to (.4,.8);
			\draw (-0.4,0.28) to (-.4,.8);
			\draw (0.4,-0.28) to (.4,-.8);
			\draw (-0.4,-.28) to (-.4,-.8);
			\draw[dashed] (0.5,0.28) to (.5,.8);
			\draw[dashed] (0.5,-0.28) to (.5,-.8);
			\node at (0.05,0.47) {$\cdots$};
			\node at (0.05,-0.5) {$\cdots$};
			\node[right,scale=0.8] at (-.6,1) {$\Gamma_k$};
			\node[left,scale=0.8] at (0.6,1) {$\Gamma_1$};
			\node[right] at (0.44,.44) {$m_0$};
			\draw[red] (-.56,-.28) to (-.56,-.4);
			\draw[red,in=-90,out=-90,looseness=2] (-.56,-.4) to (-.9,-.4);
			\draw[red] (-.9,-.4) to (-.9,.8);
			\draw[red] (-.72,-.6) to (-.72,-.8);
			\node at (-.72,-.6) {$\red{\bullet}$};
			\node at (-.72,-.85) {$\red{\bullet}$};
			\node[left,scale=0.8] at (-.8,.47) {$Q_k$};
			\node[draw,thick,rounded corners,fill=white,minimum width=40] at (0,0) {$\xi^*$};
	\end{tikzpicture}}
\ \t{ and } \ 
	\raisebox{-6mm}{ 
	\begin{tikzpicture}
		\draw [red] (-0.8,-.6) to (-.8,.4);
		\node at (-.8,-.6) {$\red{\bullet}$};
		\node[left,scale=0.8] at (-.8,0) {$Q_k$};
		\draw (-.6,-.6) to (-.6,.4);
		\draw (0,-.6) to (0,.4);
		\draw[dashed] (0.2,-.6) to (0.2,.4); 
		\node at (-.3,-.2) {$\cdots$};
		\node[right,scale=0.8] at (-.85,0.6) {$\Gamma_k$};
		\node[left,scale=0.8] at (0.25,0.6) {$\Gamma_1$};
		\node[right] at (0.2,0.1) {$m_0$};
\end{tikzpicture}} 
\]
respectively for $ \xi, \eta \in H_k $.
\end{prop}
\begin{proof}
Indeed, $\xi^{\dagger \dagger}=\xi$.
Again
\[
(\xi_1 . \xi_2)^{\dagger}=
\raisebox{-12mm}{
	\begin{tikzpicture}
		\draw[dashed] (.5,.3) to (.5,.6);
		\draw (.3,.3) to (.3,.6);
		\draw (-.3,.3) to (-.3,.6);
		\draw[dashed] (.5,-.3) to (.5,-.6);
		\draw (.3,-.3) to (.3,-.6);
		\draw (-.3,-.3) to (-.3,-.6);
		\draw[dashed] (.5,-1.2) to (.5,-1.5);
		\draw (.3,-1.2) to (.3,-1.5);
		\draw (-.2,-1.2) to (-.2,-1.5);
		\draw[red] (-.65,-.28) to (-.65,-1.5);
		\draw[red] (-.3,-1.2) to (-.3,-1.5);
		\draw[red,in=-90,out=-90,looseness=2] (-.65,-1.5) to (-.3,-1.5);
		\draw[red] (-.5,-1.7) to (-.5,-1.9);
		\draw[red,in=-90,out=-90,looseness=2] (-.5,-1.9) to (-1,-1.9);
		\draw[red] (-1,-1.9) to (-1,.6);
		\node[scale=0.7] at (-.5,-1.7) {$\red{\bullet}$};
		\node[scale=0.7] at (-.75,-2.2) {$\red{\bullet}$};
		\draw[red] (-.75,-2.2) to (-.75,-2.4);
		\node[scale=0.7] at (-.75,-2.4) {$\red{\bullet}$};
		\node at (0,.4) {$\cdots$};
		\node at (0,-.45) {$\cdots$};
		\node at (.1,-1.4) {$\cdots$};
		\node[draw,thick,rounded corners,fill=white,minimum width=40] at (0,0) {${\xi_2}^*$};
		\node[draw,thick,rounded corners,fill=white,minimum width=30] at (0.1,-.9) {${\xi_1}^*$};
\end{tikzpicture}}
=
\raisebox{-12mm}{
	\begin{tikzpicture}
		\draw[dashed] (.5,.3) to (.5,.6);
		\draw (.3,.3) to (.3,.6);
		\draw (-.3,.3) to (-.3,.6);
		\draw[dashed] (.5,-.3) to (.5,-.6);
		\draw (.3,-.3) to (.3,-.6);
		\draw (-.3,-.3) to (-.3,-.6);
		\draw[dashed] (.5,-1.2) to (.5,-1.5);
		\draw (.3,-1.2) to (.3,-1.5);
		\draw (-.2,-1.2) to (-.2,-1.5);
		\draw[red] (-.65,-.28) to (-.65,-1.5);
		\draw[red] (-.3,-1.2) to (-.3,-1.5);
		\draw[red,in=-90,out=-90,looseness=2] (-.8,-1.7) to (-.3,-1.7);
		\draw [red] (-.3,-1.7) to (-.3,-1.5);
		\draw[red] (-.55,-2) to (-.55,-2.2);
		\draw[red,in=-90,out=-90,looseness=2] (-1,-1.5) to (-.65,-1.5);
		\draw[red] (-1,-1.5) to (-1,.6);
		\node[scale=0.7] at (-.8,-1.7) {$\red{\bullet}$};
		\node[scale=0.7] at (-.55,-2.2) {$\red{\bullet}$};
		\node[scale=0.7] at (-.55,-2) {$\red{\bullet}$};
		\node at (.1,-1.35) {$\cdots$};
		\node at (.05,-.45) {$\cdots$};
		\node at (.05,.45) {$\cdots$};
		\node[draw,thick,rounded corners,fill=white,minimum width=40] at (0,0) {${\xi_2}^*$};
		\node[draw,thick,rounded corners,fill=white,minimum width=30] at (0.1,-.9) {${\xi_1}^*$};
\end{tikzpicture}}=
\raisebox{-12mm}{
	\begin{tikzpicture}
		\draw[dashed] (.5,.3) to (.5,.6);
		\draw (.3,.3) to (.3,.6);
		\draw (-.3,.3) to (-.3,.6);
		\draw[dashed] (.5,-.3) to (.5,-.6);
		\draw (.3,-.3) to (.3,-.6);
		\draw (-.3,-.3) to (-.3,-.6);
		\draw[dashed] (.5,-1.2) to (.5,-1.5);
		\draw (.3,-1.2) to (.3,-1.5);
		\draw (-.2,-1.2) to (-.2,-1.5);
		\draw[red] (-.45,-.28) to (-.45,-.4);
		\draw[red,in=-90,out=-90,looseness=2] (-1.15,-1.2) to (-.3,-1.2);
		\draw[red] (-1.15,-1.2) to (-1.15,0);
		\draw[red] (-.72,-1.7) to (-.72,-2);
		\draw[red,in=-90,out=-90,looseness=2] (-.45,-.4) to (-.75,-.4);
		\draw[red] (-.75,-.4) to (-.75,0);
		\draw[red] (-.95,.22) to (-.95,.6);
		\draw[red,in=90,out=90,looseness=2] (-1.15,0) to (-.75,0);
		\node[scale=0.7] at (-.95,.22) {$\red{\bullet}$};
		\node[scale=0.7] at (-.72,-1.7) {$\red{\bullet}$};
		\node[scale=0.7] at (-.72,-2) {$\red{\bullet}$};
		\node[scale=0.7] at (-.6,-.8) {$\red{\bullet}$};
		\node[scale=0.7] at (-.6,-.6) {$\red{\bullet}$};
		\draw[red] (-.6,-.6) to (-.6,-.8);
		\node at (.1,-1.35) {$\cdots$};
		\node at (.05,-.45) {$\cdots$};
		\node at (.05,.45) {$\cdots$};
		\node[draw,thick,rounded corners,fill=white,minimum width=40] at (0.1,0) {${\xi_2}^*$};
		\node[draw,thick,rounded corners,fill=white,minimum width=30] at (0.1,-.9) {${\xi_1}^*$};
\end{tikzpicture}}
=\xi_2^{\dagger} . \xi_1^{\dagger}
\]
where the second equality follows from associativity and the third comes from Frobenius and unitality conditions.
Also,
\[
1_{H_k} \cdot \xi = 
	 \raisebox{-6mm}{	
	\begin{tikzpicture}
	\draw (0.4,0.28) to (.4,.8);
	\draw (-0.4,0.28) to (-.4,.8);
	\draw (0.4,-0.28) to (.4,-.8);
	\draw (-0.4,-.28) to (-.4,-.8);
	\draw[dashed] (0.5,0.28) to (.5,.8);
	\draw[dashed] (0.5,-0.28) to (.5,-.8);
	\node at (0.05,0.5) {$\cdots$};
	\node at (0.05,-0.5) {$\cdots$};
	\node[right,scale=0.6] at (-.4,0.7) {$\Gamma_k$};
	\node[left,scale=0.6] at (0.5,0.7) {$\Gamma_1$};
	\node[right,scale=.8] at (0.44,.44) {$m_0$};
	\draw[red] (-.76,.6) to (-.76,1);
	\draw[red] (-.56,.26) to (-.56,.4);
	\draw[red] (-.96,.4) to (-.96,-.4);
	\node[left,scale=0.6] at (-.9,.2) {$Q_k$};
	\draw[red,in=90,out=90,looseness=2] (-.56,.4) to (-.96,.4);
	\node at (-.96,-.4) {$\red{\bullet}$};
	\node at (-.76,.6) {$\red{\bullet}$};
	\node[draw,thick,rounded corners,fill=white,minimum width=40] at (0,0) {$\xi$};
	\end{tikzpicture}} =
\raisebox{-6mm}{	
	\begin{tikzpicture}
		\draw (0.4,0.28) to (.4,.8);
		\draw (-0.4,0.28) to (-.4,.8);
		\draw (0.4,-0.28) to (.4,-.8);
		\draw (-0.4,-.28) to (-.4,-.8);
		\draw[dashed] (0.5,0.28) to (.5,.8);
		\draw[dashed] (0.5,-0.28) to (.5,-.8);
		\node at (0.05,0.5) {$\cdots$};
		\node at (0.05,-0.5) {$\cdots$};
		\node[right,scale=0.6] at (-.4,0.7) {$\Gamma_k$};
		\node[left,scale=0.6] at (0.5,0.7) {$\Gamma_1$};
		\node[right,scale=.8] at (0.44,.44) {$m_0$};
		\draw[red] (-.56,.28) to (-.56,.8);
		\node[left,scale=0.6] at (-.5,.47) {$Q_k$};
		\node[draw,thick,rounded corners,fill=white,minimum width=40] at (0,0) {$\xi$};
\end{tikzpicture}} =
\raisebox{-6mm}{	
	\begin{tikzpicture}
	\draw (0.4,0.28) to (.4,.8);
	\draw (-0.4,0.28) to (-.4,.8);
	\draw (0.4,-0.28) to (.4,-.8);
	\draw (-0.4,-.28) to (-.4,-.8);
	\draw[dashed] (0.5,0.28) to (.5,.8);
	\draw[dashed] (0.5,-0.28) to (.5,-.8);
	\node at (0.05,0.5) {$\cdots$};
	\node at (0.05,-0.5) {$\cdots$};
	\node[right,scale=0.6] at (-.4,0.7) {$\Gamma_k$};
	\node[left,scale=0.6] at (0.5,0.7) {$\Gamma_1$};
	\node[right,scale=.8] at (0.44,.44) {$m_0$};
	\draw[red] (-.9,.28) to (-.9,.8);
	\draw[red,in=90,out=90,looseness=2] (-.9,.8) to (-.6,.8);
	\node[left,scale=0.6] at (-.85,.47) {$Q_k$};
	\node at (-.6,.7) {$\red{\bullet}$};
	\node[scale=.8] at (-.75,.95) {$\red{\bullet}$};
	\draw[red] (-.75,.95) to (-.75,1.2);
	\node[draw,thick,rounded corners,fill=white,minimum width=56] at (0,0) {$\xi$};
	\end{tikzpicture}} = \xi \cdot 1_{H_k}. 
\]
Hence, $H_k$ becomes a unital *-algebra.

To prove that $ H_k $ is a C*-algebra, we show that it is isomorphic to a *-subalgebra of a finite dimensional C*-algebra.
Define
\[
S_k \coloneqq  \left\{ x \in \t{End}(Q_k \Gamma_k \cdots \Gamma_1 m_0) \ \left|
 \raisebox{-6mm}{
 \begin{tikzpicture}
 \draw (0.4,0.2) to (.4,.8);
 \draw (-0.4,0.2) to (-.4,.8);
 \draw (0.4,-0.2) to (.4,-.8);
 \draw (-0.4,-.2) to (-.4,-.8);
 \draw[dashed] (0.5,0.2) to (.5,.8);
 \draw[dashed] (0.5,-0.2) to (.5,-.8);
 \node at (0.05,0.4) {$\cdots$};
 \node at (0.05,-0.5) {$\cdots$};
 \node[right,scale=0.6] at (-.4,0.6) {$\Gamma_k$};
 \node[left,scale=0.6] at (0.5,0.6) {$\Gamma_1$};
 \node[right] at (0.4,.4) {$m_0$};
 \draw[red] (-.96,.2) to (-.96,-.6);
 \node[left,scale=0.6] at (-.7,.6) {$Q_k$};
 \draw[red] (-.76,.4) to (-.76,.8);
 \draw[red,in=90,out=90,looseness=2] (-.56,.2) to (-.96,.2);
 \node at (-.76,.4) {$\red{\bullet}$};
 \node[draw,thick,rounded corners,fill=white,minimum width=40] at (0,0) {$x$};
 \end{tikzpicture}}=
\raisebox{-6mm}{
	\begin{tikzpicture}
	\draw (0.4,0.2) to (.4,.8);
	\draw (-0.4,0.2) to (-.4,.8);
	\draw (0.4,-0.2) to (.4,-.8);
	\draw (-0.4,-.2) to (-.4,-.8);
	\draw[dashed] (0.5,0.2) to (.5,.8);
	\draw[dashed] (0.5,-0.2) to (.5,-.8);
	\node at (0.05,0.4) {$\cdots$};
	\node at (0.05,-0.5) {$\cdots$};
	\node[right,scale=0.6] at (-.4,0.6) {$\Gamma_k$};
	\node[left,scale=0.6] at (0.5,0.6) {$\Gamma_1$};
	\node[right] at (0.4,.4) {$m_0$};
	\draw[red] (-.9,.2) to (-.9,.8);
	\node[left,scale=0.6] at (-.85,.4) {$Q_k$};
	\draw[red] (-.9,-.2) to (-.9,-.4);
	\draw[red,in=90,out=90,looseness=2] (-1.1,-.6) to (-.7,-.6);
	\node at (-.9,-.4) {$\red{\bullet}$};
	\draw[red] (-1.1,-.6) to (-1.1,-.8);
	\draw[red] (-.7,-.6) to (-.7,-.8);
	\node[draw,thick,rounded corners,fill=white,minimum width=56] at (0,0) {$x$};
	\end{tikzpicture}}
\right.  \right\}
\]
sitting inside the finite dimensional C*-algebra $\t{End}(Q_k\Gamma_k\cdots\Gamma_1 m_0)$.
Clearly $ S_k $ is closed under multiplication, as well as *-closed (using Frobenius property and unitality).
Define $\phi_1^{(k)} : H_k \to S_k$ and $\phi_2^{(k)} : S_k \to H_k$ as follows:\\
  $\phi_1^{(k)} (\xi) \coloneqq \raisebox{-6mm}{	
  	\begin{tikzpicture}
  	\draw (0.4,0.28) to (.4,.8);
  	\draw (-0.4,0.28) to (-.4,.8);
  	\draw (0.4,-0.28) to (.4,-.8);
  	\draw (-0.4,-.28) to (-.4,-.8);
  	\draw[dashed] (0.5,0.28) to (.5,.8);
  	\draw[dashed] (0.5,-0.28) to (.5,-.8);
  	\node at (0.05,0.47) {$\cdots$};
  	\node at (0.05,-0.5) {$\cdots$};
  	\node[right,scale=0.6] at (-.4,0.7) {$\Gamma_k$};
  	\node[left,scale=0.6] at (0.5,0.7) {$\Gamma_1$};
  	\node[right,scale=.8] at (0.44,.44) {$m_0$};
  	\draw[red] (-.86,.6) to (-.86,1);
	\draw[red] (-1.2,.25) to (-1.2,-.65);
	\draw[red,in=90,out=90,looseness=2] (-.56,.25) to (-1.2,.25);
	\node at (-.86,.6) {$\red{\bullet}$};
  	\node[left,scale=0.6] at (-.8,.81) {$Q_k$};
  	\node[draw,thick,rounded corners,fill=white,minimum width=40] at (0,0) {$\xi$};
  	\end{tikzpicture}} \in S_k$ (by associativity of $Q_k$) 
and $\phi_2^{(k)} (x) \coloneqq \raisebox{-6mm}{	
  	\begin{tikzpicture}
  	\draw (0.4,0.2) to (.4,.8);
  	\draw (-0.4,0.2) to (-.4,.8);
  	\draw (0.4,-0.2) to (.4,-.8);
  	\draw (-0.4,-.2) to (-.4,-.8);
  	\draw[dashed] (0.5,0.2) to (.5,.8);
  	\draw[dashed] (0.5,-0.2) to (.5,-.8);
  	\node at (0.05,0.4) {$\cdots$};
  	\node at (0.05,-0.5) {$\cdots$};
  	\node[right,scale=0.6] at (-.4,0.6) {$\Gamma_k$};
  	\node[left,scale=0.6] at (0.5,0.6) {$\Gamma_1$};
  	\node[right] at (0.4,.4) {$m_0$};
  	\draw[red] (-.9,.2) to (-.9,.8);
  	\node[left,scale=0.6] at (-.85,.4) {$Q_k$};
  	\draw[red] (-.9,-.2) to (-.9,-.8);
  	\node at (-.9,-.8) {$\red{\bullet}$};
  	\node[draw,thick,rounded corners,fill=white,minimum width=56] at (0,0) {$x$};
  	\end{tikzpicture}} $.
Now, it is routine to check using the axioms of $Q$-systems that $\phi_1^{(k)}$ and $\phi_2^{(k)}$ are unital, *-homomorphisms. Also, $\phi_1^{(k)}$ and $\phi_2^{(k)}$ are mutually inverse to each other, hence they are isomorphisms.   	             
\end{proof}
\begin{lem}\label{HkCE}
The map $ \widetilde i_k : A_k \to H_k $ defined by $\widetilde i_k(a) \coloneqq
\raisebox{-8mm}{
	\begin{tikzpicture}
		\node[draw,thick,rounded corners,minimum width=35] at (0,0) {$a$};
		\draw (0.4,0.2) to (.4,.8);
		\draw (-0.4,0.2) to (-.4,.8);
		\draw (0.4,-0.2) to (.4,-.8);
		\draw (-0.4,-.2) to (-.4,-.8);
		\draw[dashed] (0.5,0.2) to (.5,.8);
		\draw[dashed] (0.5,-0.2) to (.5,-.8);
		\node at (0.05,0.4) {$\cdots$};
		\node at (0.05,-0.5) {$\cdots$};
		\node[right,scale=0.6] at (-.4,0.6) {$\Gamma_k$};
		\node[left,scale=0.6] at (0.5,0.6) {$\Gamma_1$};
		\node[right] at (0.4,.4) {$m_0$};
		\draw[red] (-.7,-.7) to (-.7,.7);
		\node at (-.7,-.7) {$\red{\bullet}$};
		\node[left,scale=0.7] at (-.7,0) {$Q_k$};
\end{tikzpicture}} $ is a unital inclusion of {\normalfont C*}-algebras.
In the reverse direction, the map $ E_k : H_k \ra A_k $ defined by $E_k(\xi) \coloneqq 
\raisebox{-7mm}{
\begin{tikzpicture}
		\draw (0.6,-.8) to (.6,.8);
		\draw (-0.2,-.8) to (-.2,.8);
		\draw[dashed] (0.7,-.8) to (.7,.8);
		\node at (0.15,0.4) {$\cdots$};
		\node at (0.15,-0.5) {$\cdots$};
		\node[right,scale=0.6] at (-.25,0.66) {$\Gamma_k$};
		\node[left,scale=0.6] at (0.7,0.66) {$\Gamma_1$};
		\node[right] at (0.6,.47) {$m_0$};
		\draw[red] (-.4,.2) to (-.4,.7);
		\node[scale=.7] at (-.6,.47) {$Q_k$};
		\node at (-.4,.7) {$\red{\bullet}$};
		\node[draw,thick,rounded corners,fill=white,minimum width=48] at (.2,0) {$\xi$};
		\node[draw,thick,rounded corners] at (-1.4,0) {$d_{Q_k} ^{-1}$};
\end{tikzpicture}}$ (where  $d_{Q_k} = \raisebox{-4mm}{\begin{tikzpicture}
\draw[red] (0,0) to (0,.8);
\node at (0,0) {$\red{\bullet}$};
\node at (0,.8) {$\red{\bullet}$};
\node[scale=.7] at (-.18,.4) {$Q_k$};
\end{tikzpicture}}$ ) is a finite index, faithful conditional expectation satisfying $ E_k \left( \eta^\dagger \cdot \xi \right)  = 
\raisebox{-9 mm}{
\begin{tikzpicture}
\draw (-.4,-.9) to (-.4,.9);
\draw (.4,-.9) to (.4,.9);
\draw[dashed] (.6,-.9) to (.6,.9);
\node at (0,.5) {$ \cdots $};
\node at (0,-.5) {$ \cdots $};
\node at (.9,.7) {$ m_0 $};
\node at (.9,-.7) {$ m_0 $};
\node at (-.6,.7) {$ \Gamma_k $};
\node at (-.6,-.7) {$ \Gamma_k $};
\node[draw,fill=white,thick,rounded corners] at (.1,0) {$\left\lab \xi , \eta \right \rab_{A_k}$};
\node[draw,thick,rounded corners] at (-1.4,0) {$d_{Q_k} ^{-1}$};
\end{tikzpicture}}$ (where $ \left\lab \cdot , \cdot \right \rab_{A_k} $ is the right $ A_k $-valued inner product on $ H_k $ as defined in \Cref{Hkinp}) for each $k \geq l \ $.
\end{lem}
\begin{proof}
We make use of the *-algebra isomorphisms $ \phi_1, \phi_2 $ between $ H_k $ and $ S_k $ and find that the map $ \widetilde i_k : A_k \ra H_k $ corresponds to $ A_k \ni a \os{\displaystyle \phi_1^{(k)} \circ \widetilde i_k}{\longmapsto} Q_k \, a  = \raisebox{-6 mm}{
\begin{tikzpicture}
\draw (0.3,-0.6) to (.3,.6);
\draw (-0.3,-0.6) to (-.3,.6);
\draw[dashed] (.5,-.6) to (.5,.6);
\node at (0.8,.4) {$m_0$};
\node at (0.8,-.45) {$m_0$};
\draw[red] (-.6,-.6) to (-.6,.6);
\node at (-.85,0) {$Q_k$};
\node at (0.04,0.4) {$\cdots$};
\node at (0.04,-0.4) {$\cdots$};
\node[draw,thick,rounded corners,fill=white,minimum width=30] at (0.1,0) {$a$};
\end{tikzpicture}
}
\in S_k $ which is indeed an inclusion since $ Q_k $ is a bi-faithful functor.
Now, $ Q_k $ is symmetrically self-dual with the solution to conjugate equation given by \raisebox{-5 mm}{
\begin{tikzpicture}
\draw[red] (-.4,.2) to[in=180,out=-90] (0,-.2) to[out=0,in=-90] (.4,.2);
\draw[red] (0,-.2) to (0,-.6);
\node at (0,-.6) {\red{$\bullet$}};
\node at (0,-.2) {\red{$\bullet$}};
\end{tikzpicture}
}.
Thus, we have a conditional expectation given by
\[
S_k \ni x \os {\displaystyle E'} \longmapsto  \raisebox{-10 mm}{
\begin{tikzpicture}
\draw (0.3,-1) to (.3,1);
\draw (-0.3,-1) to (-.3,1);
\draw[dashed] (.5,-1) to (.5,1);
\node at (0.8,.4) {$m_0$};
\node at (0.8,-.45) {$m_0$};
\draw[red] (-.5,-.3) to (-.5,.3) to[out=90,in=0] (-.7,.5) to[out=180,in=90] (-.9,.3) to (-.9,-.3) to[out=-90,in=-180] (-.7,-.5) to[out=0,in=-90] (-.5,-.3);
\draw[red] (-.7,.5) to (-.7,.9);
\draw[red] (-.7,-.5) to (-.7,-.9);
\node at (-.7,.5) {\red{$\bullet$}};
\node at (-.7,-.5) {\red{$\bullet$}};
\node at (-.7,.9) {\red{$\bullet$}};
\node at (-.7,-.9) {\red{$\bullet$}};
\node at (0.04,0.4) {$\cdots$};
\node at (0.04,-0.4) {$\cdots$};
\node[draw,thick,rounded corners,fill=white,minimum width=40] at (0,0) {$x$};
\node[draw,thick,rounded corners] at (-1.5,0) {$d^{-1}_{Q_k}$};
\end{tikzpicture}
} =   \raisebox{-9 mm}{
\begin{tikzpicture}
	\draw (0.3,-1) to (.3,1);
	\draw (-0.3,-1) to (-.3,1);
	\draw[dashed] (.5,-1) to (.5,1);
	\node at (0.8,.4) {$m_0$};
	\node at (0.8,-.45) {$m_0$};
\draw[red] (-.5,-.6) to (-.5,.6);
\node at (-.5,-.6) {\red{$\bullet$}};
\node at (-.5,.6) {\red{$\bullet$}};
	\node at (0.04,0.4) {$\cdots$};
	\node at (0.04,-0.4) {$\cdots$};
	\node[draw,thick,rounded corners,fill=white,minimum width=40] at (0,0) {$x$};
	\node[draw,thick,rounded corners] at (-1.4,0) {$d^{-1}_{Q_k}$};
\end{tikzpicture}
} \in A_k 
\]
where the equality follows from the definition of $ S_k $ and separability axiom.
This conditional expectation is automatically faithful and translates into $ E_k $ (defined in the statement) via the *-isomorphism $ \phi_2^{(k)} $.
Now, for $x \in S_k ^ +$, we have
\[\raisebox{-8mm}{
	\begin{tikzpicture}
		\draw (0.4,0.2) to (.4,.8);
		\draw (-0.4,0.2) to (-.4,.8);
		\draw (0.4,-0.2) to (.4,-.8);
		\draw (-0.4,-.2) to (-.4,-.8);
		\draw[dashed] (0.5,0.2) to (.5,.8);
		\draw[dashed] (0.5,-0.2) to (.5,-.8);
		\node at (0.05,0.4) {$\cdots$};
		\node at (0.05,-0.5) {$\cdots$};
		\node[right,scale=0.6] at (-.4,0.6) {$\Gamma_k$};
		\node[left,scale=0.6] at (0.5,0.6) {$\Gamma_1$};
		\node[right] at (0.4,.4) {$m_0$};
		\draw[red] (-.6,.2) to (-.6,.7);
		\draw[red] (-.6,-.2) to (-.6,-.7);
		\node[draw,thick,rounded corners,fill=white,minimum width=48] at (0,0) {$x$};
		\end{tikzpicture}}
= 
\raisebox{-12mm}{
	\begin{tikzpicture}
		\draw (0.4,0.2) to (.4,.8);
		\draw (-0.4,0.2) to (-.4,.8);
		\draw (0.4,-0.2) to (.4,-.8);
		\draw (-0.4,-.2) to (-.4,-.8);
		\draw[dashed] (0.5,0.2) to (.5,.8);
		\draw[dashed] (0.5,-0.2) to (.5,-.8);
		\node at (0.05,0.4) {$\cdots$};
		\node at (0.05,-0.5) {$\cdots$};
		\node[right,scale=0.6] at (-.4,0.6) {$\Gamma_k$};
		\node[left,scale=0.6] at (0.5,0.6) {$\Gamma_1$};
		\node[right] at (0.4,.4) {$m_0$};
		\draw[red] (-.6,.2) to (-.6,.7);
		\draw[red] (-.6,-.2) to (-.6,-.7);
		\draw[red,in=90,out=90,looseness=2] (-.6,.7) to (-1.1,.7);
		\draw[red,in=-90,out=-90,looseness=2] (-1.1,.7) to (-1.6,.7);
		\draw[red] (-1.6,.7) to (-1.6,1);
		\draw[red,in=-90,out=-90,looseness=2] (-.6,-.7) to (-1.1,-.7);
		\draw[red,in=90,out=90,looseness=2] (-1.1,-.7) to (-1.6,-.7);
		\draw[red] (-1.6,-.7) to (-1.6,-1);
		\node[scale=0.7] at (-1.35,-.4) {$\red{\bullet}$};
		\node[scale=0.7] at (-.85,1) {$\red{\bullet}$};
		\node[scale=0.7] at (-1.35,-.2) {$\red{\bullet}$};
		\node[scale=0.7] at (-1.35,.35) {$\red{\bullet}$};
		\node[scale=0.7] at (-.85,-1) {$\red{\bullet}$};
		\draw[red] (-1.35,-.4) to (-1.35,-.2);
		\draw[red] (-1.35,.35) to (-1.35,.15);
		\node[scale=0.7] at (-1.35,.15) {$\red{\bullet}$};
		\draw[red] (-.85,1) to (-.85,1.2);
		\node[scale=0.7] at (-.85,1.2) {$\red{\bullet}$};
		\draw[red] (-.85,-1) to (-.85,-1.2);
		\node[scale=0.7] at (-.85,-1.2) {$\red{\bullet}$};
		\node[draw,thick,rounded corners,fill=white,minimum width=48] at (0,0) {$x$};
\end{tikzpicture}} \leq \norm{d_{Q_k}}
\raisebox{-8mm}{
	\begin{tikzpicture}
		\draw[red] (-1,.8) to (-1,-.8);
		\node[scale=0.7] at (-.6,.7) {$\red{\bullet}$};
		\node[scale=0.7] at (-.6,-.7) {$\red{\bullet}$};
		\draw (0.4,0.2) to (.4,.8);
		\draw (-0.4,0.2) to (-.4,.8);
		\draw (0.4,-0.2) to (.4,-.8);
		\draw (-0.4,-.2) to (-.4,-.8);
		\draw[dashed] (0.5,0.2) to (.5,.8);
		\draw[dashed] (0.5,-0.2) to (.5,-.8);
		\node at (0.05,0.4) {$\cdots$};
		\node at (0.05,-0.5) {$\cdots$};
		\node[right,scale=0.6] at (-.4,0.6) {$\Gamma_k$};
		\node[left,scale=0.6] at (0.5,0.6) {$\Gamma_1$};
		\node[right] at (0.4,.4) {$m_0$};
		\draw[red] (-.6,.2) to (-.6,.7);
		\draw[red] (-.6,-.2) to (-.6,-.7);
		\node[draw,thick,rounded corners,fill=white,minimum width=48] at (0,0) {$x$};
\end{tikzpicture}}
\]	
where the first equality follows from (F1) of \Cref{F1} and the second inequality from (F2) of \Cref{F1} and the definition of $ S_k $.
We rewrite the last term as
$
\raisebox{-8mm}{
	\begin{tikzpicture}
		\draw[red] (-1.8,.8) to (-1.8,-.8);
		\draw[red] (-1.6,.5) to (-1.6,-.5);
		\node[scale=0.7] at (-1.6,.5) {$\red{\bullet}$};
		\node[scale=0.7] at (-1.6,-.5) {$\red{\bullet}$};
		\node[scale=0.7] at (-.6,.7) {$\red{\bullet}$};
		\node[scale=0.7] at (-.6,-.7) {$\red{\bullet}$};
		\node[draw,thick,rounded corners,scale=0.7] at (-1.2,0) {$d_{Q_k}^{-1}$};
		\draw (0.4,0.2) to (.4,.8);
		\draw (-0.4,0.2) to (-.4,.8);
		\draw (0.4,-0.2) to (.4,-.8);
		\draw (-0.4,-.2) to (-.4,-.8);
		\draw[dashed] (0.5,0.2) to (.5,.8);
		\draw[dashed] (0.5,-0.2) to (.5,-.8);
		\node at (0.05,0.4) {$\cdots$};
		\node at (0.05,-0.5) {$\cdots$};
		\node[right,scale=0.6] at (-.4,0.6) {$\Gamma_k$};
		\node[left,scale=0.6] at (0.5,0.6) {$\Gamma_1$};
		\node[right] at (0.4,.4) {$m_0$};
		\draw[red] (-.6,.2) to (-.6,.7);
		\draw[red] (-.6,-.2) to (-.6,-.7);
		\node[draw,thick,rounded corners,fill=white,minimum width=48] at (0,0) {$x$};
\end{tikzpicture}}
\leq
\norm{d_{Q_k}}\ Q_k \left(E'(x)\right)$.
Hence, the conditional expectation $ E' $, and thereby $ E_k $	has finite index.
\end{proof}

\comments{\begin{cor}\label{HkQsystem}
	$_{A_k}{H_k}_{A_k}$ is isomorphic to $_{H_k}{H_k}_{H_k}$ as Q-systems in {\normalfont \textbf{C*Alg}}.
\end{cor}
\begin{proof}
Consider the inclusion $A_k \os{\displaystyle \widetilde i_k}\hookrightarrow H_k$.
Using the finite Pimsner-Popa basis of the right $ A_k $-correspondence $ H_k $ (see Lemma 3.3 DGGJ)  and the relation between the right $ A_k $-valued inner product and the conditional expectation $ E_k :H_k \ra A_k $ in \Cref{HkCE}, one may obtain a finite right Pimsner-Popa basis for $ E_k $.
Since $ E_k $ is  a faithful, finite index conditional expectation, by [example 3.25 CHJP], we have the result.  
\end{proof}}

Next, we will test the compatibility of the countable family of finite dimensional C*-algebras $ \left\{H_k\right\}_{k\geq 0} $ and the inclusions $ H_k \os {\displaystyle I_{k+1}}\hookrightarrow H_{k+1} $ for $ k\geq 0 $ (as described in \ref{Hkinclusion}).
\begin{lem}
The inclusion $ H_k \os {\displaystyle I_{k+1}}\hookrightarrow H_{k+1} $ is a $ * $-algebra homomorphism eventually for all $ k $.
Further, the unital filtration $ \left\{A_k\right\}_{k\geq 0} $ of finite dimensional C*-algebras (as described \ref{Akinclusion}) sits inside $ H_\infty = \us k \cup H_k $ via the inclusions $ \widetilde i_k: A_k \ra H_k $ eventually for all $ k $.
In particular, the above conditions commence when $ (m_k, m_{k+1}) $ and $ (i_k, i_{k+1}) $ start satisfying the exchange relation.
\end{lem} 
\begin{proof}
	This easily follows from the exchange relation of $m_k$ and $i_k$, and the definitions of $ \widetilde m_k $ and $ \widetilde i_k $.
\end{proof}

\begin{rem}\label{Hkbasis}
We can obtain $\mscr S_k \subset H_k$ such that $\underset{\sigma \in \mscr S_k}{\sum}\sigma \sigma^* = 1_{Q_k \Gamma_k \cdots \Gamma_1 m_0}$ using  \Cref{PPbasispreC*} and \Cref{Hkinclusion} .
\end{rem}

\section{Splitting of $(Q_\bullet, m_\bullet, i_\bullet)$}

In this section we will first construct a suitable $0$-cell in $\textbf{UC}$ using results from the previous section. Then move on to construct a dualizable $1$-cell $X_\bullet$ from $\Gamma_\bullet$ to the newly constructed $0$-cell. Subsequently we build a unitary from $\ol X_\bullet \boxtimes X_\bullet$ to $Q_\bullet$ which intertwine the algebra maps as well as satisfy exchange relations eventually.

\vspace*{2mm}

\noindent \textit{Notation:} Thoughout this section, given a finite dimensional C*-algebra $ A $, we will use the notation $ \mcal R_A $ for the category of finite-dimensional (as a complex vector space) right $A$-correpondences.
Note that $\mcal R_{A}$ is a finite, semisimple C*-category.

\comments{Using the results from previous section, we first construct a bi-faithful functor from $X_k : \mcal M_k \to \mcal R_{G_k}$ for each $k \geq 0 \ $. Then proceed to establish an isomorphism of Q-systems between $\bar{X}_k X_k$ and $Q_k$ for each $k \geq l$. In the subsequent section we will then construct a suitable $0$-cell $\left(\Delta_\bullet, \mcal R_{G_\bullet} \right)$ in $\textbf{UC}$ and a suitable unitary connection $W_\bullet$ so that $\left(X_\bullet, W_\bullet \right)$ becomes a dualizable $1$-cell in $\textbf{UC}$. In the final step we show that the isomorphism of Q-systems between $\bar{X}_k X_k$ and $Q_k$ for $k \geq l$, satisfy exchange relations for $k \geq l \ $.}

\subsection{New $0$-cells in $\textbf{UC}$}\label{2new0cells}\

Let $ l \in \N $ be as in \Cref{lremark}

For each $k \geq 0$, consider the C*-algebra inclusions $A_k \os{\displaystyle Q_k}\hookrightarrow C_k \coloneqq \t{End}(Q_k \Gamma_k \cdots \Gamma_{1} m_0) \ $ and $ C_k \ni \gamma \longmapsto \raisebox{-6 mm}{
\begin{tikzpicture}
\draw[dashed] (.5,-.7) to (.5,.7);
\draw (.3,-.7) to (.3,.7);
\draw (-.3,-.7) to (-.3,.7);
\draw (-.5,-.7) to[out=90,in=-90] (-.7,-.3) to (-.7,.3) to[out=90,in=-90] (-.5,.7);
\draw[white, line width=1mm] (-.7,-.7) to[out=90,in=-90] (-.5,-.3) to (-.5,.3) to[out=90,in=-90] (-.7,.7);
\draw[red] (-.7,-.7) to[out=90,in=-90] (-.5,-.3) to (-.5,.3) to[out=90,in=-90] (-.7,.7);
\node at (0,.5) {$ \cdots $};
\node at (0,-.5) {$ \cdots $};
\node[draw,thick,rounded corners,fill=white,minimum width=35] at (0,0) {$\gamma$};
\end{tikzpicture}} \in C_{k+1} $.
Note that $Q_k \left(A_k\right) \subset S_k \subset C_k $ for all $ k \geq l $.
Consider the filtration of C*-algebras $\left\{B_k \right\}_{k \geq 0}$ defined as follows:
\[B_k = \begin{cases}
H_k & \ \ \ \ \t{if} \ k \geq l \\
S_l \cap  C_{k} &  \ \ \ \ \t{if} \ 0 \leq k \leq l-1 
\end{cases}\]
where the inclusion $B_k \hookrightarrow B_{k+1}$ is given by $ I_k $ for $ k\geq l  $, set inclusions for $0 \leq k \leq l-2$ and the remaining inclusion $ B_{l-1} \hookrightarrow B_l $ is $\left.\phi_2^{(l)}\right|_{S_l \cap C_{l-1}} : S_l \cap C_{l-1} \to H_l$ (where $ \phi^{(l)}_2 :S_l \ra H_l$ is the isomorphism defined in \Cref{HkCstar}).


Define $\Delta_{k+1} \coloneqq \bullet \us{B_k} \boxtimes B_{k+1} : \mcal R_{B_k} \to \mcal R_{B_{k+1}}$ 
\comments{
as follows.
\[\Delta_{k+1} = \begin{cases}
 \bullet \us{H_k} \boxtimes H_{k+1} & \ \ \ \ \t{if} \ k \geq l \\
Id_{\mcal R_{G_k}} &  \ \ \ \ \t{if} \ 0 \leq k \leq l-1 
\end{cases}\]
}
for $k \geq 0$.
Each $\Delta_k$ is a bi-faithful functor (which follows from the unital inclusion $ B_k \hookrightarrow B_{k+1} $ of finite dimensional C*-algebra for $k \geq 0$).
Thus, we have a $0$-cell $\left(\Delta_\bullet, \mcal R_{B_\bullet} \right) \in \textbf{UC}_0 \ $.

\vspace*{4mm}

Similarly, using the unital filtration $ \left\{A_k\right\}_{k\geq 0} $ (resp., $ \left\{C_k\right\}_{k\geq 0} $) of finite dimensional C*-algebras, we define another $ 0 $-cell  $ \Sigma_\bullet $ (resp. $ \Psi_\bullet $)
defined by $\Sigma_k \coloneqq \bullet \us {A_{k-1}} \btimes A_k\ :\ \mcal R_{A_{k-1}} \ra \mcal R_{A_k} $ (resp., $\Psi_k \coloneqq \bullet \us {C_{k-1}} \btimes C_k\ :\ \mcal R_{C_{k-1}} \ra \mcal R_{C_k} $) for $ k\geq 1 $.

\subsection{Construction of dualizable $1$-cell from $\left(\Gamma_\bullet , \mcal M_\bullet\right)$ to $\left(\Delta_\bullet , \mcal R_{B_\bullet}\right)$} \
\comments{
Our goal in this subsection is to construct bi-faithful functors $X_k : \mcal M_k \to \mcal R_{G_k}$ for each $k \geq 0 \ $ in such a way, so that $X_\bullet$ becomes a dualizable $1$-cell from $\Gamma_\bullet$ to $\Delta_\bullet$ in $\textbf{UC}$. Then in the subsequent subsections we prove that $\ol X_k X_k$ becomes unitarily isomorphic to $Q_k$ as Q-systems for $k \geq l$ and the unitaries satisfy exchange relations eventually . 

\vspace*{2mm}}

Our strategy is to build two dualizable $1$-cells $ \left( F_\bullet , W_\bullet^F \right) : \Gamma_\bullet \to \Sigma_\bullet$ and $\left(\Lambda_{\bullet}, W_\bullet^\Lambda \right) : \Sigma_\bullet \to \Delta_\bullet$ and define $\left(X_\bullet, W_\bullet\right)$ to be their composition in $\textbf{UC}$ as depicted in \ref{tensor1cell} of \Cref{UCdefn} and thereby obtaining our desired dualizable $1$-cell $X_\bullet : \Gamma_\bullet \to \Delta_\bullet$ in $\textbf{UC}$.
We first prove the following easy fact.
\begin{prop}\label{Fequivalence}
Given a finite semisimple C*-category $ \mcal M $ and an object $ m $ which contains every simple object as a sub-object, the functor $F \coloneqq \mcal M (m ,\bullet) : \mcal M \to \mcal R_{A}$  is an equivalence where $ A = \normalfont\t{End} (m) $ and $ \mcal R_A $ is the category of right $ A $-correspondences.
\end{prop}
\begin{proof}
For $ x \in \t{ob} (\mcal M) $, $F(x)$ becomes a right $A$-correspondence with the $ A $-action and $ A $-valued inner product defined in the following way
\[
F(x) \times A \ni (u,a) \longmapsto u \, a \in F(x) \ \t{ and } \ \left\lab u, v \right  \rab \coloneqq v^* u \ .
\]
For $f \in \mcal M(x,y), F(f)(u)= f\,u \in F(y)$ for each $u \in F(x)$. Indeed, $ F $-action on any morphism of $ \mcal M $ is adjointable ($F(f)^* = F(f^*)$) and $ A $-linear.
Clearly, each $F$ is a faithful functor.
	
Let $T \in \mcal R_{A}(F(x),F(y))$.
Since every simple appears as a sub-object in $ m $, we can find $\mathscr S_x \subseteq F(x) $ such that $ \displaystyle \sum_{u \in \mathscr S_x} u \, u^* = 1_x$. 
Define $f \coloneqq \displaystyle \sum_{u \in \mathscr S_x} T(u)u^* \in \mcal M (x,y)$. 
For $v \in F(x)$,we have,
	\begin{equation*}
		\begin{split}
			T(v) &= T\left(\displaystyle \sum_{u \in \mathscr S_x} u\, u^* v \right) \\
			&= \displaystyle \sum_{u \in \mathscr S_x} T(u)\, u^* v \ \ \t{(since $T$ is right $A$-linear)} \\
			&= F(f)(v) 	\ \ \t{(by definition of $f$)}
		\end{split} 
	\end{equation*}
	Thus, $F$ is full.

	Now, we show that each $F$ is essentially surjective. Since $F$ is fully faithful by Schur's lemma, we have , $F(x)$ is simple if $x$ is simple. We show that for simple $H \in \mcal R_{A}$ there is a simple $x$ in $\mcal M$ such that $F(x)_{A} \simeq H_{A} $. Choose, $\xi \in H \setminus \left\{0 \right\} $
	such that $\langle \xi,\xi \rangle_{A} \neq 0 $. By spectral decomposition of $\langle \xi,\xi \rangle_{A}$ , there is a minimal  projection $p$ in $A$ such that $\langle \xi,\xi \rangle_{A} p \in \C p \setminus \left\{0 \right\} $. Now ,since p is minimal, $\langle \xi p,\xi p \rangle_{A}= p \langle \xi,\xi \rangle_{A} p \in \C p \setminus \left\{0 \right\}$. Without loss of generality, we assume $\xi p=\xi$. Now, $H$ being irreducible, we have, $H= \xi A$. Now, by semi-simplicity of $\mcal M$ there is a simple $x \in \mcal M$ and an isometry $\alpha : x \to m$ such that $p=\alpha {\alpha^*}$. Observe that, $\alpha^* \in F(x)$ and $F(x)= {\alpha^*}{A}$. Define $T' : {F(x)}_A \to H_{A}$ as $T'({\alpha^*}a)= \xi a$ for all $a \in A$. Clearly, $T'$ is well-defined, right-$A$ linear and onto. Thus, $T'$ is an isomorphism. Hence, $F$ is an equivalence.
\end{proof}
\subsubsection{Construction of $ \left(F_\bullet , W^F_\bullet\right) \in \t{\normalfont \textbf{UC}}_1 \left(\Gamma_\bullet , \Sigma_\bullet \right) $}\

For each $ k \geq 0 $, setting $ m = \Gamma_k \cdots \Gamma_1 m_0$ in \Cref{Fequivalence}, we obtain the functor $F_k \coloneqq \mcal M_k (\Gamma_k \cdots \Gamma_1 m_0, \bullet) : \mcal M_k \to \mcal R_{A_k}$ which is an equivalence.
\begin{rem}\label{Fkequivalence}
$ F_k $ being an equivalence is a part of an adjoint equivalence \cite{JY21}, so we may obtain an adjoint $ \ol F_k $ of $F_k$, and evaluation and coevaulation implementing the duality which are both natural unitaries.
Thus, for each $k \geq 0$, bi-faithfulness of $F_k$ is immediate .
\end{rem}
Before we describe the unitary connections for $F_k$'s, we digress a bit to prove some results which will be useful in the construction.

\vspace*{4mm}

Suppose $\mcal N$ is a C*-semisimple category. For $x , y \in \t{Ob}(\mcal N)$ , consider the morphism space $\mcal N(x,y)$ and consider the C*-algebra $A=\t{End}(x)$. Then, $\mcal N(x,y)$ becomes a right-$A$ correspondence with $A$-valued inner product, $\langle u,v \rangle_{A}= {u^*}v.$

We proceed with the following lemma.
\begin{lem}\label{Unitary}
	Suppose $\mcal M$ and $\mcal N$ are finite, C*-semisimple categories. Let $\Gamma_1 : \mcal M \to \mcal N$ and $\Gamma_2 : \mcal N \to \mcal N$ be bi-faithful, $*$-linear functors. Then the map $
	T :  \mcal N(\Gamma_1 m_0,x) \us{\normalfont\t{End}(\Gamma_1 m_0)}\boxtimes \mcal N(\Gamma_1 m_0,\Gamma_2 \Gamma_1 m_0) \to \mcal N(\Gamma_1 m_0,\Gamma_2 x) $ given by $u \boxtimes v \os{T}\longmapsto 
	\raisebox{-16mm}{
		\begin{tikzpicture}
		\draw (-.5,.2) to (-.5,1.8);
		\draw (0,.2) to (0,1);
		\draw[dashed] (0.25,.2) to (0.25,.5);
		\draw[ densely dotted] (.08,.89) to (.08,1.8);
		\draw (0,-.2) to (0,-.8);
		\draw[dashed] (0.25,1) to (0.25,-.8);
		\node[draw,thick,rounded corners,fill=white,minimum width=30] at (-.1,0) {$v$};
		\node[draw,thick,rounded corners,fill=white,minimum width=10] at (0.1,1) {$u$};
		\node[scale=.8] at (0,-1) {$\Gamma_{1}$};
		\node[scale=.8] at (-.2,.5) {$\Gamma_{1}$};
		\node[scale=.8] at (-.8,1.2) {$\Gamma_2$};
		\node at (.6,-.5) {$m_0$};
		\node at (.3,1.5) {$x$}; 
		\end{tikzpicture}}$ is a unitary as a right- $\normalfont \t{End}(\Gamma_1 m_0)$-linear map.   
\end{lem}
\begin{proof}
	Let $A= \t{End}(\Gamma_1 m_0)$. Clearly, $T$ is middle A-linear. Now,
	\[\langle T(u_1 \boxtimes v_1) , T(u_2 \boxtimes v_2) \rangle_{A} =
	\raisebox{-12mm}{
		\begin{tikzpicture}
		\draw (-.4,.25) to (-.4,1.3);
		\draw (0,.2) to (0,.45);
		\draw[dashed] (0.25,.2) to (0.25,.45);
		\draw (0,1) to (0,1.3);
		\draw (0,-.2) to (0,-.5);
		\draw (0,1.9) to (0,2.2);
		\draw[dashed] (0.25,-.2) to (0.25,-.5);
		\draw[dashed] (0.25,1) to (0.25,1.4);
		\draw[dashed] (0.25,1.9) to (0.25,2.2); 
		\node[draw,thick,rounded corners,minimum width=40,fill=white] at (0,1.6) {$v_1^{*}$};
		\node[draw,thick,rounded corners,minimum width=40,fill=white] at (0,0) {$v_2$};
		\node[draw,thick,rounded corners,fill=white] at (0.2,0.74) {$u_1 ^{*} u_2$};
		\end{tikzpicture}}
	= \langle v_1 , \langle u_1,u_2 \rangle_{A} v_2 \rangle_{A}= \langle u_1 \boxtimes v_1, u_2 \boxtimes v_2 \rangle_{A} . 
	\]
	Hence, $T$ is an isometry. If we can show that $T$ is surjective then we get our desired result from [Lance]. Now, let $y \in \mcal N(\Gamma_1 m_0, \Gamma_2 x)$. Then, 
	\raisebox{-6mm}{
		\begin{tikzpicture}
		\draw (-.1,.24) to (-.1,.7);
		\draw[densely dotted] (0.1,.24) to(.1,.7);
		\draw (-.1,-.24) to (-.1,-.7);
		\draw[dashed] (.1,-.24) to (.1,-.7);
		\node[right] at (.06,.5) {$x$};
		\node[left,scale=0.8] at (-.12,.6) {$\Gamma_2$};
		\node[left,scale=0.8] at (-.12,-.6) {$\Gamma_1$};
		\node at (.5,-.6) {$m_0$};
		\node[draw,thick,rounded corners,fill=white,minimum width=20] at (0,0) {$y$};
		\end{tikzpicture}}
	= 
	\raisebox{-4mm}{
		\begin{tikzpicture}
		\draw (-.1,.23) to (-.1,.2);
		\draw[densely dotted] (0.1,.23) to(.1,.7);
		\draw (-.1,-.24) to (-.1,-.7);
		\draw[dashed] (.1,-.24) to (.1,-.7);
		\node[right] at (.06,.4) {$x$};
		\draw[in=90,out=90,looseness=2] (-.1,.2) to (-.5,.2);
		\draw (-.5,.2) to (-.5,-0.25);
		\draw[in=-90,out=-90,looseness=2] (-.5,-.25) to (-.9,-.25);
		\draw (-.9,-.25) to (-.9,.25);
		\node[draw,thick,rounded corners,fill=white,minimum width=20] at (0,0) {$y$};
		\end{tikzpicture}}= $\displaystyle \sum_{\alpha \in \mathscr S}$
	\raisebox{-12mm}{
		\begin{tikzpicture}
		\draw (-.1,.23) to (-.1,.2);
		\draw[densely dotted] (0.1,.23) to(.1,.7);
		\draw (-.1,-.24) to (-.1,-.5);
		\draw[dashed] (.1,-.24) to (.1,-.5);
		\node[right] at (.06,.4) {$x$};
		\draw[in=90,out=90,looseness=2] (-.1,.2) to (-.5,.2);
		\draw (-.5,.2) to (-.5,-0.5);
		\draw (-.1,-.9) to (-.1,-1.2);
		\draw[dashed] (.1,-.9) to (.1,-1.2);
		\draw[dashed] (0.1,-1.6) to(.1,-2);
		\draw (-.1,-1.6) to (-.1,-1.9);
		\draw (-.4,-1.6) to (-.4,-1.7);
		\draw[in=-90,out=-90,looseness=2] (-.4,-1.7) to (-.8,-1.7);
		\draw (-.8,-1.7) to (-.8,.7);
		\node[draw,thick,rounded corners,fill=white,minimum width=20] at (0,0) {$y$};
		\node[draw,thick,rounded corners,fill=white,minimum width=30] at (-.1,-.7) {$\alpha$};
		\node[draw,thick,rounded corners,minimum width=30,scale=0.8,fill=white] at (-.1,-1.4) {$\alpha^{*}$};
		\end{tikzpicture}}.
	Last equality follows from the fact that, we can find such a set $\mathscr S \subseteq \t{End}\left( \ol \Gamma_2 \Gamma_1 m_0 \right)$ because of bi-faithfulness of $\Gamma_2$, $\Gamma_{1}$ and $m_0$ contains all irreducibles of $\mcal M$. Now, $T \left( \displaystyle \sum_{\alpha \in \mathscr S}
	\raisebox{-6mm}{
		\begin{tikzpicture}
		\draw (-.1,.23) to (-.1,.2);
		\draw[densely dotted] (0.1,.23) to(.1,.7);
		\draw (-.1,-.24) to (-.1,-.5);
		\draw[dashed] (.1,-.24) to (.1,-.5);
		\node[right] at (.06,.5) {$x$};
		\draw[in=90,out=90,looseness=2] (-.1,.2) to (-.5,.2);
		\draw (-.5,.2) to (-.5,-0.5);
		\draw (-.1,-.9) to (-.1,-1.2);
		\draw[dashed] (.1,-.9) to (.1,-1.2);
		\node[draw,thick,rounded corners,minimum width=20,fill=white] at (0,0) {$y$};
		\node[draw,thick,rounded corners,minimum width=30,fill=white] at (-.1,-.7) {$\alpha$};
		\end{tikzpicture}} \boxtimes
	\raisebox{-6mm}{
		\begin{tikzpicture}
		\draw[dashed] (0.1,-1.6) to(.1,-2);
		\draw (-.1,-1.6) to (-.1,-1.9);
		\draw (-.4,-1.6) to (-.4,-1.7);
		\draw[in=-90,out=-90,looseness=2] (-.4,-1.7) to (-.8,-1.7);
		\draw[dashed] (0.1,-1.2) to(.1,-.7);
		\draw (-.1,-1.2) to (-.1,-.7);
		\draw (-.8,-1.7) to (-.8,-.7);
		\node[draw,thick,rounded corners,minimum width=30,scale=0.8,fill=white] at (-.1,-1.4) {$\alpha^{*}$};
		\end{tikzpicture}}\right)
	=y$. Hence, $T$ is surjective. So, $T$ is an unitary.	    
\end{proof}


\begin{cor}\label{uninat}
	The maps  $T_x^k :F_k(x)\us{A_k} \boxtimes {A_{k+1}}_{A_{k+1}} \to \mcal M_k (\Gamma_{k+1} \Gamma_{k} \cdots \Gamma_1 m_0,\Gamma_{k+1} x)_{A_{k+1}}$ given by $u \boxtimes \alpha \os{T^k_x}\longmapsto 
	\raisebox{-12mm}{
		\begin{tikzpicture}
		\draw[dashed] (.45,-.7) to (.45,.7);
		\draw (.35,-.7) to (.35,.7);
		\draw (-.4,-.7) to (-.4,1.6);
		\draw (-.18,-.7) to (-.18,.7);
		\node at (0.1,.45) {$\cdots$};
		\node at (0.1,-.5) {$\cdots$};
		\draw[densely dotted] (.15,1.1) to (.15,1.6);
		\node[right] at (.12,1.4) {$x$};
		\node[scale=.8] at (-.18,-.9) {$\Gamma_{k}$};
		\node[scale=.8] at (.35,-.9) {$\Gamma_{1}$};
		\node at (.8,-.5) {$m_0$};
		\node[left,scale=0.8] at (-.4,.8) {$\Gamma_{k+1}$};
		\node[draw,thick,rounded corners,fill=white,minimum width=24] at (0.15,.9) {$u$};
		\node[draw,thick,rounded corners,fill=white,minimum width=35] at (0,0) {$\alpha$};
		\end{tikzpicture}}$ are unitaries and they are natural in x, for each $x \in \mcal M_k$ and $k \geq 0$.
\end{cor}
\begin{proof}
	Clearly, $T_x^k$ are right-$A_{k+1}$ linear.
	Unitarity of $T_x^k$ follows from \Cref{Unitary}. Naturality of $T^k$ follows from the definition of $F_k$ acting on morphism spaces as in \Cref{Fequivalence}.
\end{proof}

We now define the unitary connections for $\left\{F_k\right\}_{k\geq 0}$ as $W_{k+1}^F \coloneqq T^k : \Sigma_{k+1} F_k \to F_{k+1} \Gamma_{k+1}$ as defined in \Cref{uninat}, for each $k \geq 0 \ $.
Pictorially we denote, for each $k \geq 0$, $F_k$ by \raisebox{-4mm}{\begin{tikzpicture}
	\draw[violet,->] (0,0) to (0,1);
	\end{tikzpicture}} and $\ol F_k$ by \raisebox{-4mm}{\begin{tikzpicture}
	\draw[violet,<-] (0,0) to (0,1);
	\end{tikzpicture}} and for each $k \geq 1$, $W_k^F$ by
\raisebox{-4mm}{
	\begin{tikzpicture}
	\draw[white,line width=1mm,out=-90,in=90] (-.25,.5) to (.25,-.5);
	\draw[out=-90,in=90,violet,<-] (-.25,.5) to (.25,-.5);
	\begin{scope}[on background layer]
	\draw[out=90,in=-90] (-.25,-.5) to (.25,.5);
	\end{scope}
	\node[right] at (0.15,-0.4) {$F_{k-1}$};
	\node[right] at (0.15,0.3) {$\Gamma_{k}$};
	\node[left] at (-0.15,-0.4) {$\Sigma_{k}$};
	\node[left] at (-0.15,0.3) {$F_{k}$};
	\end{tikzpicture}}
and $\left(W_k^F\right)^*$ by
\raisebox{-4mm}{
	\begin{tikzpicture}
	\draw[white,line width=1mm,out=-90,in=90] (.25,.5) to (-.25,-.5);
	\draw[out=-90,in=90,violet,<-] (.25,.5) to (-.25,-.5);
	\begin{scope}[on background layer]
	\draw[out=90,in=-90] (.25,-.5) to (-.25,.5);
	\end{scope}
	\node[right] at (0.15,-0.4) {$\Gamma_{k}$};
	\node[right] at (0.15,0.3) {$F_{k-1}$};
	\node[left] at (-0.15,-0.4) {$F_{k}$};
	\node[left] at (-0.15,0.3) {$\Sigma_{k}$};
	\end{tikzpicture}} .
Hence, we have a $1$-cell $\left(F_\bullet, W_\bullet^F \right) \in \textbf{UC}_1 (\Gamma_\bullet, \Sigma_\bullet)$.

For each $k \geq 1$, define
\[ \ol W_k^F \coloneqq \raisebox{-6mm}{
	\begin{tikzpicture}
	\draw[white,line width=1mm,out=-90,in=90] (-.25,.5) to (.25,-.5);
	\draw[out=-90,in=90,violet,->] (-.25,.5) to (.25,-.5);
	\begin{scope}[on background layer]
	\draw[out=90,in=-90] (-.25,-.5) to (.25,.5);
	\end{scope}
	\node[right] at (0.2,-0.4) {$\ol F_{k-1}$};
	\node[right] at (0.15,0.3) {$\Sigma_{k}$};
	\node[left] at (-0.15,-0.4) {$\Gamma_{k}$};
	\node[left] at (-0.15,0.3) {$\ol F_{k}$};
	\end{tikzpicture}}
\coloneqq
\raisebox{-10mm}{
	\begin{tikzpicture}
	\draw[white,line width=1mm,out=-90,in=90] (.25,.5) to (-.25,-.5);
	\draw[out=-90,in=90,violet,<-] (.25,.5) to (-.25,-.5);
	\begin{scope}[on background layer]
	\draw[out=90,in=-90] (.25,-.5) to (-.25,.5);
	\draw (-.25,.5) to (-.25,.8);
	\draw (.25,-.5) to (.25,-.8);
	\end{scope}
	\node[right,scale=0.7] at (0,-1) {$\Gamma_{k}$};
	\node[left,scale=0.7] at (0,1) {$\Sigma_{k}$};
	\draw[violet,in=90,out=90,looseness=2,->] (.25,.5) to (.75,.5);
	\draw[violet,->] (.75,.5) to (.75,-.8);
	\draw[violet,in=-90,out=-90,looseness=2,<-] (-.25,-.5) to (-.75,-.5);
	\draw[violet,<-] (-.75,-.5) to (-.75,.8);
	\end{tikzpicture}} \	\t{and} \ \left(\ol W_k^F \right)^*  \coloneqq 
\raisebox{-8mm}{
	\begin{tikzpicture}
	\draw[out=-90,in=90](-.25,.5) to (.25,-.5);
	\draw[white,line width=1mm,out=-90,in=90] (-.25,-.5) to (.25,.5);
	\draw[violet,out=90,in=-90,<-] (-.25,-.5) to (.25,.5);
	\node[right] at (0.2,-0.4) {$\Sigma_{k}$};
	\node[right] at (0.15,0.3) {$\ol F_{k-1}$};
	\node[left] at (-0.15,-0.4) {$\ol F_{k}$};
	\node[left] at (-0.15,0.3) {$\Gamma_{k}$};
	\end{tikzpicture}} 
\coloneqq
\raisebox{-8mm}{
	\begin{tikzpicture}
	\draw[white,line width=1mm,out=-90,in=90] (-.25,.5) to (.25,-.5);
	\draw[out=-90,in=90,violet,<-] (-.25,.5) to (.25,-.5);
	\draw[violet,in=90,out=90,looseness=2,->] (-.25,.5) to (-.75,.5);
	\draw[violet] (-.75,.5) to (-.75,-.8);
	\draw[violet,in=-90,out=-90,looseness=2,<-] (.25,-.5) to (.75,-.5);
	\draw[violet,<-] (.75,-.5) to (.75,.8);
	\begin{scope}[on background layer]
	\draw[out=90,in=-90] (-.25,-.5) to (.25,.5);
	\end{scope}
	\draw (.25,.5) to (.25,.8);
	\draw (-.25,-.5) to (-.25,-.8);
	\node[right,scale=0.8] at (0.16,0.9) {$\Gamma_{k}$};
	\node[left,scale=0.8] at (-0.15,-0.9) {$\Sigma_{k}$};
	\end{tikzpicture}}\] 
Since the evaluation and coevaluation are chosen (in \Cref{Fkequivalence}) to be  unitaries, therefore $ \ol W^F_k $'s are also so.
We claim that $F_\bullet$ is a dualizable $1$-cell in $\textbf{UC}$ with dual $\left(\ol F_\bullet, \ol  W_\bullet^F \right)$.
For this,  we verify that solutions to conjugate equations (as in \Cref{Fkequivalence}) satisfy exchange relations for $k \geq 0$, which is equivalent to the equations by which $W_k^F$'s and $ \ol W^F_k $'s become unitaries.

\begin{rem}\label{CkMkequivalence}
Observe that by \Cref{Fequivalence}, we have an adjoint equivalence $ G_k : \mcal M_k \to \mcal R_{C_k} $ using the fact that $ Q_k \Gamma_k \cdots \Gamma_1 m_0 $ contains every simple of $ \mcal M_k $ as a subobject for each $k \geq 0$.
Further, the square \begin{tikzcd}
	\mcal M_{k+1} \arrow[r,"G_{k+1}"] & \mcal R_{C_{k+1}}  \\
	\mcal M_k \arrow[u,"\Gamma_k"] \arrow[swap,r,"G_k"] & \mcal R_{C_k} \arrow[swap,u,"\bullet \us {C_k} \btimes C_{k+1}"] 
\end{tikzcd}
commutes up to a natural unitary, say $ W^G_{k+1} $, which can be proven exactly the same was done for $ F_k $'s, and thereby yeilding a dualizable $ 1 $-cell $ \left(G_\bullet , W^G_\bullet \right) $ in \textbf{UC} from $ \Gamma_\bullet $ to $ \Psi_\bullet $.

Picking a dual  $ \ol G_\bullet \in \t{\textbf{UC}}_1 \left( \Psi_\bullet , \Gamma_\bullet \right) $ of $ G_\bullet $, we set $ \left(R_\bullet, W^R_\bullet \right)\coloneqq F_\bullet \btimes \ol G_\bullet \in \t{\textbf{UC}}_1 \left( \Psi_\bullet , \Sigma_\bullet \right)$.
That is, $ R_k = F_k \ol G_k : \mcal R_{C_k} \to \mcal R_{A_k} $ for $ k\geq 0 $ which along with the unitary connections are compatible with the $ \Sigma_k $'s and $ \Psi_k $'s.
\comments{Observe that by \Cref{Fequivalence}, the category $\mcal R_{C_k}$, just like $\mcal R_{A_k}$, is equivalent to $\mcal M_k$ for each $k \geq 0$ since $ Q_k \Gamma_k \cdots \Gamma_1 m_0 $ contains every simple of $ \mcal M_k $ as a subobject.
So, we have an equivalence $R_k : \mcal R_{C_k} \to \mcal R_{A_k} \ $ for each $k \geq 0 \ $.
Further, note that these $ R_k $'s (which factors through $\mcal M_k $'s via \Cref{Fequivalence}) satisfy the condition that the square of functors
\begin{tikzcd}[row sep=small]
&\mcal R_{C_{k+1}}
\arrow[drr,bend left,"R_{k+1}"']
\arrow[dr, dashed] & & \\
&& \mcal M_{k+1} \arrow[r, "F_{k+1}"] 
& \mcal R_{A_{k+1}}  \\
&& \mcal M_k \arrow[u, "\Gamma_{k+1}"] \arrow[r, "F_k"]
& \mcal R_{A_{k}} \arrow[u, "\Sigma_{k+1}"']\\
&\mcal R_{C_{k}} \arrow[uuu,"\bullet \us{C_k}\boxtimes C_{k+1}"] \arrow [ur,dashed] \arrow[urr,bend right,"R_k"]  & &
\end{tikzcd} 
\begin{tikzcd}
	\mcal R_{C_{k+1}} \arrow[r,"R_{k+1}"] & \mcal R_{A_{k+1}}  \\
	\mcal R_{C_k} \arrow[u,"\bullet \us {C_k} \btimes C_{k+1}"] \arrow[swap,r,"R_k"] & \mcal R_{A_k} \arrow[swap,u,"\Sigma_{k+1}"] 
\end{tikzcd}
 is commutative up to a natural unitary (where the dashed lines are the equivalences between $\mcal R_{C_k}$ and $\mcal M_k$).
}
\end{rem}
\subsubsection{Construction of $\left(\Lambda_\bullet, W_\bullet^\Lambda \right) \in \t{\normalfont \textbf{UC}}_1 \left(\Sigma_\bullet , \Delta_\bullet\right)$ and its dual}\

Observe that in \Cref{2new0cells},  for each $ k\geq 0 $, we have unital inclusions $ A_k \hookrightarrow  B_k$ of C*-algebras; in particular, for $ k \geq l $, this is given in \Cref{HkCE}.
As a result, the functor $\Lambda_k \coloneqq \bullet \us {A_k} \btimes {B_k} : \mcal R_{A_k} \ra \mcal R_{B_k}$ turns out to bi-faithful for each $k \geq 0$.
Next, we need to define the unitary connection for $ \Lambda_\bullet $.
We acheive this using the following easy fact.
\comments{Before we desribe the unitary connections for $k \geq l$}
\begin{fact}\label{lambdakuni}
Suppose $A,B,C,D$ are finite dimensional C*-algebras such that we have a square of unital inclusions \begin{tikzcd}
	C \arrow[r,hook] & D  \\
	A \arrow[u,hook] \arrow[r,hook] & B \arrow[u,hook] 
\end{tikzcd}. This induces a square of categories and functors \begin{tikzcd}
	\mcal R_C \arrow[r," \bullet \us{C}\boxtimes D"] & \mcal R_D  \\
	\mcal R_A \arrow[u,"\bullet \us{A}\boxtimes C"] \arrow[swap,r,"\bullet \us{A}\boxtimes B"] & \mcal R_B \arrow[swap,u,"\bullet \us{B}\boxtimes D"] 
\end{tikzcd}.
Corresponding to this last square, there exists a unitary natural transformation between the functors $\bullet \us{A}\boxtimes B \us{B}\boxtimes D$ and $\bullet \us{A}\boxtimes C \us{C}\boxtimes D$.
\end{fact}
For $0 \leq k \leq l-1 \ $, the unitaries $W_{k+1}^\Lambda$ may be obtained by applying \Cref{lambdakuni} to the squares \begin{tikzcd}
 \mcal R_{A_{k+1}} \arrow[r,"\Lambda_{k+1}"] & \mcal R_{B_{k+1}}  \\
 \mcal R_{A_k} \arrow[u,"\Sigma_{k+1}"] \arrow[swap,r,"\Lambda_k"] & \mcal R_{B_k} \arrow[swap,u,"\Delta_{k+1}"] 
 \end{tikzcd}. 
 
We now explicitly describe the unitaries $W_k^\Lambda : \Delta_k \Lambda_{k-1} \to \Lambda_k \Sigma_k$ for each $k \geq l+1 \ $.
For  $V \in \t{Ob}\left(\mcal R_{A_{k-1}}\right)$, define $\left( W_k^\Lambda \right)_V : V \us{A_{k-1}}\boxtimes H_{k-1} \us{H_{k-1}}\boxtimes {H_k}_{H_k} \to   V \us{A_{k-1}}\boxtimes A_k \us{A_k}\boxtimes {H_k}_{H_k} $ as follows :
\begin{equation*}\label{lambdaunipageref}
V \us{A_{k-1}}\boxtimes H_{k-1} \us{H_{k-1}}\boxtimes {H_k} \ni q \us{A_{k-1}}\boxtimes \xi_1 \us{H_{k-1}}\boxtimes \xi_2 \os{\left( W_k^\Lambda \right)_V} \longmapsto q \us{A_{k-1}}\boxtimes 1_{A_k} \us{A_k}\boxtimes \xi_1 \cdot \xi_2 \ \ \t{for each} \ \ q \in V \ .
\end{equation*}
It is easy to see that each $\left( W_k^\Lambda \right)_V$ is a unitary and natural in $V$, and
 	$\left( W_k^\Lambda \right)_V^*$
is given as follows:
 \[V \us{A_{k-1}}\boxtimes A_k \us{A_k}\boxtimes {H_k} \ni q \us{A_{k-1}}\boxtimes \alpha \us{A_k}\boxtimes \xi \os{\left( W_k^\Lambda \right)_V^*} \longmapsto q \us{A_{k-1}}\boxtimes 1_{H_{k-1}} \us{A_k}\boxtimes \xi_1 \cdot \xi_2 \ \ \t{for each} \ \ q \in V\ . \]
Thus, we get a $1$-cell $\left(\Lambda_\bullet , W_\bullet \right) : \Sigma_\bullet \to \Delta_\bullet$ in $\textbf{UC} \ $. 

\vspace*{4mm}

We now define $\left(\ol \Lambda_{\bullet} , \ol W_\bullet^\Lambda \right) \in \textbf{UC}_1 \left(\Delta_\bullet, \Sigma_\bullet \right) $ so that it becomes dual to $\left(\Lambda_{\bullet}, W_\bullet \right)$ in $\textbf{UC}$.\\
For $0 \leq k \leq l-1$, define $\ol \Lambda_k \coloneqq R_k \circ \left(\bullet \us{B_k}\boxtimes C_{k}\right) : \mcal R_{B_k} \to \mcal R_{A_k} \ $ where $R_k : \mcal R_{C_k} \ra \mcal R_{A_k}$ is the equivalence given in \Cref{CkMkequivalence}.\\
For $ k \geq l $, define $ \ol \Lambda_k \coloneqq \bullet \us {H_k}\btimes H_k : \mcal R_{H_k} \ra \mcal R_{A_k} $.
Here the right action of $ A_k $ on $ H_k $ is given by the inclusion $ A_k \hookrightarrow H_k $ (as in \Cref{HkCE}) and the multiplication in C*-algebra $ H_k $; however, the right $ A_k $-valued inner product is the one defined in \ref{Hkinp} (and not the one coming from conditional expectation).
\comments{\red{Delete this paragraph!} For each $k \geq 0$, define $\ol \Lambda_k : \mcal R_{G_k} \to \mcal R_{A_k}$ as follows:
\[\ol \Lambda_k = \begin{cases}
 \bullet \us{H_k}\boxtimes {H_k}_{A_k} & \ \ \ \ \t{if} \ k \geq l \\
 \Lambda'_k &  \ \ \ \ \t{if} \ 0 \leq k \leq l-1 
 \end{cases}\]
} 
\comments{Define $\ol \Lambda_k \coloneqq  \bullet \us{H_k}\boxtimes {H_k}_{A_k} : \mcal R_{H_k} \to \mcal R_{A_k}$ for each $k \geq l \ $.}
\begin{rem}
Although the functors $\ol \Lambda_k$ may not be adjoint to $\Lambda_k$ for $0 \leq k \leq l-1 $, we will need these functors
to define an adjoint of $\left(\Lambda_\bullet, W_\bullet^\Lambda \right)$ in $\textbf{UC} \ $.
\end{rem}
Our next job is to define the unitary connections $ \left\{\ol W^\Lambda_k\right\}_{k\geq 1} $ for $ \ol \Lambda_\bullet $.
This will be divide into three different ranges for $ k $, namely $ \left\{1, \ldots , l-1\right\}$, $\{l\}$ and $\{l+1, l+2, \ldots\} $; the choice of the natural unitaries in the first two ranges could be arbitrary.

\vspace*{4mm}

\noindent{\textbf{Case $0 \leq k \leq l-2 $}} :
For the unitary connection $\ol W_{k+1}^\Lambda$, we look at the following horizontally stacked squares of functors.
\[
\begin{tikzcd}
	\mcal R_{B_{k+1}} \arrow[r,"\bullet \us {B_{k+1} }\btimes C_{k+1}"] & \mcal R_{C_{k+1}} \arrow[r,"R_{k+1}"] &\mcal R_{A_{k+1}}  \\
	\mcal R_{B_k} \arrow[u,"\Delta_{k+1}"] \arrow[swap,r,"\bullet \us {B_{k} }\btimes C_{k}"] & \mcal R_{C_k} \arrow[swap,u,"\bullet \us {C_k} \btimes C_{k+1}"]  \arrow[swap,r,"R_k"] & \mcal R_{A_k} \arrow[swap,u, "\Sigma_{k+1} "]
\end{tikzcd} \ .
\]
Both the squares are commutative up to natural unitaries; the left one follows from \Cref{lambdakuni} and the right comes from \Cref{CkMkequivalence}.
$\ol W_{k+1}^\Lambda$ is defined as the appropriate composition of above two natural unitaries.

\vspace*{4mm}

\noindent {\textbf{Case $k=l$}}: To define the natural unitary $ \ol W^\Lambda_l : \Sigma_l \ \ol \Lambda_{l-1} \to \ol \Lambda_l \ \Delta_l  $, it is enough to check whether the square
\begin{tikzcd}[row sep = huge , column sep = huge]
\mcal R_{H_{l}} \arrow[r,"\bullet \us{H_l}\boxtimes H_l" , "\ol \Lambda_l"'] & \hspace*{2mm} \mcal R_{A_l}  \\
\mcal R_{S_l \cap C_{l-1}} \arrow[u,"\bullet \us{B_{l-1}}\boxtimes B_l","\Delta_l"'] \arrow[r,"\ol \Lambda_{l-1}", " R_{l-1} \circ \left(\bullet \us{B_{l-1}}\boxtimes C_{l-1}\right)"'] & \hspace*{2mm}\mcal R_{A_{l-1}} \arrow[u,"\bullet \us{A_{l-1}}\boxtimes A_l"',"\Sigma_l"] 
\end{tikzcd}
commutes up to a natural isomorphism; let us call this square $ \mathbb S $.
Consider the horizontal pair of sqaures
\begin{tikzcd}[column sep=huge]
	\mcal R_{S_l} \arrow[r,"\bullet \us {S_l}\btimes C_l"] & \mcal R_{C_l} \arrow[r,"R_l"] &\mcal R_{A_l}  \\
	\mcal R_{S_l \cap C_{l-1}} \arrow[u,"\bullet \us {S_l \cap C_{l-1}} \btimes S_l"] \arrow[swap,r,"\bullet \us {S_l \cap C_{l-1}}\btimes C_{l-1}"] & \mcal R_{C_{l-1}} \arrow[swap,u,"\Psi_l"]  \arrow[swap,r,"R_{l-1}"] & \mcal R_{A_{l-1}} \arrow[swap,u, "\Sigma_l "]
\end{tikzcd} referred as $ \mathbb S_1 $
; the first square of $ \mathbb S_1 $ commutes by \Cref{lambdakuni} and the second follows from \Cref{CkMkequivalence}.
Note that the bottom and the right sides of $ \mathbb S $ matches with that of $ \mathbb S_1 $.

We next claim that the top side of $ \mathbb S_1 $ is naturally isomorphic to $ \bullet \us {S_l} \btimes S_l : \mcal R_{S_l} \to \mcal R_{A_l} $.
To see this, consider the square 
\begin{tikzcd}
	\mcal R_{C_{l}} \arrow[d,"\bullet \us {C_l} \btimes C_l"'] &  \mcal M_l \arrow[l,"G_l"'] \arrow[d,"F_l"] \\
	\mcal R_{S_l}  & \mcal R_{A_l} \arrow[l,"\bullet \us {A_l} \btimes S_l"] 
\end{tikzcd} referred as $ \mathbb S_2 $.
For $ x \in \t{Ob} \left(\mcal M_l\right)  $, the map
\[
F_l (x) \us {A_l} \btimes S_l \ni \xi \us {A_l} \btimes \gamma \longmapsto 
\raisebox{-11 mm}{
\begin{tikzpicture}
\draw[densely dotted] (.1,1.1) to (.1,2);
\draw (-.3,-.7) to (-.3,1.1);
\draw (.3,-.7) to (.3,1.1);
\draw[dashed] (.5,-.7) to (.5,1.1);
\draw[red] (-.5,-.7) to (-.5,.5);
\node at (0,.5) {$ \cdots $};
\node at (0,-.5) {$ \cdots $};
\node at (-.5,.5) {\red{$ \bullet $}};
\node[scale=.8] at (-.75,-.5) {$Q_l$};
\node at (.9,-.5) {$m_0$};
\node at (.9,.5) {$m_0$};
\node at (.3,1.7) {$x$};
\node[draw,thick, rounded corners,fill=white,minimum width=40] at (0,0) {$ \gamma $};
\node[draw,thick, rounded corners,fill=white,minimum width=35] at (.1,1.1) {$ \xi $};
\end{tikzpicture}
}
\us {C_l} \btimes 1_{C_l} \in G_l (x) \us {C_l} \btimes C_l
\]
is $ S_l $-linear and natural in $ x $.
To show that the map is surjective, pick a basic tensor $ \zeta \us {C_l} \btimes 1_{C_l} \in G_l (x) \us {C_l} \btimes C_l $; note that it can be expressed as the image of $ \displaystyle \sum_{\sigma \in \mscr S_l} \zeta \circ \sigma \us {A_l} \btimes \phi^{(l)}_1 \left(\sigma^\dagger \right)$ where $\mscr S_l$ is as in \Cref{Hkbasis} and $\phi_1^{(l)} : H_l \to S_l$ is the isomorphism mentioned in \Cref{HkCstar}.
This concludes natural commutativity of $\mathbb S_2$.
Now, the adjoint of the functors  $ \bullet \us {C_l} \btimes C_l : \mcal R_{C_l} \to \mcal R_{S_l} $ and $ \bullet \us {A_l} \btimes S_l : \mcal R_{A_l} \to \mcal R_{S_l} $ (appearing in the square $ \mathbb S_1 $) are given by $ \bullet \us {S_l} \btimes C_l : \mcal R_{S_l} \to \mcal R_{C_l} $ and $ \bullet \us {S_l} \btimes A_l : \mcal R_{S_l} \to \mcal R_{A_l} $ respectively; this can be achieved by solving the conjugate equations using the set $ \mscr S_l  $ again and the conditional expectations.
Thus, dualizing the square $ \mathbb S_2 $ and using the fact that $ F_l $ is an adjoint equivalence, we get $ R_l \circ \left(\bullet \us {S_l} \btimes {C_l}_{C_l}\right) \cong \left(\bullet \us {S_l} \btimes  {S_l}_{A_l}\right) $.
Using this natural isomorphism and natural commutativity of the square $\mathbb S_1$, we obtain natural commutativity of
\begin{tikzcd}[column sep=huge, row sep=large]
\mcal R_{S_l} \hspace*{2mm} \arrow[r,"\bullet \us {S_l}\btimes S_l"] & \hspace*{2mm}\mcal R_{A_l}  \\
\mcal R_{S_l \cap C_{l-1}} \hspace*{2mm} \arrow[u,"\bullet \us {S_l \cap C_{l-1}} \btimes S_l"] \arrow[swap,r,"R_{l-1} \circ \left(\bullet \us {S_l \cap C_{l-1}}\btimes C_{l-1}\right)"] & \hspace*{2mm} \mcal R_{A_{l-1}} \arrow[swap,u,"\Sigma_l"] 
\end{tikzcd}.
Finally, using the isomorphism $\phi_2^{(l)} : S_l \to H_l $ (as in \Cref{HkCstar}), we get our desired natural commutativity of $\mathbb S$.
Set $\ol W_l^\Lambda$ to be a natural unitary implementing commutativity of $\mathbb S$.  

\vspace*{4mm}

\noindent {\textbf{Case $k \geq l$}}:
To define $\ol W_{k+1}^\Lambda$, we will need the solutions to conjugate equations for $\Lambda_k$ and $\ol \Lambda_k$ for each $k \geq l $.
We will use the following pictorial notations:
\[ \Lambda_k \coloneqq \raisebox{-4mm}{\begin{tikzpicture}
		\draw[orange,->] (0,0) to (0,1);
\end{tikzpicture}} \ \ \ \ \t{and} \ \ \ \ \ol \Lambda_k \coloneqq \raisebox{-4mm}{\begin{tikzpicture}
		\draw[orange,<-] (0,0) to (0,1);
\end{tikzpicture}} \ \ \ \ \t{for each} \ \ k \geq 0  \]

\begin{defn}\label{lambdaksol}
\

\begin{itemize}	
\item [(i)] $\raisebox{-4mm}{\begin{tikzpicture}
\draw[orange,in=-90,out=-90,looseness=2,<-] (0,0) to (1,0);
\end{tikzpicture}} : \t{Id}_{\mcal R_{H_k}} \to \Lambda_k \ol \Lambda_k $ is the natural transformation defined as: \\
$\raisebox{-4mm}{\begin{tikzpicture}
\draw[orange,in=-90,out=-90,looseness=2,<-] (0,0) to (1,0);
\draw[densely dotted] (1.2,0) to (1.2,-.7);
\node[scale=.8] at (1.35,-.5) {$V$};
\end{tikzpicture}} : V \to V \us{H_k}\boxtimes H_k \us{A_k}\boxtimes H_k$ is given by $ q \longmapsto \us{\sigma \in \mscr S_k}{\sum} q \us{H_k}\boxtimes \sigma \us{A_k}\btimes \sigma^\dagger \ $ where $V \in \mcal R_{H_k}$ and $\mscr S_k$ is as in \Cref{Hkbasis}.

\vspace*{2mm}

\item[(ii)]
	$\raisebox{-2mm}{\begin{tikzpicture}
	\draw[orange,in=90,out=90,looseness=2,->] (0,0) to (1,0);
	\end{tikzpicture}} : \Lambda_k \ol \Lambda_k \to \t{Id}_{\mcal R_{H_k}}  $ is the natural transformation defined as : \\ 
	$\raisebox{-3mm}{\begin{tikzpicture}
	\draw[orange,in=90,out=90,looseness=2,->] (0,-.7) to (1,-.7);
	\draw[densely dotted] (1.2,0) to (1.2,-.7);
	\node[scale=.8] at (1.35,-.5) {$V$};
	\end{tikzpicture}}  : V \us{H_k}\boxtimes H_k \us{A_k}\boxtimes H_k \to V $ is given by $q \us{H_k} \boxtimes \xi_1 \us{A_k} \boxtimes \xi_2 \longmapsto q.(\xi_1 \cdot \xi_2)$ where $V \in \mcal R_{H_k}$.

\vspace*{2mm}

 \item [(iii)]$\raisebox{-5mm}{\begin{tikzpicture}
	\draw[orange,in=-90,out=-90,looseness=2,->] (0,0) to (1,0);
	\end{tikzpicture}} : \t{Id}_{\mcal R_{A_k}} \to \ol \Lambda_k \Lambda_k  $ is the natural transformation defined as :\\  
	 $\raisebox{-4mm}{\begin{tikzpicture}
	\draw[orange,in=-90,out=-90,looseness=2,->] (0,0) to (1,0);
	\draw[densely dotted] (1.2,0) to (1.2,-.7);
	\node[scale=.8] at (1.35,-.5) {$V$};
	\end{tikzpicture}} : V \to V \us{A_k}\boxtimes H_k \us{H_k}\boxtimes H_k$ is given by $q  \longmapsto  q \us{A_k} \boxtimes 1_{H_k} \us{H_k}\boxtimes 1_{H_k}$ where $V \in \mcal R_{A_k} $.

\vspace*{2mm}

\item [(iv)]$\raisebox{-2mm}{\begin{tikzpicture}
	\draw[orange,in=90,out=90,looseness=2,<-] (0,0) to (1,0);
	\end{tikzpicture}} : \ol \Lambda_k \Lambda_k \to \t{Id}_{\mcal R_{A_k}} $ is the natural transformation defined as : \\
	 $\raisebox{-4mm}{\begin{tikzpicture}
	\draw[orange,in=90,out=90,looseness=2,<-] (0,-.7) to (1,-.7);
	\draw[densely dotted] (1.2,0) to (1.2,-.7);
	\node[scale=.8] at (1.35,-.5) {$V$};
	\end{tikzpicture}}  : V \us{A_k}\boxtimes H_k \us{H_k}\boxtimes H_k \to V $ is given by $q \us{A_k} \boxtimes \xi_1 \us{H_k} \boxtimes \xi_2  \longmapsto q.\lab \xi_2 , \xi_1^{\dagger} \rab_{A_k}$ where $V \in \mcal R_{A_k}$.
\end{itemize}	

\end{defn}

\begin{lem}\label{lambdaklem}
	\begin{itemize}
	
	\item[(i)] $\raisebox{-3mm}{\begin{tikzpicture}
		\draw[orange,in=-90,out=-90,looseness=2,<-] (0,0) to (1,0);
		\end{tikzpicture}}$, $\raisebox{-1mm}{\begin{tikzpicture}
		\draw[orange,in=90,out=90,looseness=2,->] (0,0) to (1,0);
		\end{tikzpicture}}$, $\raisebox{-3mm}{\begin{tikzpicture}
		\draw[orange,in=-90,out=-90,looseness=2,->] (0,0) to (1,0);
		\end{tikzpicture}}$, $\raisebox{-1mm}{\begin{tikzpicture}
		\draw[orange,in=90,out=90,looseness=2,<-] (0,0) to (1,0);
		\end{tikzpicture}}$ satisfy conjugate equations for $\Lambda_k , \ol \Lambda_k$ for each $k \geq l \ $. 
		
	\item[(ii)]	Also, $\left( \, \raisebox{-3mm}{\begin{tikzpicture}
		\draw[orange,in=-90,out=-90,looseness=2,<-] (0,0) to (1,0);
		\end{tikzpicture}} \, \right)^* = \raisebox{-1mm}{\begin{tikzpicture}
		\draw[orange,in=90,out=90,looseness=2,->] (0,0) to (1,0);
		\end{tikzpicture}} $ and $\left( \, \raisebox{-3mm}{\begin{tikzpicture}
		\draw[orange,in=-90,out=-90,looseness=2,->] (0,0) to (1,0);
		\end{tikzpicture}} \, \right)^* = \raisebox{-1mm}{\begin{tikzpicture}
		\draw[orange,in=90,out=90,looseness=2,<-] (0,0) to (1,0);
		\end{tikzpicture}} \ $.
	\item[(iii)] $
	\raisebox{-6mm}{\begin{tikzpicture}
		\draw[orange,in=-90,out=-90,looseness=2,<-] (0,0) to (1,0);
		\draw[orange,in=90,out=90,looseness=2,->] (0,0) to (1,0);
		\end{tikzpicture}} = {\normalfont\t{id}_{\t{Id}_{\mcal R_{H_k}}}}
	$
	
\end{itemize}  
\end{lem}
\begin{proof}
The verification of (i) and (ii) are straightforward and is left to the reader.

For (iii), we use \Cref{HkCE} and [Example 3.25 of \cite{CPJP}].
\end{proof}
\comments{\begin{rem}\label{lambdaWk}
For our later purposes we now describe the unitaries $W_k^\Lambda : \Delta_k \Lambda_{k-1} \to \Lambda_k \Sigma_k$ for each $k \geq l+1 \ $.  

For every $V \in \mcal R_{A_{k-1}}$ define $\left( W_k^\Lambda \right)_V : V \us{A_{k-1}}\boxtimes H_{k-1} \us{H_{k-1}}\boxtimes {H_k}_{H_k} \to   V \us{A_{k-1}}\boxtimes A_k \us{A_k}\boxtimes {H_k}_{H_k} $ as follows :
\[V \us{A_{k-1}}\boxtimes H_{k-1} \us{H_{k-1}}\boxtimes {H_k} \ni q \us{A_{k-1}}\boxtimes \xi_1 \us{H_{k-1}}\boxtimes \xi_2 \os{\left( W_k^\Lambda \right)_V} \longmapsto q \us{A_{k-1}}\boxtimes 1_{A_k} \us{A_k}\boxtimes \xi_1 \cdot \xi_2 \ \ \t{for each} \ \ q \in V \]
It is easy to see that each $\left( W_k^\Lambda \right)_V$ is a unitary natural in $V$ and
$\left( W_k^\Lambda \right)_V^* : V \us{A_{k-1}}\boxtimes A_k \us{A_k}\boxtimes {H_k}_{H_k} \to V \us{A_{k-1}}\boxtimes H_{k-1} \us{H_{k-1}}\boxtimes {H_k}_{H_k} $ is given as follows :
\[V \us{A_{k-1}}\boxtimes A_k \us{A_k}\boxtimes {H_k} \ni q \us{A_{k-1}}\boxtimes \alpha \us{A_k}\boxtimes \xi \os{\left( W_k^\Lambda \right)_V^*} \longmapsto q \us{A_{k-1}}\boxtimes 1_{H_{k-1}} \us{A_k}\boxtimes \xi_1 \cdot \xi_2 \ \ \t{for each} \ \ q \in V \]
\end{rem} }
Pictorially, we denote $W_k^\Lambda$ by
\raisebox{-4mm}{
	\begin{tikzpicture}
	\draw[white,line width=1mm,out=-90,in=90] (-.25,.5) to (.25,-.5);
	\draw[out=-90,in=90,orange,<-] (-.25,.5) to (.25,-.5);
	\begin{scope}[on background layer]
	\draw[out=90,in=-90] (-.25,-.5) to (.25,.5);
	\end{scope}
	\node[right] at (0.15,-0.4) {$\Lambda_{k-1}$};
	\node[right] at (0.15,0.3) {$\Sigma_{k}$};
	\node[left] at (-0.15,-0.4) {$\Delta_{k}$};
	\node[left] at (-0.15,0.3) {$\Lambda_{k}$};
	\end{tikzpicture}} ,
${\left(W_k^\Lambda\right)}^*$ by
\raisebox{-4mm}{
	\begin{tikzpicture}
	\draw[white,line width=1mm,out=-90,in=90] (.25,.5) to (-.25,-.5);
	\draw[out=-90,in=90,orange,<-] (.25,.5) to (-.25,-.5);
	\begin{scope}[on background layer]
	\draw[out=90,in=-90] (.25,-.5) to (-.25,.5);
	\end{scope}
	\node[right] at (0.15,-0.4) {$\Sigma_{k}$};
	\node[right] at (0.15,0.3) {$\Lambda_{k-1}$};
	\node[left] at (-0.15,-0.4) {$\Lambda_{k}$};
	\node[left] at (-0.15,0.3) {$\Delta_{k}$};
	\end{tikzpicture}}, 
 $\ol W_k^\Lambda$ by \raisebox{-6mm}{
	\begin{tikzpicture}
	\draw[white,line width=1mm,out=-90,in=90] (-.25,.5) to (.25,-.5);
	\draw[out=-90,in=90,orange,->] (-.25,.5) to (.25,-.5);
	\begin{scope}[on background layer]
	\draw[out=90,in=-90] (-.25,-.5) to (.25,.5);
	\end{scope}
	\node[right] at (0.2,-0.4) {$\ol \Lambda_{k-1}$};
	\node[right] at (0.15,0.3) {$\Delta_{k}$};
	\node[left] at (-0.15,-0.4) {$\Sigma_{k}$};
	\node[left] at (-0.15,0.3) {$\ol \Lambda_{k}$};
	\end{tikzpicture}} and $\left(\ol W_k^\Lambda\right)^* $  by 
 \raisebox{-8mm}{
 	\begin{tikzpicture}
 	\draw[out=-90,in=90](-.25,.5) to (.25,-.5);
 	\draw[white,line width=1mm,out=-90,in=90] (-.25,-.5) to (.25,.5);
 	\draw[orange,out=90,in=-90,<-] (-.25,-.5) to (.25,.5);
 	\node[right] at (0.2,-0.4) {$\Delta_{k}$};
 	\node[right] at (0.15,0.3) {$\ol \Lambda_{k-1}$};
 	\node[left] at (-0.15,-0.4) {$\ol \Lambda_{k}$};
 	\node[left] at (-0.15,0.3) {$\Sigma_{k}$};
 	\end{tikzpicture}} for each $k \geq 1 \ $.
We have already defined all $ W^\Lambda_k $'s in \cpageref{lambdaunipageref}, and $ \ol W^\Lambda_k $ for $ 1\leq k \leq l $ in the above two cases.
Now, for $k \geq l$, we define

\noindent $\ol W_{k+1}^\Lambda = \raisebox{-6mm}{
	\begin{tikzpicture}
	\draw[white,line width=1mm,out=-90,in=90] (-.25,.5) to (.25,-.5);
	\draw[out=-90,in=90,orange,->] (-.25,.5) to (.25,-.5);
	\begin{scope}[on background layer]
	\draw[out=90,in=-90] (-.25,-.5) to (.25,.5);
	\end{scope}
	\node[right] at (0.2,-0.4) {$\ol \Lambda_{k}$};
	\node[right] at (0.15,0.3) {$\Delta_{k+1}$};
	\node[left] at (-0.15,-0.4) {$\Sigma_{k+1}$};
	\node[left] at (-0.15,0.3) {$\ol \Lambda_{k+1}$};
	\end{tikzpicture}}
\coloneqq
\raisebox{-10mm}{
	\begin{tikzpicture}
	\draw[white,line width=1mm,out=-90,in=90] (.25,.5) to (-.25,-.5);
	\draw[out=-90,in=90,orange,<-] (.25,.5) to (-.25,-.5);
	\begin{scope}[on background layer]
	\draw[out=90,in=-90] (.25,-.5) to (-.25,.5);
	\draw (-.25,.5) to (-.25,.8);
	\draw (.25,-.5) to (.25,-.8);
	\end{scope}
	\node[right,scale=0.7] at (0,-1) {$\Sigma_{k+1}$};
	\node[left,scale=0.7] at (0,1) {$\Delta_{k+1}$};
	\draw[orange,in=90,out=90,looseness=2,->] (.25,.5) to (.75,.5);
	\draw[orange,->] (.75,.5) to (.75,-.8);
	\draw[orange,in=-90,out=-90,looseness=2,<-] (-.25,-.5) to (-.75,-.5);
	\draw[orange,<-] (-.75,-.5) to (-.75,.8);
	\end{tikzpicture}} \	\t{and} \ $ $\left(\ol W_{k+1}^\Lambda\right)^* $ $ = 
\raisebox{-8mm}{
	\begin{tikzpicture}
	\draw[out=-90,in=90](-.25,.5) to (.25,-.5);
	\draw[white,line width=1mm,out=-90,in=90] (-.25,-.5) to (.25,.5);
	\draw[orange,out=90,in=-90,<-] (-.25,-.5) to (.25,.5);
	\node[right] at (0.2,-0.4) {$\Delta_{k+1}$};
	\node[right] at (0.15,0.3) {$\ol \Lambda_{k}$};
	\node[left] at (-0.15,-0.4) {$\ol \Lambda_{k+1}$};
	\node[left] at (-0.15,0.3) {$\Sigma_{k+1}$};
	\end{tikzpicture}} 
\coloneqq
\raisebox{-8mm}{
	\begin{tikzpicture}
	\draw[white,line width=1mm,out=-90,in=90] (-.25,.5) to (.25,-.5);
	\draw[out=-90,in=90,orange,<-] (-.25,.5) to (.25,-.5);
	\draw[orange,in=90,out=90,looseness=2,->] (-.25,.5) to (-.75,.5);
	\draw[orange] (-.75,.5) to (-.75,-.8);
	\draw[orange,in=-90,out=-90,looseness=2,<-] (.25,-.5) to (.75,-.5);
	\draw[orange,<-] (.75,-.5) to (.75,.8);
	\begin{scope}[on background layer]
	\draw[out=90,in=-90] (-.25,-.5) to (.25,.5);
	\end{scope}
	\draw (.25,.5) to (.25,.8);
	\draw (-.25,-.5) to (-.25,-.8);
	\node[right,scale=0.8] at (0.16,0.9) {$\Sigma_{k+1}$};
	\node[left,scale=0.8] at (-0.15,-0.9) {$\Delta_{k+1}$};
	\end{tikzpicture}} $\\
which turn out to be natural unitaries by the following remark.
\begin{rem}\label{lamdakunidual}
For each $k \geq l$ and $V \in \mcal R_{H_k}$, $q \in V$, $\xi \in H_k$, $\alpha \in A_{k+1}$, $\eta$, $\zeta \in H_{k+1}$ the element $ \left(\ol W_{k+1}^\Lambda \right)_V \left(q \us{H_k}\boxtimes \xi \us{A_k}\boxtimes \alpha \right) $ can be expressed as
\begin{align*}
	\raisebox{-10mm}{
	\begin{tikzpicture}
	\draw[white,line width=1mm,out=-90,in=90] (.25,.5) to (-.25,-.5);
	\draw[out=-90,in=90,orange,<-] (.25,.5) to (-.25,-.5);
	\begin{scope}[on background layer]
	\draw[out=90,in=-90] (.25,-.5) to (-.25,.5);
	\draw (-.25,.5) to (-.25,.8);
	\draw (.25,-.5) to (.25,-.8);
	\end{scope}
	\node[right,scale=0.7] at (-.2,-.9) {$\Sigma_{k+1}$};
	\node[left,scale=0.7] at (0,1) {$\Delta_{k+1}$};
	\node[scale=.8] at (1.4,-.5) {$V$}; 
	\draw[orange,in=90,out=90,looseness=2,->] (.25,.5) to (.75,.5);
	\draw[orange,->] (.75,.5) to (.75,-.8);
	\draw[orange,in=-90,out=-90,looseness=2,<-] (-.25,-.5) to (-.75,-.5);
	\draw[orange,<-] (-.75,-.5) to (-.75,.8);
	\draw[dotted] (1.2,.8) to (1.2,-.8);
	\end{tikzpicture}} \Big(q \us{H_k}\boxtimes \xi \us{A_k}\boxtimes \alpha \Big)
&= \raisebox{-10mm}{
	\begin{tikzpicture}
	\draw[white,line width=1mm,out=-90,in=90] (.25,.5) to (-.25,-.5);
	\draw[out=-90,in=90,orange,<-] (.25,.5) to (-.25,-.5);
	\begin{scope}[on background layer]
	\draw[out=90,in=-90] (.25,-.5) to (-.25,.5);
	\draw (-.25,.5) to (-.25,.8);
	\draw (.25,-.5) to (.25,-.8);
	\end{scope}
	\node[right,scale=0.7] at (-.2,-.9) {$\Sigma_{k+1}$};
	\node[left,scale=0.7] at (0,1) {$\Delta_{k+1}$};
	\node[scale=.8] at (1.4,-.5) {$V$}; 
	\draw[orange,in=90,out=90,looseness=2,->] (.25,.5) to (.75,.5);
	\draw[orange,->] (.75,.5) to (.75,-.8);
	\draw[orange,<-] (-.75,-.8) to (-.75,.8);
	\draw[orange,<-] (-.25,-.5) to (-.25,-.8);
	\draw[dotted] (1.2,.8) to (1.2,-.8);
	\end{tikzpicture}} \Big(q \us{H_k}\boxtimes \xi \us{A_k}\boxtimes \alpha \us{A_{k+1}} \boxtimes 1_{H_{k+1}} \us{H_{k+1}} \boxtimes 1_{H_{k+1}} \Big) \\
&= \raisebox{-8mm}{
	\begin{tikzpicture}
	\begin{scope}[on background layer]
	\draw (-.25,-.8) to (-.25,.8);
	\end{scope}
	\node[left,scale=0.7] at (0,1) {$\Delta_{k+1}$};
	\node[scale=.8] at (1.4,-.5) {$V$}; 
	\draw[orange,in=90,out=90,looseness=2,->] (.25,.5) to (.75,.5);
	\draw[orange,->] (.75,.5) to (.75,-.8);
	\draw[orange,<-] (-.75,-.8) to (-.75,.8);
	\draw[orange,<-] (.25,.5) to (.25,-.8);
	\draw[dotted] (1.2,.8) to (1.2,-.8);
	\end{tikzpicture}} \Big(q \us{H_k}\boxtimes \xi \us{A_k}\boxtimes 1_{H_k} \us{A_{k+1}} \boxtimes \alpha \us{H_{k+1}} \boxtimes 1_{H_{k+1}} \Big) = q.\xi \us{H_k} \boxtimes \alpha \us{H_{k+1}} \boxtimes 1_{H_{k+1}}
\end{align*}
and $\left(\ol W_{k+1}^\Lambda \right)_V^* \left(q \us{H_k}\boxtimes \eta  \us{H_{k+1}}\boxtimes \zeta \right)$ can be expressed as 
\begin{align*}
	\raisebox{-10mm}{
	\begin{tikzpicture}
	\draw[white,line width=1mm,out=-90,in=90] (-.25,.5) to (.25,-.5);
	\draw[out=-90,in=90,orange,<-] (-.25,.5) to (.25,-.5);
	\draw[orange,in=90,out=90,looseness=2,->] (-.25,.5) to (-.75,.5);
	\draw[orange] (-.75,.5) to (-.75,-.8);
	\draw[orange,in=-90,out=-90,looseness=2,<-] (.25,-.5) to (.75,-.5);
	\draw[orange,<-] (.75,-.5) to (.75,.8);
	\begin{scope}[on background layer]
	\draw[out=90,in=-90] (-.25,-.5) to (.25,.5);
	\end{scope}
	\draw (.25,.5) to (.25,.8);
	\draw (-.25,-.5) to (-.25,-.8);
	\draw[dotted] (1,-.8) to (1,.8);
	\node[scale=.8] at (1.4,-.5) {$V$};
	\node[right,scale=0.8] at (-.2,1) {$\Sigma_{k+1}$};
	\node[left,scale=0.8] at (0.15,-0.8) {$\Delta_{k+1}$};
	\end{tikzpicture}}\left(q \us{H_k}\boxtimes \eta  \us{H_{k+1}}\boxtimes \zeta \right)
	&= \raisebox{-10mm}{
	\begin{tikzpicture}
	\draw[white,line width=1mm,out=-90,in=90] (-.25,.5) to (.25,-.5);
	\draw[out=-90,in=90,orange,<-] (-.25,.5) to (.25,-.5);
	\draw[orange,in=90,out=90,looseness=2,->] (-.25,.5) to (-.75,.5);
	\draw[orange] (-.75,.5) to (-.75,-.8);
	\draw[orange,<-] (.75,-.8) to (.75,.8);
	\draw[orange] (.25,-.5) to (.25,-.8);
	\begin{scope}[on background layer]
	\draw[out=90,in=-90] (-.25,-.5) to (.25,.5);
	\end{scope}
	\draw (.25,.5) to (.25,.8);
	\draw (-.25,-.5) to (-.25,-.8);
	\draw[dotted] (1,-.8) to (1,.8);
	\node[scale=.8] at (1.4,-.5) {$V$};
	\node[right,scale=0.8] at (-.2,1) {$\Sigma_{k+1}$};
	\node[left,scale=0.8] at (0.15,-0.8) {$\Delta_{k+1}$};
	\end{tikzpicture}} \left(\us{\sigma \in \mscr S_k}{\sum} q \us{H_k}\boxtimes \sigma \us{A_k}\boxtimes \sigma^\dagger \us{H_k}\boxtimes \eta \us{H_{k+1}}\boxtimes \zeta \right)\\ 
  &= \raisebox{-10mm}{
 	\begin{tikzpicture}
 	\draw[orange,in=90,out=90,looseness=2,->] (-.25,.5) to (-.75,.5);
 	\draw[orange] (-.75,.5) to (-.75,-.8);
 	\draw[orange,<-] (.75,-.8) to (.75,.8);
 	\draw (.25,.8) to (.25,-.8);
 	\begin{scope}[on background layer]
 	\end{scope}
 	\draw (.25,.5) to (.25,.8);
 	\draw[orange,<-] (-.25,.5) to (-.25,-.8);
 	\draw[dotted] (1,-.8) to (1,.8);
 	\node[scale=.8] at (1.4,-.5) {$V$};
 	\node[right,scale=0.8] at (-.2,1) {$\Sigma_{k+1}$};
 	\end{tikzpicture}} \left(\us{\sigma \in \mscr S_k}{\sum} q \us{H_k}\boxtimes \sigma \us{A_k}\boxtimes 1_{A_{k+1}} \us{A_{k+1}}\boxtimes \sigma^\dagger \cdot \eta \us{H_{k+1}}\boxtimes \zeta \right)\\
  &= \us{\sigma \in \mscr S_k}{\sum} q \us{H_k}\boxtimes \sigma \us{A_k}\boxtimes \lab \zeta , \eta^\dagger \cdot \sigma \rab_{A_{k+1}} \ .
 \end{align*}
It is a straightforward verification that each $\left(\ol W_{k+1}^\Lambda\right)_V$ is a unitary and natural in $V \ $.
\end{rem}
Thus, we have defined a $ 1 $-cell $\left(\ol \Lambda_{\bullet}, \ol W_\bullet^\Lambda\right) $ in $\textbf{UC}$ from $\Delta_\bullet$ to $ \Sigma_\bullet $.
We need to prove that $\left(\ol \Lambda_{\bullet}, \ol W_\bullet^\Lambda\right)$ is dual to $\left( \Lambda_{\bullet},  W_\bullet^\Lambda\right) \ $.
In order to define the solution to conjugate equation (which is in fact a pair of $ 2 $-cells in \textbf{UC}), we have the liberty to ignore finitely many terms and define them eventually (by \Cref{UCisomorphism}).
By \Cref{lambdaklem}, we see that there are solutions to conjugate equations for $\Lambda_{k}$ and $\ol \Lambda_{k}$ for each $k \geq l \ $. So, we are only left with showing exchange relations of solutions eventually.

 We now verify that $\raisebox{-4mm}{\begin{tikzpicture}
	\draw[orange,in=-90,out=-90,looseness=2,<-] (0,0) to (1,0);
	\end{tikzpicture}}$ and $\raisebox{-4mm}{\begin{tikzpicture}
	\draw[orange,in=-90,out=-90,looseness=2,->] (0,0) to (1,0);
	\end{tikzpicture}}$ satisfy exchange relations for $k \geq l \ $. 

\begin{rem}\label{xrellambdak}
The solutions to conjugate equations for $\Lambda_k$ and $ \Lambda_{k+1} $ (as in \Cref{lambdaksol}) satisfy exchange relation eventually for all $ k $ with respect to $W_\bullet^\Lambda$ and $\ol W_\bullet^\Lambda \ $.
This directly follows from the fact $ W^\Lambda_k $'s and $ \ol W^\Lambda_k $'s are unitaries.
Nevertheless, we still furnish a proof below.
Note that
\[
\raisebox{-10mm}{
 	\begin{tikzpicture}[rotate=60,scale=1.2]
 	\draw[orange,in=-90,out=-90,looseness=2,->] (0,0) to (0.5,0);
 	\draw[orange,<-] (0,0) to (0,1);
 	\draw[orange] (0.5,0) to (0.5,1);
 	\draw (-.6,.6) to (-.05,.6);
 	\draw (0.05,.6) to (0.45,.6);
 	\draw (0.55,.6) to (1,.6);
 	\node[left,scale=0.7] at (-.6,.5) {$\Sigma_{k+1}$};
 	\node[left,scale=0.7] at (-.2,0) {$\ol \Lambda_k$};
 	\node[left,scale=0.6] at (.4,.6) {$\Delta_{k+1}$};
 	\node[left,scale=0.7] at (0,.7) {$\ol \Lambda_{k+1}$};
 	\node[right,scale=0.7] at (.8,.56) {$\Sigma_{k+1}$};  
 	\node[right,scale=0.7] at (.5,0) {$\Lambda_k$};
 	\node[right,scale=0.7] at (.5,1.2) {$\Lambda_{k+1}$};
 	\end{tikzpicture}} = 
 \raisebox{-16mm}{
 	\begin{tikzpicture}
 	\draw[white,line width=1mm,out=-90,in=90] (.25,.5) to (-.25,-.5);
 	\draw[out=-90,in=90,orange,<-] (.25,.5) to (-.25,-.5);
 	\begin{scope}[on background layer]
 	\draw[out=90,in=-90] (.25,-.5) to (-.25,.5);
 	\end{scope}
 	\node[right,scale=0.7] at (0.15,-0.7) {$\Sigma_{k+1}$};
 	\node[left,scale=0.7] at (-0.15,0.3) {$\Delta_{k+1}$};
 	\draw[orange,in=90,out=90,looseness=2,->] (.25,.5) to (.75,.5);
 	\draw[orange,->] (.75,.5) to (.75,-.2);
 	\draw[orange,in=-90,out=-90,looseness=2,<-] (-.25,-.5) to (-.75,-.5);
 	\draw[orange,<-] (-.75,-.5) to (-.75,0);
 	\draw[orange,in=-90,out=-90,looseness=2] (.75,-.2) to (1.25,-.2);
 	\draw[orange] (1.25,-.2) to (1.25,1);
 	\draw[in=-90,out=90] (-.25,.5) to (.25,1.5);
 	\draw (.25,1.5) to (.25,2);
 	\draw[white,line width=1mm,in=-90,out=90] (1.25,1) to (-.5,2);
 	\draw[orange,in=-90,out=90,->] (1.25,1) to (-.5,2);
 	\end{tikzpicture}}
  = \raisebox{-10mm}{
 	\begin{tikzpicture}
 	\draw[white,line width=1mm,out=-90,in=90] (.25,.5) to (-.25,-.5);
 	\draw[out=-90,in=90,orange,<-] (.25,.5) to (-.25,-.5);
 	\begin{scope}[on background layer]
 	\draw[out=90,in=-90] (.25,-.5) to (-.25,.5);
 	\end{scope}
 	\node[right,scale=0.7] at (0.2,-0.4) {$\Sigma_{k+1}$};
 	\node[left,scale=0.7] at (-0.15,0.3) {$\Delta_{k+1}$};
 	\draw[orange,in=-90,out=-90,looseness=2,<-] (-.25,-.5) to (-.75,-.5);
 	\draw[orange,<-] (-.75,-.5) to (-.75,0);
 	\draw[in=-90,out=90] (-.25,.5) to (.25,1.5);
 	\draw[white,line width=1mm,in=-90,out=90,looseness=2] (.25,.5) to (-.25,1.5);
 	\draw[orange,in=-90,out=90,looseness=2] (.25,.5) to (-.25,1.5);
 	\end{tikzpicture}}
 =
  \raisebox{-8mm}{
 	\begin{tikzpicture}
 	\draw[orange,in=-90,out=-90,looseness=2,->] (0,0) to (.5,0);
 	\draw[orange,<-] (0,0) to (0,1);
 	\draw[orange,->] (0.5,0) to (0.5,1);
 	\draw (1,-.4) to (1,1);
 	\node[right,scale=0.7] at (1,.5) {$\Sigma_{k+1}$};
 	\end{tikzpicture}}
\]
and
\[\raisebox{-10mm}{
 	\begin{tikzpicture}[rotate=60,scale=1.2]
 	\draw[orange,in=-90,out=-90,looseness=2,<-] (0,0) to (0.5,0);
 	\draw[orange] (0,0) to (0,1);
 	\draw[orange,<-] (0.5,0) to (0.5,1);
 	\draw (-.6,.6) to (-.05,.6);
 	\draw (0.05,.6) to (0.45,.6);
 	\draw (0.55,.6) to (1,.6);
 	\node[left,scale=0.7] at (-.6,.5) {$\Delta_{k+1}$};
 	\node[left,scale=0.7] at (-.2,0) {$\Lambda_k$};
 	\node[left,scale=0.6] at (.4,.6) {$\Sigma_{k+1}$};
 	\node[left,scale=0.7] at (0,.7) {$\Lambda_{k+1}$};
 	\node[right,scale=0.7] at (.8,.56) {$\Delta_{k+1}$};  
 	\node[right,scale=0.7] at (.5,0) {$\ol \Lambda_k$};
 	\node[right,scale=0.7] at (.5,1.2) {$\ol \Lambda_{k+1}$};
 	\end{tikzpicture}}=
 \raisebox{-12mm}{
 	\begin{tikzpicture}
 	\draw[white,line width=1mm,out=-90,in=90] (.25,.5) to (-.25,-.5);
 	\draw[out=-90,in=90,orange,<-] (.25,.5) to (-.25,-.5);
 	\begin{scope}[on background layer]
 	\draw[out=90,in=-90] (.25,-.5) to (-.25,.5);
 	\end{scope}
 	\node[right,scale=0.7] at (-.2,-1.2) {$\Sigma_{k+1}$};
 	\node[left,scale=0.7] at (-0.15,0.3) {$\Delta_{k+1}$};
 	\draw[orange,in=90,out=90,looseness=2,->] (.25,.5) to (.75,.5);
 	\draw[orange,->] (.75,.5) to (.75,-.2);
 	\draw[orange,in=-90,out=-90,looseness=2,<-] (-.25,-.5) to (-.75,-.5);
 	\draw[orange,<-] (-.75,-.5) to (-.75,0);
 	\draw[in=90,out=-90] (.25,-.5) to (-.75,-1.5);
 	\draw (-.75,-1.5) to (-.75,-2.4);
 	\draw[white,line width=1mm] (-.25,-2) to (-1.25,-1);
 	\draw[orange,in=-90,out=90,->] (-.25,-2) to (-1.25,-1);
 	\draw[orange,in=-90,out=-90,looseness=2,<-] (-.25,-2) to (.75,-2);
 	\draw[orange] (.75,-2) to (.75,-.2);
 	\draw[orange] (-1.25,-1) to (-1.25,0);
 	\end{tikzpicture}}
 = \raisebox{-6mm}{
	\begin{tikzpicture}
		\draw[orange,in=-90,out=-90,looseness=2,<-] (0,0) to (.5,0);
		\draw[orange] (0,0) to (0,1);
		\draw[orange,<-] (.5,0) to (.5,1);
		\draw (1,-.4) to (1,1);
		\node[scale=.8] at (1.5,.5) {$\Delta_{k+1}$}; 
\end{tikzpicture}} \ \ \ \ .
 \]
\end{rem}
Hence, $\left(\Lambda_{\bullet}, W_\bullet^\Lambda\right) : \Sigma_\bullet \to \Delta_\bullet$ is a dualizable $1$-cell in $\textbf{UC} \ $ with dual $\left(\ol \Lambda_{\bullet}, \ol W_\bullet^\Lambda\right)$ as described above .

\vspace*{4mm}

We are now in a position to describe our desired dualizable $1$-cell $\left(X_\bullet, W_\bullet \right)$ which will split $\left(Q_\bullet, m_\bullet, i_\bullet \right)$ as Q-system.\\
Define $\left(X_\bullet, W_\bullet \right) \coloneqq \Lambda_{\bullet} \boxtimes F_\bullet = \left( \left\{\Lambda_k F_k\right\}_{k \geq 0} , \left\{\raisebox{-10mm}{\begin{tikzpicture}
	\draw[in=-90,out=90,looseness=2] (0,0) to (2,2);
	\draw[white,line width=1mm,in=-90,out=90,looseness=2,->] (1,0) to (.5,2);
	\draw[white,line width=1mm,in=-90,out=90,looseness=2,->] (1.7,0) to (1.2,2);
	\draw[orange,in=-90,out=90,looseness=2,->] (1,0) to (.5,2);
	\draw[violet,in=-90,out=90,looseness=2,->] (1.7,0) to (1.2,2);
	\node[scale=.8] at (.8,-.2) {$\Lambda_{k-1}$};
	\node[scale=.8] at (1.5,-.2) {$F_{k-1}$};
	\node[scale=.8] at (-.2,-.2) {$\Delta_{k}$};
	\node[scale=.6] at (1,1.2) {$\Sigma_{k}$};
	\node[scale=.8] at (2.2,2.2) {$\Gamma_{k}$};
	\node[scale=.8] at (.5,2.2) {$\Lambda_{k}$};
	\node[scale=.8] at (1.3,2.2) {$F_{k}$};
	\end{tikzpicture}}\right\}_{k \geq 1} \right) \in 
\t{\textbf{UC}}_1 \left(\Gamma_\bullet , \Delta_\bullet\right)$ .  \\
\comments{For each $k \geq 0$ we define $W_{k+1} : \Delta_{k+1} X_{k} \to X_{k+1} \Gamma_{k+1}$ as follows: 
\[  W_{k+1} \coloneqq \begin{cases}
\raisebox{-10mm}{\begin{tikzpicture}
\draw[in=-90,out=90,looseness=2] (0,0) to (2,2);
\draw[white,line width=1mm,in=-90,out=90,looseness=2,->] (1,0) to (.5,2);
\draw[white,line width=1mm,in=-90,out=90,looseness=2,->] (1.7,0) to (1.2,2);
\draw[orange,in=-90,out=90,looseness=2,->] (1,0) to (.5,2);
\draw[violet,in=-90,out=90,looseness=2,->] (1.7,0) to (1.2,2);
\node[scale=.8] at (.8,-.2) {$\Lambda_k$};
\node[scale=.8] at (1.5,-.2) {$F_k$};
\node[scale=.8] at (-.2,-.2) {$\Delta_{k+1}$};
\node[scale=.6] at (1,1.2) {$\Sigma_{k+1}$};
\node[scale=.8] at (2.2,2.2) {$\Gamma_{k+1}$};
\node[scale=.8] at (.5,2.2) {$\Lambda_{k+1}$};
\node[scale=.8] at (1.3,2.2) {$F_{k+1}$};
\end{tikzpicture}} & \ \ \ \ \t{if} \ k \geq l \\
\t{Id}_{X_l \, \Gamma_{l} \cdots \Gamma_{k+1}}  & \ \ \ \ \t{if} \ 0 \leq k \leq l-1 \ .
\end{cases}\]
}
Pictorially, we denote $X_k$ by \raisebox{-4mm}{\begin{tikzpicture}
	\draw[blue,->] (0,0) to (0,1);
	\end{tikzpicture}} , $\ol X_k$ by \raisebox{-4mm}{\begin{tikzpicture}
	\draw[blue,<-] (0,0) to (0,1);
	\end{tikzpicture}} , $W_k$ by
\raisebox{-4mm}{
	\begin{tikzpicture}
	\draw[white,line width=1mm,out=-90,in=90] (-.25,.5) to (.25,-.5);
	\draw[out=-90,in=90,blue,<-] (-.25,.5) to (.25,-.5);
	\begin{scope}[on background layer]
	\draw[out=90,in=-90] (-.25,-.5) to (.25,.5);
	\end{scope}
	\node[right] at (0.15,-0.4) {$X_{k-1}$};
	\node[right] at (0.15,0.3) {$\Gamma_{k}$};
	\node[left] at (-0.15,-0.4) {$\Delta_{k}$};
	\node[left] at (-0.15,0.3) {$X_{k}$};
	\end{tikzpicture}}
and $W_k^*$ by
\raisebox{-4mm}{
	\begin{tikzpicture}
	\draw[white,line width=1mm,out=-90,in=90] (.25,.5) to (-.25,-.5);
	\draw[out=-90,in=90,blue,<-] (.25,.5) to (-.25,-.5);
	\begin{scope}[on background layer]
	\draw[out=90,in=-90] (.25,-.5) to (-.25,.5);
	\end{scope}
	\node[right] at (0.15,-0.4) {$\Gamma_{k}$};
	\node[right] at (0.15,0.3) {$X_{k-1}$};
	\node[left] at (-0.15,-0.4) {$X_{k}$};
	\node[left] at (-0.15,0.3) {$\Delta_{k}$};
	\end{tikzpicture}} .\\
Define $\ol W_k \coloneqq \raisebox{-6mm}{
	\begin{tikzpicture}
	\draw[white,line width=1mm,out=-90,in=90] (-.25,.5) to (.25,-.5);
	\draw[out=-90,in=90,blue,->] (-.25,.5) to (.25,-.5);
	\begin{scope}[on background layer]
	\draw[out=90,in=-90] (-.25,-.5) to (.25,.5);
	\end{scope}
	\node[right] at (0.2,-0.4) {$\ol X_{k-1}$};
	\node[right] at (0.15,0.3) {$\Delta_{k}$};
	\node[left] at (-0.15,-0.4) {$\Gamma_{k}$};
	\node[left] at (-0.15,0.3) {$\ol X_{k}$};
	\end{tikzpicture}}
\coloneqq
\raisebox{-10mm}{
	\begin{tikzpicture}
	\draw[white,line width=1mm,out=-90,in=90] (.25,.5) to (-.25,-.5);
	\draw[out=-90,in=90,blue,<-] (.25,.5) to (-.25,-.5);
	\begin{scope}[on background layer]
	\draw[out=90,in=-90] (.25,-.5) to (-.25,.5);
	\draw (-.25,.5) to (-.25,.8);
	\draw (.25,-.5) to (.25,-.8);
	\end{scope}
	\node[right,scale=0.7] at (0,-1) {$\Gamma_{k}$};
	\node[left,scale=0.7] at (0,1) {$\Delta_{k}$};
	\draw[blue,in=90,out=90,looseness=2,->] (.25,.5) to (.75,.5);
	\draw[blue,->] (.75,.5) to (.75,-.8);
	\draw[blue,in=-90,out=-90,looseness=2,<-] (-.25,-.5) to (-.75,-.5);
	\draw[blue,<-] (-.75,-.5) to (-.75,.8);
	\end{tikzpicture}} \	\t{and} \ $ $\left(\ol W_k\right)^* $ $ \coloneqq 
\raisebox{-8mm}{
	\begin{tikzpicture}
	\draw[out=-90,in=90](-.25,.5) to (.25,-.5);
	\draw[white,line width=1mm,out=-90,in=90] (-.25,-.5) to (.25,.5);
	\draw[blue,out=90,in=-90,<-] (-.25,-.5) to (.25,.5);
	\node[right] at (0.2,-0.4) {$\Delta_{k}$};
	\node[right] at (0.15,0.3) {$\ol X_{k-1}$};
	\node[left] at (-0.15,-0.4) {$\ol X_{k}$};
	\node[left] at (-0.15,0.3) {$\Gamma_{k}$};
	\end{tikzpicture}} 
\coloneqq
\raisebox{-8mm}{
	\begin{tikzpicture}
	\draw[white,line width=1mm,out=-90,in=90] (-.25,.5) to (.25,-.5);
	\draw[out=-90,in=90,blue,<-] (-.25,.5) to (.25,-.5);
	\draw[blue,in=90,out=90,looseness=2,->] (-.25,.5) to (-.75,.5);
	\draw[blue] (-.75,.5) to (-.75,-.8);
	\draw[blue,in=-90,out=-90,looseness=2,<-] (.25,-.5) to (.75,-.5);
	\draw[blue,<-] (.75,-.5) to (.75,.8);
	\begin{scope}[on background layer]
	\draw[out=90,in=-90] (-.25,-.5) to (.25,.5);
	\end{scope}
	\draw (.25,.5) to (.25,.8);
	\draw (-.25,-.5) to (-.25,-.8);
	\node[right,scale=0.8] at (0.16,0.9) {$\Gamma_{k}$};
	\node[left,scale=0.8] at (-0.15,-0.9) {$\Delta_{k}$};
	\end{tikzpicture}} $ .
Thus, we arrive at our desired $1$-cell $\left(X_\bullet, W_\bullet \right) \in \textbf{UC}_1(\Gamma_\bullet, \Delta_\bullet)$. We list some of the properties of $\left(X_\bullet, W_\bullet \right) \ $ .
\begin{lem}\
 	\begin{itemize}
		\item[(i)] $\left(X_\bullet, W_\bullet \right)$ is a dualizable $1$-cell in $\normalfont \textbf{UC} \ $.
		\item [(ii)] $\left(X_\bullet, W_\bullet \right)$ has a unitarily separable dual in $\normalfont \textbf{UC} \ $.
	\end{itemize}
\end{lem}
\begin{proof}
\begin{itemize}
	\item [(i)]
$\left(X_\bullet, W_\bullet\right)$ being a composition of two dualizable $1$-cells $\left(\Lambda_{\bullet}, W_\bullet^\Lambda \right)$ and $\left(F_\bullet, W_\bullet^F \right)$ concludes the result .

\item[(ii)] This is immediate from \Cref{UCisomorphism} and (iii) of \Cref{lambdaklem} .
\end{itemize}
\end{proof}

\subsection{Isomorphism of $Q$-systems}\

In this subsection, we build an isomorhpism between $\ol X_\bullet \boxtimes X_\bullet$ and $Q_\bullet$. We construct unitaries $\gamma^{(k)} : \ol X_k X_k \to Q_k$ for each $k \geq l$ which intertwines the mutliplication and unit maps. In the next subsection, we verify the exchange relation of $\gamma^{(k)}$ for each $k \geq l$, thus implementing isomorphism of the aforementioned $Q$-systems in $\textbf{UC} \ $.

\vspace*{2mm}

For $k \geq l$ and for each $x \in \t{Ob}(\mcal M_k)$, define a map $\beta_x^{(k)} : \ol \Lambda_k \Lambda_k F_k(x) \to F_k Q_k (x) $ as follows :
\[ F_k(x) \us{A_k} \boxtimes H_k \us{H_k} \boxtimes {H_k}_{A_k} \ni u \boxtimes \xi_1 \boxtimes \xi_2 \os{\beta_x^{(k)}}\longmapsto \raisebox{-10mm}{\begin{tikzpicture}
	\draw (-.3,0) to (-.3,-2);
	\draw (.3,0) to (.3,-2);
	\draw[dashed] (.5,0) to (.5,-2);
	\draw[densely dotted] (.1,0) to (.1,.8);
	\draw[red] (-.7,-1.1) to (-.7,.8);
	\node at (0,-.5) {$\cdots$};
	\node at (0,-1.6) {$\cdots$};
	\node[scale=.8] at (-.5,-2) {$\Gamma_k$};
	\node[scale=.8] at (.1,-2) {$\Gamma_1$};
	\node at (.9,-.6) {$m_0$};
	\node at (.9,-1.8) {$m_0$};
	\node[scale=.8] at (-1,.4) {$Q_k$};
	\node at (.3,.5) {$x$};
	\node[draw,thick,rounded corners,fill=white,minimum width=35] at (.1,0) {$u$};
	\node[draw,thick,rounded corners,fill=white,minimum width=45] at (0,-1.1) {$\xi_1 \cdot \xi_2$};
	\end{tikzpicture}} \in F_k Q_k(x)  \]  
It is easy to see that, each $\beta_x^{(k)}$ is an isometry. Since, $ _{A_k} H_k \us{H_k} \boxtimes {H_k}_{A_k}$ is unitarily isomorhpic to $_{A_k}{H_k}_{A_k}$ and by application of \Cref{Unitary} we see that $\ol \Lambda_k \Lambda_k F_k(x)$ and $F_k Q_k (x)$ has same dimension (as a vector space). Hence surjectiveness will follow. Thus, each $\beta_x^{(k)}$ is a unitary. Also, it easily follows that each $\beta_x^{(k)}$ is a natural in $x$. Thus, we get a unitary natural transformation $\beta^{(k)} : \ol \Lambda_k \Lambda_k F_k \to F_k Q_k \ $.

Define $\gamma^{(k)} \coloneqq \raisebox{-12mm}{\begin{tikzpicture}
	\draw[red] (0.2,.3) to (0.2,1);
	\draw[violet,in=90,out=90,looseness=2] (-.4,.3) to (-.9,.3);
	\draw[violet,<-] (.4,-.35) to (.4,-1);
	\draw[orange,<-] (0,-.35) to (0,-1);
	\draw[orange,->] (-.4,-.35) to (-.4,-1);
	\draw[violet,->] (-.9,.3) to (-.9,-1);
	\node[draw,thick,rounded corners,fill=white,minimum width=35] at (0,0) {$\beta^{(k)}$};
	\node[scale=.8] at (.3,1.2) {$Q_k$};
	\node[scale=.8] at (.5,-1.2) {$F_k$};
	\node[scale=.8] at (.1,-1.2) {$\Lambda_k$};
	\node[scale=.8] at (-.4,-1.3) {$\ol \Lambda_k$};
	\node[scale=.8] at (-1.1,-.6) {$\ol F_k$};
	\end{tikzpicture}} : \ol X_k X_k \to Q_k$. 
We show that $\gamma^{(k)}$ is an isomorphism of Q-systems $\ol X_k X_k$ and $Q_k$ for $k \geq l$.  Each $\gamma^{(k)}$ is a unitary because each $\beta^{(k)}$ is so and each $F_k$ is an adjoint equivalence (see \Cref{Fkequivalence}). We need to show that $\gamma^{(k)}$ intertwines the multiplication and unit maps. We need to show that $\raisebox{-8mm}{
	\begin{tikzpicture}
	\draw[blue,->] (-.2,-.3) to (-.2,-1);
	\draw[blue,<-] (.2,-.35) to (.2,-1);
	\draw[blue,->] (.8,-.3) to (.8,-1);
	\draw[blue,<-] (1.2,-.35) to (1.2,-1);
	\draw[red,in=90,out=90,looseness=2] (0,.3) to (1,.3);
	\node at (.5,.88) {$\red{\bullet}$};
	\draw[red] (.5,.88) to (.5,1.4);
	\draw[blue,in=90,out=90,looseness=2,->] (2.7,-1) to (3.2,-1);
	\node at (2,0) {$=$};
	\draw[blue,->] (2.5,-.3) to (2.5,-1);
	\draw[blue,<-] (3.4,-.35) to (3.4,-1);
	\draw[red] (3,.3) to (3,1.2);
	\node[draw,thick,rounded corners,fill=white] at (0,0) {$\gamma^{(k)}$};
	\node[draw,thick,rounded corners,fill=white] at (1,0) {$\gamma^{(k)}$};
	\node[draw,thick,rounded corners,fill=white,minimum width=35] at (3,0) {$\gamma^{(k)}$};
	\end{tikzpicture}}$ and 
$\raisebox{-10mm}{
	\begin{tikzpicture}
	\draw[red] (0,.3) to (0,1);
	\draw[blue,->] (-.2,-.3) to (-.2,-1);
	\draw[blue,<-] (.2,-.3) to (.2,-1);
	\draw[blue,in=-90,out=-90,looseness=2] (-.2,-1) to (.2,-1);
	\node[draw,thick,rounded corners,fill=white] at (0,0) {$\gamma^{(k)}$};
	\end{tikzpicture}} \ = \ 
\raisebox{-6mm}{
	\begin{tikzpicture}
	\draw[red] (0,-.4) to (0,1);
	\node at (0,-.4) {$\red{\bullet}$};
	\end{tikzpicture}}$ for $k \geq l$. This is what we prove next.
\begin{prop}\label{Qsysiso}
For $k \geq l$,	$\gamma^{(k)} : \ol X_kX_k \to Q_k$ is an isomorphism of Q-systems.
\end{prop}
\begin{proof}
It easily follows that, $\raisebox{-8mm}{
		\begin{tikzpicture}
		\draw[blue,->] (-.2,-.3) to (-.2,-1);
		\draw[blue,<-] (.2,-.35) to (.2,-1);
		\draw[blue,->] (.8,-.3) to (.8,-1);
		\draw[blue,<-] (1.2,-.35) to (1.2,-1);
		\draw[red,in=90,out=90,looseness=2] (0,.3) to (1,.3);
		\node at (.5,.88) {$\red{\bullet}$};
		\draw[red] (.5,.88) to (.5,1.4);
		\draw[blue,in=90,out=90,looseness=2,->] (2.7,-1) to (3.2,-1);
		\node at (2,0) {$=$};
		\draw[blue,->] (2.5,-.3) to (2.5,-1);
		\draw[blue,<-] (3.4,-.35) to (3.4,-1);
		\draw[red] (3,.3) to (3,1.2);
		\node[draw,thick,rounded corners,fill=white] at (0,0) {$\gamma^{(k)}$};
		\node[draw,thick,rounded corners,fill=white] at (1,0) {$\gamma^{(k)}$};
		\node[draw,thick,rounded corners,fill=white,minimum width=35] at (3,0) {$\gamma^{(k)}$};
		\end{tikzpicture}}$ if and only if \, $\raisebox{-16mm}{
		\begin{tikzpicture}
		\draw[orange,->] (-.7,-.3) to (-.7,-2);
		\draw[orange,<-] (-.3,-.35) to (-.3,-2);
		\draw[orange,->] (.3,-1.55) to (.3,-2);
		\draw[orange,<-] (.7,-1.55) to (.7,-2);
		\draw[red,in=90,out=90,looseness=2] (0,.3) to (1,.3);
		\node at (.5,.88) {$\red{\bullet}$};
		\draw[red] (.5,.88) to (.5,1.4);
		\draw[red] (1,.3) to (1,-1);
		\draw[orange,in=90,out=90,looseness=2,->] (2.7,-1) to (3.2,-1);
		\node at (2,0) {$=$};
		\draw[orange,->] (2.5,-.3) to (2.5,-1);
		\draw[orange,<-] (3.4,-.35) to (3.4,-1);
		\draw[red] (3,.3) to (3,1.2);
		\draw[violet,->] (-.4,.3) to (-.4,1.4);
		\draw[violet,->] (2.6,.3) to (2.6,1.2);
		\draw[violet,<-] (.1,-.35) to (.1,-1);
		\draw[violet,<-] (1,-1.55) to (1,-2);
		\draw[violet,<-] (3.7,-.35) to (3.7,-1);
		\node[draw,thick,rounded corners,fill=white,minimum width=35] at (-.3,0) {$\beta^{(k)}$};
		\node[draw,thick,rounded corners,fill=white,minimum width=35] at (.6,-1.2) {$\beta^{(k)}$};
		\node[draw,thick,rounded corners,fill=white,minimum width=45] at (3.1,0) {$\beta^{(k)}$};
		\end{tikzpicture}}$. 
	
	Now the map, $\raisebox{-14mm}{
		\begin{tikzpicture}
		\draw[orange,->] (-.7,-.3) to (-.7,-2);
		\draw[orange,<-] (-.3,-.35) to (-.3,-2);
		\draw[orange,->] (.3,-1.55) to (.3,-2);
		\draw[orange,<-] (.7,-1.55) to (.7,-2);
		\draw[red,in=90,out=90,looseness=2] (0,.3) to (1,.3);
		\node at (.5,.88) {$\red{\bullet}$};
		\draw[red] (.5,.88) to (.5,1.4);
		\draw[red] (1,.3) to (1,-1);
		\draw[violet,->] (-.4,.3) to (-.4,1.4);
		\draw[violet,<-] (.1,-.35) to (.1,-1);
		\draw[violet,<-] (1,-1.55) to (1,-2);
		\draw[densely dotted] (1.6,-2) to (1.6,1.4);
		\node[draw,thick,rounded corners,fill=white,minimum width=35] at (-.3,0) {$\beta^{(k)}$};
		\node[draw,thick,rounded corners,fill=white,minimum width=35] at (.6,-1.2) {$\beta^{(k)}$};
		\end{tikzpicture}} : F_k(x) \us{A_k} \boxtimes H_k \us{H_k} \boxtimes H_k \us{A_k} \boxtimes H_k \us{H_k} \boxtimes H_k \to F_k Q_k(x) $ given as follows:
	\[\raisebox{-14mm}{
		\begin{tikzpicture}
		\draw[orange,->] (-.7,-.3) to (-.7,-2);
		\draw[orange,<-] (-.3,-.35) to (-.3,-2);
		\draw[orange,->] (.3,-1.55) to (.3,-2);
		\draw[orange,<-] (.7,-1.55) to (.7,-2);
		\draw[red,in=90,out=90,looseness=2] (0,.3) to (1,.3);
		\node at (.5,.88) {$\red{\bullet}$};
		\draw[red] (.5,.88) to (.5,1.4);
		\draw[red] (1,.3) to (1,-1);
		\draw[violet,->] (-.4,.3) to (-.4,1.4);
		\draw[violet,<-] (.1,-.35) to (.1,-1);
		\draw[violet,<-] (1,-1.55) to (1,-2);
		\draw[densely dotted] (1.6,-2) to (1.6,1.4);
		\node[draw,thick,rounded corners,fill=white,minimum width=35] at (-.3,0) {$\beta^{(k)}$};
		\node[draw,thick,rounded corners,fill=white,minimum width=35] at (.6,-1.2) {$\beta^{(k)}$};
		\end{tikzpicture}} \Bigg( u \boxtimes \xi_1 \boxtimes \xi_2 \boxtimes \xi_3 \boxtimes \xi_4 \Bigg) \  = \  \raisebox{-12mm}{\begin{tikzpicture}
		\draw (-.3,0) to (-.3,-3);
		\draw (.3,0) to (.3,-3);
		\draw[dashed] (.5,0) to (.5,-3);
		\draw[densely dotted] (.1,0) to (.1,.8);
		\draw[red] (-1,-.8) to (-1,-2);
		\draw[red] (-.75,-.5) to (-.75,.8);
		\draw[red,in=90,out=90,looseness=2] (-.5,-.8) to (-1,-.8);
		\node at (0,-.5) {$\cdots$};
		\node at (0,-1.7) {$\cdots$};
		\node at (0,-2.8) {$\cdots$};
		\node at (-.75,-.5) {$\red{\bullet}$};
		\node[draw,thick,rounded corners,fill=white,minimum width=35] at (.1,0) {$u$};
		\node[draw,thick,rounded corners,fill=white,minimum width=45] at (0,-1.1) {$\xi_1 \cdot \xi_2$};
		\node[draw,thick,rounded corners,fill=white,minimum width=50] at (-.2,-2.2) {$\xi_3 \cdot \xi_4$};
		\end{tikzpicture}} \ = \ \raisebox{-10mm}{\begin{tikzpicture}
		\draw (-.3,0) to (-.3,-2);
		\draw (.3,0) to (.3,-2);
		\draw[dashed] (.5,0) to (.5,-2);
		\draw[densely dotted] (.1,0) to (.1,.8);
		\draw[red] (-.7,-1.1) to (-.7,.8);
		\node at (0,-.5) {$\cdots$};
		\node at (0,-1.7) {$\cdots$};
		\node[draw,thick,rounded corners,fill=white,minimum width=35] at (.1,0) {$u$};
		\node[draw,thick,rounded corners,fill=white,minimum width=45] at (0,-1.1) {$(\xi_1 \cdot \xi_2)\cdot (\xi_3 \cdot \xi_4)$};
		\end{tikzpicture}}\] for every $u \in F_k(x)$ and $ \xi_1,\xi_2,\xi_3,\xi_4 \in H_k \ $. 
	It is straightforward to show that \[\raisebox{-10mm}{
		\begin{tikzpicture}
		\draw[densely dotted] (4.2,-1) to (4.2,1.2);
		\draw[orange,in=90,out=90,looseness=2,->] (2.7,-1) to (3.2,-1);
		\draw[orange,->] (2.5,-.3) to (2.5,-1);
		\draw[orange,<-] (3.4,-.35) to (3.4,-1);
		\draw[red] (3,.3) to (3,1.2);
		\draw[violet,->] (2.6,.3) to (2.6,1.2);
		\draw[violet,<-] (3.7,-.35) to (3.7,-1);
		\node[draw,thick,rounded corners,fill=white,minimum width=45] at (3.1,0) {$\beta^{(k)}$};
		\end{tikzpicture}} \Bigg( u \boxtimes \xi_1 \boxtimes \xi_2 \boxtimes \xi_3 \boxtimes \xi_4 \Bigg) \ = \ \raisebox{-10mm}{\begin{tikzpicture}
		\draw (-.3,0) to (-.3,-2);
		\draw (.3,0) to (.3,-2);
		\draw[dashed] (.5,0) to (.5,-2);
		\draw[densely dotted] (.1,0) to (.1,.8);
		\draw[red] (-.7,-1.1) to (-.7,.8);
		\node at (0,-.5) {$\cdots$};
		\node at (0,-1.7) {$\cdots$};
		\node[draw,thick,rounded corners,fill=white,minimum width=35] at (.1,0) {$u$};
		\node[draw,thick,rounded corners,fill=white,minimum width=45] at (0,-1.1) {$(\xi_1 \cdot (\xi_2 \cdot \xi_3)) \cdot \xi_4$};
		\end{tikzpicture}} \ = \ \raisebox{-10mm}{\begin{tikzpicture}
		\draw (-.3,0) to (-.3,-2);
		\draw (.3,0) to (.3,-2);
		\draw[dashed] (.5,0) to (.5,-2);
		\draw[densely dotted] (.1,0) to (.1,.8);
		\draw[red] (-.7,-1.1) to (-.7,.8);
		\node at (0,-.5) {$\cdots$};
		\node at (0,-1.7) {$\cdots$};
		\node[draw,thick,rounded corners,fill=white,minimum width=35] at (.1,0) {$u$};
		\node[draw,thick,rounded corners,fill=white,minimum width=45] at (0,-1.1) {$(\xi_1 \cdot \xi_2)\cdot (\xi_3 \cdot \xi_4)$};
		\end{tikzpicture}} \] for every $u \in F_k(x)$ and $ \xi_1,\xi_2,\xi_3,\xi_4 \in H_k \ $. The last equality follows because of associativity of $H_k$ as shown in \Cref{HkCstar}. Thus, $\gamma^{(k)}$ intertwines the multiplication maps for each $k \geq l$.

	Also, it is easy to see that $\raisebox{-10mm}{
		\begin{tikzpicture}
		\draw[red] (0,.3) to (0,1);
		\draw[blue,->] (-.2,-.3) to (-.2,-1);
		\draw[blue,<-] (.2,-.3) to (.2,-1);
		\draw[blue,in=-90,out=-90,looseness=2] (-.2,-1) to (.2,-1);
		\node[draw,thick,rounded corners,fill=white] at (0,0) {$\gamma^{(k)}$};
		\end{tikzpicture}} \ = \ 
	\raisebox{-6mm}{
		\begin{tikzpicture}
		\draw[red] (0,-.4) to (0,1);
		\node at (0,-.4) {$\red{\bullet}$};
		\end{tikzpicture}}$  if and only if \, $\raisebox{-8mm}{\begin{tikzpicture}
		\draw[red] (0.2,.3) to (0.2,1);
		\draw[orange,in=-90,out=-90,looseness=2] (0,-1) to (-.4,-1);
		\draw[violet,<-] (.4,-.35) to (.4,-1);
		\draw[orange,<-] (0,-.35) to (0,-1);
		\draw[orange,->] (-.4,-.35) to (-.4,-1);
		\draw[violet,->] (0,.3) to (0,1);
		\node[draw,thick,rounded corners,fill=white,minimum width=35] at (0,0) {$\beta^{(k)}$};
		\end{tikzpicture}} \ = \ \raisebox{-8mm}{\begin{tikzpicture}
		\draw[violet,->] (0,-1) to (0,1);
		\draw[red] (.5,-1) to (.5,1);
		\node at (.5,-1) {$\red{\bullet}$};
		\end{tikzpicture}} \ $. Now, the map $\raisebox{-8mm}{\begin{tikzpicture}
		\draw[red] (0.2,.3) to (0.2,1);
		\draw[orange,in=-90,out=-90,looseness=2] (0,-1) to (-.4,-1);
		\draw[violet,<-] (.4,-.35) to (.4,-1);
		\draw[orange,<-] (0,-.35) to (0,-1);
		\draw[orange,->] (-.4,-.35) to (-.4,-1);
		\draw[violet,->] (0,.3) to (0,1);
		\draw[densely dotted] (.8,-1) to (.8,1);
		\node[draw,thick,rounded corners,fill=white,minimum width=35] at (0,0) {$\beta^{(k)}$};
		\end{tikzpicture}} : F_k(x) \to F_k Q_k (x)$ is given as follows :
	\[\raisebox{-10mm}{\begin{tikzpicture}
		\draw[red] (0.2,.3) to (0.2,1);
		\draw[orange,in=-90,out=-90,looseness=2] (0,-1) to (-.4,-1);
		\draw[violet,<-] (.4,-.35) to (.4,-1);
		\draw[orange,<-] (0,-.35) to (0,-1);
		\draw[orange,->] (-.4,-.35) to (-.4,-1);
		\draw[violet,->] (0,.3) to (0,1);
		\draw[densely dotted] (.8,-1) to (.8,1);
		\node[draw,thick,rounded corners,fill=white,minimum width=35] at (0,0) {$\beta^{(k)}$};
		\end{tikzpicture}} \Big(u \Big) = \raisebox{-8mm}{\begin{tikzpicture}
		\draw[red] (0.2,.3) to (0.2,1);
		\draw[violet,<-] (.4,-.35) to (.4,-1);
		\draw[orange,<-] (0,-.35) to (0,-1);
		\draw[orange,->] (-.4,-.35) to (-.4,-1);
		\draw[violet,->] (0,.3) to (0,1);
		\draw[densely dotted] (.8,-1) to (.8,1);
		\node[draw,thick,rounded corners,fill=white,minimum width=35] at (0,0) {$\beta^{(k)}$};
		\end{tikzpicture}}\Big( \underset{ \sigma \in \mscr S_k}{\sum} u \boxtimes \sigma \boxtimes \sigma^{\dagger} \Big) = \raisebox{-6mm}{\begin{tikzpicture}
		\draw[densely dotted] (0,.2) to (0,1);
		\draw (-.3,0) to (-.3,-1);
		\draw (.3,0) to (.3,-1);
		\draw[red] (-.6,1) to (-.6,-1);
		\node at (-.6,-1) {$\red{\bullet}$};
		\node at (0,-.5) {$\cdots$};
		\node[draw,thick,rounded corners,fill=white,minimum width=30] at (0,0) {$u$};
		\end{tikzpicture}} \ = \ \raisebox{-8mm}{\begin{tikzpicture}
		\draw[violet,->] (0,-1) to (0,1);
		\draw[red] (.5,-1) to (.5,1);
		\draw[densely dotted] (.7,-1) to (.7,1);
		\node at (.5,-1) {$\red{\bullet}$};
		\end{tikzpicture}} \Big(u \Big) \ \ \ \ \t{for every} \ \ u \in F_k(x) \ \  \] where $\mscr S_k \subset H_k$ is as given in \Cref{Hkbasis} . Thus, $\gamma^{(k)}$ intertwines the unit maps for each $k \geq l \ $. This concludes the proposition .
\end{proof}

\subsection{Exchange relation of $\gamma^{(k)}$'s}\

To achieve isomorphism in \textbf{UC}, we still have to show that $\gamma^{(k)}$'s satisfy exchange relation for $ k \geq l $. This will establish `splitting' of $\left(Q_\bullet, m_\bullet, i_\bullet \right) \in \textbf{UC}_1(\Gamma_\bullet , \Gamma_\bullet)$ by $\left(X_\bullet, W_\bullet \right) \in \textbf{UC}_1(\Gamma_\bullet, \Delta_\bullet) $.

\begin{rem}\label{xrelgammak}
	In order to show that $\gamma^{(k)}$'s will  satisfy exchange relation for $ k \geq l $, it is enough to show that $\beta^{(k)}$'s also does so because solutions to conjugate equations for $F_k$'s and $\ol F_k$'s satisfy exchange relations for each $k \geq l$. So instead of showing exchange relation of $\gamma^{(k)}$'s we will show that $\beta^{(k)}$'s satisfy exchange relation for $k \geq l \ $.
\end{rem}

We now proceed to show that $\beta^{(k)}$'s satisfy exchange relation for $k \geq l \ $. 

\begin{prop}\label{xrelbetak}
	For $k \geq l$, $\beta^{(k)}$'s satisfy exchange relation.
\end{prop}

\begin{proof}
For $x \in \t{Ob}(\mcal M_k)$ the map, \[\raisebox{-12mm}{\begin{tikzpicture}
	\draw[red] (0.2,.3) to (0.2,1);
	\draw[in=-90,out=90,looseness=2] (-.9,-1.2) to (.8,-.2);
	\draw[white,line width=1mm] (.4,-.35) to (.4,-1.2);
	\draw[white,line width=1mm] (0,-.35) to (0,-1.2);
	\draw[white,line width=1mm] (-.4,-.35) to (-.4,-1.2);
	\draw[violet,<-] (.4,-.35) to (.4,-1.2);
	\draw[orange,<-] (0,-.35) to (0,-1.2);
	\draw[orange,->] (-.4,-.35) to (-.4,-1.2);
	\draw[violet,->] (-.4,.35) to (-.4,1);
	\draw (.8,-.2) to (.8,1);
	\node[draw,thick,rounded corners,fill=white,minimum width=35] at (0,0) {$\beta^{(k+1)}$};
	\node[scale=.8] at (.3,1.2) {$Q_{k+1}$};
	\node[scale=.8] at (.5,-1.4) {$F_k$};
	\node[scale=.8] at (.1,-1.4) {$\Lambda_k$};
	\node[scale=.8] at (-.4,-1.5) {$\ol \Lambda_k$};
	\draw[densely dotted] (1.2,-1.2) to (1.2,1);
	\node[scale=.8] at (-1.1,-1.3) {$\Sigma_{k+1}$};
	\end{tikzpicture}} \ = \ \raisebox{-16mm}{
	\begin{tikzpicture}
	\draw[white,line width=1mm,out=-90,in=90] (.25,.5) to (-.25,-.5);
	\draw[out=-90,in=90,orange,<-] (.25,.5) to (-.25,-.5);
	\draw[in=-90,out=90,looseness=2] (-.25,.5) to (2,1.8);
	\begin{scope}[on background layer]
	\draw[out=90,in=-90] (.25,-.5) to (-.25,.5);
	\draw (.25,-.5) to (.25,-.8);
	\end{scope}
	\node[right,scale=0.7] at (-.1,-1) {$\Sigma_{k+1}$};
	\node[scale=.7] at (2.4,-.5) {$x$}; 
	\draw[orange,in=90,out=90,looseness=2,->] (.25,.5) to (.75,.5);
	\draw[orange,->] (.75,.5) to (.75,-.8);
	\draw[orange,in=-90,out=-90,looseness=2,<-] (-.25,-.5) to (-.75,-.5);
	\draw[orange,<-] (-.75,-.5) to (-.75,1.6);
	\draw[white,line width=1mm] (1,-.8) to (1,1.6);
	\draw[white,line width=1mm] (1.4,-.8) to (1.4,1.6);
	\draw[orange,->] (1,-.8) to (1,1.6);
	\draw[violet,->] (1.4,-.8) to (1.4,1.6);
	\draw[densely dotted] (2.2,2.8) to (2.2,-.8);
	\draw[violet,->] (.3,2.2) to (.3,2.8);
	\draw[red] (.8,2.2) to (.8,2.8);
	\draw (2,1.8) to (2,2.8);
	\node[draw,thick,rounded corners,fill=white,minimum width=72] at (.4,1.95) {$\beta^{(k+1)}$};
	\end{tikzpicture}} : F_k(x) \us{A_k} \boxtimes H_k \us{H_k} \boxtimes H_k \us{A_k} \boxtimes A_{k+1} \to F_{k+1} Q_{k+1} \Gamma_{k+1}(x) \] 	is given as follows:

\[\raisebox{-18mm}{
	\begin{tikzpicture}
	\draw[white,line width=1mm,out=-90,in=90] (.25,.5) to (-.25,-.5);
	\draw[out=-90,in=90,orange,<-] (.25,.5) to (-.25,-.5);
	\draw[in=-90,out=90,looseness=2] (-.25,.5) to (2,1.8);
	\begin{scope}[on background layer]
	\draw[out=90,in=-90] (.25,-.5) to (-.25,.5);
	\draw (.25,-.5) to (.25,-.8);
	\end{scope}
	\node[right,scale=0.7] at (-.1,-1) {$\Sigma_{k+1}$};
	\node[scale=.7] at (2.4,-.5) {$x$}; 
	\draw[orange,in=90,out=90,looseness=2,->] (.25,.5) to (.75,.5);
	\draw[orange,->] (.75,.5) to (.75,-.8);
	\draw[orange,in=-90,out=-90,looseness=2,<-] (-.25,-.5) to (-.75,-.5);
	\draw[orange,<-] (-.75,-.5) to (-.75,1.6);
	\draw[white,line width=1mm] (1,-.8) to (1,1.6);
	\draw[white,line width=1mm] (1.4,-.8) to (1.4,1.6);
	\draw[orange,->] (1,-.8) to (1,1.6);
	\draw[violet,->] (1.4,-.8) to (1.4,1.6);
	\draw[densely dotted] (2.2,2.8) to (2.2,-.8);
	\draw[violet,->] (.3,2.2) to (.3,2.8);
	\draw[red] (.8,2.2) to (.8,2.8);
	\draw (2,1.8) to (2,2.8);
	\node[draw,thick,rounded corners,fill=white,minimum width=72] at (.4,1.95) {$\beta^{(k+1)}$};
	\end{tikzpicture}} \Big(u \us{A_k}\boxtimes \xi_1 \us{H_k}\boxtimes \xi_2 \us{A_k}\boxtimes \alpha \Big) = \raisebox{-18mm}{
	\begin{tikzpicture}
	\draw[in=-90,out=90,looseness=2] (-.25,.5) to (2,1.8);
	\draw (-.25,.5) to (-.25,-.8);
	\begin{scope}[on background layer]
	\end{scope}
	\node[right,scale=0.7] at (-.2,-1) {$\Delta_{k+1}$};
	\node[scale=.7] at (2.4,-.5) {$x$}; 
	\draw[orange,<-] (-.75,-.8) to (-.75,1.6);
	\draw[white,line width=1mm] (1,-.8) to (1,1.6);
	\draw[white,line width=1mm] (1.4,-.8) to (1.4,1.6);
	\draw[orange,->] (1,-.8) to (1,1.6);
	\draw[violet,->] (1.4,-.8) to (1.4,1.6);
	\draw[densely dotted] (2.2,2.8) to (2.2,-.8);
	\draw[violet,->] (.3,2.2) to (.3,2.8);
	\draw[red] (.8,2.2) to (.8,2.8);
	\draw (2,1.8) to (2,2.8);
	\node[draw,thick,rounded corners,fill=white,minimum width=72] at (.4,1.95) {$\beta^{(k+1)}$};
	\end{tikzpicture}} \Big(u \us{A_k}\boxtimes \xi_1 \cdot \xi_2 \us{H_k}\boxtimes \alpha \us{H_{k+1}}\boxtimes 1_{H_{k+1}} \Big)  \]
\[\raisebox{-18mm}{
	\begin{tikzpicture}
	\draw[in=-90,out=90,looseness=2] (1,.5) to (2,1.8);
	\draw (1,.5) to (1,-.8);
	\begin{scope}[on background layer]
	\end{scope}
	\node[right,scale=0.7] at (.6,-1) {$\Sigma_{k+1}$};
	\node[scale=.7] at (2.4,-.5) {$x$}; 
	\draw[orange,<-] (-.75,-.8) to (-.75,1.6);
	\draw[white,line width=1mm] (1.4,-.8) to (1.4,1.6);
	\draw[orange,->] (.4,-.8) to (.4,1.6);
	\draw[violet,->] (1.4,-.8) to (1.4,1.6);
	\draw[densely dotted] (2.2,2.8) to (2.2,-.8);
	\draw[violet,->] (.3,2.2) to (.3,2.8);
	\draw[red] (.8,2.2) to (.8,2.8);
	\draw (2,1.8) to (2,2.8);
	\node[draw,thick,rounded corners,fill=white,minimum width=72] at (.4,1.95) {$\beta^{(k+1)}$};
	\end{tikzpicture}} \Big(u \us{A_k}\boxtimes 1_{A_{k+1}} \us{H_k}\boxtimes (\xi_1 \cdot \xi_2)\alpha \us{H_{k+1}}\boxtimes 1_{H_{k+1}} \Big) = \raisebox{-10mm}{\begin{tikzpicture}
\draw[red] (-.8,0) to (-.8,1.8);
\draw (-.4,-1) to (-.4,1.8);
\draw (0,-1) to (0,1);
\draw (.6,-1) to (.6,1);
\draw[dashed] (.8,-1) to (.8,1);
\draw[densely dotted] (.4,1.4) to (.4,1.8);
\node[draw,thick,rounded corners,fill=white,minimum width=60] at (0,0) {$(\xi_1 \cdot \xi_2)\alpha$};
\node[draw,thick,rounded corners,fill=white,minimum width=30] at (.4,1.2) {$u$};
\node at (.3,.7) {$\cdots$};
\node at (.3,-.7) {$\cdots$};
\node[scale=.8] at (-.6,-1.2) {$\Gamma_{k+1}$};
\node[scale=.8] at (.1,-1.2) {$\Gamma_{k}$}; 
\node[scale=.8] at (.7,-1.2) {$\Gamma_1$};
\node[scale=.8] at (-1.2,1.2) {$Q_{k+1}$};
\node at (1.2,-.8) {$m_0$};
\node at (.6,1.7) {$x$};
\node[scale=.8] at (-.5,2) {$\Gamma_{k+1}$};
\end{tikzpicture}} = \raisebox{-10mm}{\begin{tikzpicture}
\draw (-.7,.3) to (-.7,-2);
\draw[in=-90,out=90,looseness=2] (-.7,.3) to (-.2,1.8);
\draw[white,line width=1mm,in=-90,out=90,looseness=2] (-.2,.3) to (-.7,1.8);
\draw[red,in=-90,out=90,looseness=2] (-.2,.3) to (-.7,1.8);
\draw (0,-2) to (0,1);
\draw (.6,-2) to (.6,1);
\draw[dashed] (.8,-2) to (.8,1);
\draw[densely dotted] (.4,1.4) to (.4,1.8);
\node[draw,thick,rounded corners,fill=white,minimum width=45] at (.4,0) {$\xi_1 \cdot \xi_2$};
\node[draw,thick,rounded corners,fill=white,minimum width=30] at (.4,1.2) {$u$};
\node[draw,thick,rounded corners,fill=white,minimum width=55] at (0,-1.1) {$\alpha$};
\node at (.3,.7) {$\cdots$};
\node at (.3,-.7) {$\cdots$};
\node at (.3,-1.7) {$\cdots$};
\node[scale=.8] at (-.6,-2.2) {$\Gamma_{k+1}$};
\node[scale=.8] at (.1,-2.2) {$\Gamma_{k}$}; 
\node[scale=.8] at (.7,-2.2) {$\Gamma_1$};
\node[scale=.8] at (-1.1,1.4) {$Q_{k+1}$};
\node at (1.2,-1.8) {$m_0$};
\node at (.6,1.7) {$x$};
\node[scale=.8] at (-.3,2) {$\Gamma_{k+1}$};
\end{tikzpicture}}  \] $\t{for every} \ \  u \in F_k(x), \xi_1, \xi_2 \in H_k, \alpha \in A_{k+1} \ $. Also, it  will easily follow from the definition of $\beta^{(k)}$'s that for every $u \in F_k(x),\,\xi_1,\, \xi_2 \in H_k,\,\alpha \in A_{k+1} \ $ we have,
\[\raisebox{-12mm}{\begin{tikzpicture}
	\draw[in=-90,out=90,looseness=2] (-.9,.4) to (.6,1);
	\draw[white,line width=1mm] (0.2,.3) to (0.2,1.5);
	\draw[white,line width=1mm] (-.4,.35) to (-.4,1.5);
	\draw[violet,<-] (.4,-.35) to (.4,-1.2);
	\draw[orange,<-] (0,-.35) to (0,-1.2);
	\draw[orange,->] (-.4,-.35) to (-.4,-1.2);
	\draw[violet,->] (-.4,.35) to (-.4,1.5);
	\draw (-.9,.4) to (-.9,-1.2);
	\draw[red] (0.2,.3) to (0.2,1.5);
	\draw (.6,1) to (.6,1.5);
	\node[draw,thick,rounded corners,fill=white,minimum width=35] at (0,0) {$\beta^{(k)}$};
	\node[scale=.8] at (0,1.7) {$Q_{k+1}$};
	\node[scale=.8] at (-.8,1.7) {$F_{k+1}$};
	\node[scale=.8] at (.5,-1.4) {$F_k$};
	\node[scale=.8] at (.1,-1.4) {$\Lambda_k$};
	\node[scale=.8] at (-.4,-1.5) {$\ol \Lambda_k$};
	\draw[densely dotted] (1.2,-1.2) to (1.2,1.5);
	\node[scale=.8] at (-1.1,-1.3) {$\Sigma_{k+1}$};
	\node[scale=.8] at (.8,1.7) {$\Gamma_{k+1}$};
	\end{tikzpicture}} \Big(u \us{A_k}\boxtimes \xi_1 \us{H_k}\boxtimes \xi_2 \us{A_k}\boxtimes \alpha \Big) = \raisebox{-12mm}{\begin{tikzpicture}
	\draw (-.7,.3) to (-.7,-2);
	\draw[in=-90,out=90,looseness=2] (-.7,.3) to (-.2,1.8);
	\draw[white,line width=1mm,in=-90,out=90,looseness=2] (-.2,.3) to (-.7,1.8);
	\draw[red,in=-90,out=90,looseness=2] (-.2,.3) to (-.7,1.8);
	\draw (0,-2) to (0,1);
	\draw (.6,-2) to (.6,1);
	\draw[dashed] (.8,-2) to (.8,1);
	\draw[densely dotted] (.4,1.4) to (.4,1.8);
	\node[draw,thick,rounded corners,fill=white,minimum width=45] at (.4,0) {$\xi_1 \cdot \xi_2$};
	\node[draw,thick,rounded corners,fill=white,minimum width=30] at (.4,1.2) {$u$};
	\node[draw,thick,rounded corners,fill=white,minimum width=55] at (0,-1.1) {$\alpha$};
	\node at (.3,.7) {$\cdots$};
	\node at (.3,-.7) {$\cdots$};
	\node at (.3,-1.7) {$\cdots$};
	\node[scale=.8] at (-.6,-2.2) {$\Gamma_{k+1}$};
	\node[scale=.8] at (.1,-2.2) {$\Gamma_{k}$}; 
	\node[scale=.8] at (.7,-2.2) {$\Gamma_1$};
	\node[scale=.8] at (-1.1,1.4) {$Q_{k+1}$};
	\node at (1.2,-1.8) {$m_0$};
	\node at (.6,1.7) {$x$};
	\node[scale=.8] at (-.3,2) {$\Gamma_{k+1}$};
	\end{tikzpicture}} \]
	
	Thus, $\beta^{(k)}$'s satisfy exchange relation for $k \geq l \ $.
\end{proof}

From \Cref{UCisomorphism}, \Cref{Qsysiso}, \Cref{xrelgammak} and \Cref{xrelbetak} we get the following theorem.
\begin{thm}
	$  (Q_\bullet,{W^Q}_\bullet)$ is isomorphic to $\left(\left\{\ol X_k X_k\right\}_{k \geq 0} ,  
	\left\{\raisebox{-12mm}{
		\begin{tikzpicture}[scale=1.5]
		\draw[in=-90,out=90] (0,0) to (.5,1);
		\draw[white,line width=1mm,in=-90,out=90] (.35,0) to (-.25,1);
		\draw[white,line width=1mm,in=-90,out=90] (.7,0) to (0.2,1); 
		\draw[blue,in=-90,out=90][<-] (.35,0) to (-.25,1);
		\draw[blue,in=-90,out=90][->] (.7,0) to (0.2,1);
		\node[scale=.8] at (.3,-.2) {$\ol X_k$};
		\node[scale=.8] at (.75,-.2) {$X_k$};
		\node[scale=.8] at (-.6,1) {$\ol X_{k+1}$};
		\node[scale=.8] at (.1,1.2) {$X_{k+1}$}; 
		\node[left,scale=0.7] at (0,.1) {$\Gamma_{k+1}$};
		\node[right,scale=0.7] at (.5,1) {$\Gamma_{k+1}$};
		\end{tikzpicture}}\right\}_{k \geq 1} \right)$ as $Q$-systems in ${\normalfont \textbf{UC}} \ $.
\end{thm}

\Contact



\begin{thebibliography}{1}

 
 
 
 
 
 
 \vspace{2mm}
 \bibitem[AMP15] {AMP}\ N.Afzaly, S.Morrison and D.Penneys 2015. The classification of subfactors with index at most $5 \frac{1}{4}$. 
 https://doi.org/10.48550/arXiv.1509.00038  to appear Mem. Amer. Math. Soc. \vspace{1mm}
 
 \bibitem[B97]{B97} Dietmar Bisch 1997. Bimodules, higher relative commutants and the fusion algebra associated to a subfactor, Operator algebras and their applications (Waterloo, ON, 1994/1995), 13-63, Fields Inst. Commun., 13, Amer. Math. Soc., Providence, RI, 1997.\vspace{1mm}
 
 
 
 
 \bibitem[CPJ22] {CPJ}\ Q.Chen, R.H. Palomares and C.Jones 2022. K-theoretic classification of inductive limit actions of fusion categories on AF-algebras. https://doi.org/10.48550/arXiv.2207.11854
 \vspace{1mm}	
 
 
 
 \bibitem[CPJP22] {CPJP}\ Q.Chen, R.H. Palomares, C.Jones and D.Penneys 2022.  Q-System completion of C*-2-categories. \textit{Journal of Functional Analysis}, Volume 283, Issue 3, 2022, 109524, ISSN 0022-1236, https://doi.org/10.1016/j.jfa.2022.109524. (https://www.sciencedirect.com/science/article/pii/S0022123622001446)
 \vspace{1mm}
 
 \bibitem[DGGJ22]{DGGJ}\ P.Das, M.Ghosh, S.Ghosh and C.Jones 2022. Unitary connections on Bratteli diagrams. 	arXiv:2211.03822.
 \vspace*{1mm}
 
 \bibitem[DR18]{DR18}\ C.L. Douglas and D.Reutter 2018. Fusion 2-categories and a state-sum invariant for 4-manifolds. arXiv:1812.11933.
 \vspace*{1mm}
 
 
 \bibitem[EK98] {EvKaw}\ D. Evans and Y. Kawahigashi 1998. Quantum symmetries and operator algebras. Oxford Mathematical Monographs, Oxford University Press.\vspace{1mm}
 
 \bibitem[GJF19] {GJF19}\ D.Gaiotto and T.Johnson-Freyd 2019. Condensations in higher categories. arXiv:1905.09566.
 \vspace*{1mm}
 
 
 \bibitem[GLR85]{GLR85}\ P. Ghez, R. Lima and J. E. Roberts. W*-categories 1985. \textit{Pacific Journal of Mathematics}, 120(1): 79-109 1985. https://doi.org/pjm/1102703884.
 \vspace*{1mm}
 
 
 
 
 
 
 
 
 
 
 \bibitem[J99] {J99}\ V.F.R. Jones. Planar Algebras I. arxiv:math.QA/9909027.\vspace{1mm}
 
 \bibitem[J83] {J83}\ V.F.R. Jones 1983. Index for Subfactors. \textit{Invent. Math}. 73, pp. 1-25.  \vspace{1mm}
 
 \bibitem[JMS14]{JMS}\ V.Jones, S.Morrison and N.Snyder.  The classification of subfactors of index at most 5. \textit{Bull. Amer. Math. Soc.} (N.S.), 51(2):277–327, 2014. arXiv:1304.6141, doi :10.1090/S0273-0979-2013-01442-3. \vspace{1mm}
 
 \bibitem[JP19] {JP19}\ C.Jones and D.Penneys. Realizations of algebra objects and discrete subfactors. \textit{Adv.
 	Math.}, 350:588–661, 2019. MR3948170 DOI:10.1016/j.aim.2019.04.039 . arXiv:1704.02035. \vspace*{1mm}
 
 \bibitem[JP20] {JP20}\ C.Jones and D.Penneys. Q-systems and compact W*-algebra objects. Topological phases of matter and quantum computation, volume 747 of Contemp. Math., pages 63–88. Amer. Math. Soc., Providence, RI, 2020. MR4079745 DOI:10.1090/conm/747/15039. arXiv:1707.02155. \vspace*{1mm}
 
 \bibitem[JY21] {JY21}\  N.Johnson and D.Yau. $2$-dimensional categories. Oxford University Press. https://doi.org/10.1093/oso/9780198871378.001.0001. arXiv:2002.06055.
 \vspace*{1mm}
  
 
 
 \bibitem[L94]{Lon}\ R.Longo 1994. A duality for Hopf algebras and for subfactors. I. \textit{Communications in Mathematical
 	Physics}, 159(1) 133-150 1994. https://doi.org/cmp/1104254494 . \vspace{1mm}
 
 \bibitem[LR97]{LR} R. Longo and J.E. Roberts  1997. A theory of dimension. \textit{K-theory} 11(2): pp. 103-159, 1997.
 \vspace*{1mm}
 
 
 \bibitem[M03] {M03}\ M. M{\"u}ger 2003. From subfactors to categories and topology I: Frobenius algebras and Morita equivalence of tensor categories. \textit{Journal of Pure and Applied Algebra}. 180.1, pp. 81-157. \vspace{1mm}
 
 
 
 
 
 
 
 
 \bibitem[O88] {O88}\ A. Ocneanu 1988.  Quantized groups, string algebras and Galois theory for algebras. Operator algebras and applications, London Math. Soc. Lecture Note Ser., 136: pp. 119-172. \vspace{1mm} 
 
 
 \bibitem[P89] {P89} S. Popa 1989. Relative dimension, towers of projections and commuting squares of subfactors. \textit{Pacific Journal of Mathematics}. Vol. 137 (1989), No. 1, 181–207. \vspace{1mm}
 
 \bibitem[P94] {P94}\ S. Popa 1994. Classification of amenable Subfactors of type II. \textit{Acta. Math}. 172, pp. 163-225. \vspace{1mm}
 
 \bibitem[P95] {P95}\ S. Popa 1995.  An axiomatization of the lattice of higher relative commutants. \textit{Invent. Math}. 120, pp. 237-252.
 \vspace*{1mm}
 
 
 
 
 
 
 
 
 \bibitem[Z07]{Z07}\ Pasquale A. Zito 2007. 2-C*-categories with non-simple units. \textit{Adv. Math.}, 210(1):122–164, 2007. doi :10.1016/j.aim.2006.05.017. arXiv:math/0509266.
 
 
  


\end{thebibliography}
\end{document}